\documentclass[12pt,twoside]{amsart}
\usepackage[latin1]{inputenc}
\usepackage{amsmath, amsthm, amscd, amsfonts, amssymb, graphicx}
\usepackage[bookmarksnumbered, plainpages]{hyperref}

\newtheorem{thm}{Theorem}[section]

\newtheorem{lem}[thm]{Lemma}

\newtheorem{defn}[thm]{Definition}

\numberwithin{equation}{section}

\begin{document}

\title{\bf A Kastler-Kalau-Walze type theorem for the $J$-twist $D_{J}$ of the Dirac operator}
\author{Siyao Liu \hskip 0.4 true cm  Yong Wang$^{*}$}

\thanks{{\scriptsize
\hskip -0.4 true cm \textit{2010 Mathematics Subject Classification:}
53C40; 53C42.
\newline \textit{Key words and phrases:} Dirac operator; the $J$-twist of the Dirac operator; Kastler-Kalau-Walze type theorems.
\newline \textit{$^{*}$Corresponding author}}}

\maketitle

\begin{abstract}
 \indent In this paper, we give a Lichnerowicz type formula for the $J$-twist $D_{J}$ of the Dirac operator. And we prove a Kastler-Kalau-Walze type theorem for the $J$-twist $D_{J}$ of the Dirac operator on 3-dimensional and 4-dimensional almost product Riemannian spin manifold with boundary.
\end{abstract}

\vskip 0.2 true cm

%------------------------------------------------------------------------------------%

\pagestyle{myheadings}
\markboth{\rightline {\scriptsize Liu}}
         {\leftline{\scriptsize  A Kastler-Kalau-Walze type theorem for the $J$-twist $D_{J}$ of the Dirac operator}}

\bigskip
\bigskip

%------------------------------------------------------------------------------------%
%------------------------------------------------------------------------------------%

\section{ Introduction}
The noncommutative residue was found in \cite{Gu,Wo}. Since noncommutative residues are of great importance to the study of noncommutative geometry, more and more attention has been attached to the study of noncommutative residues.
Connes put forward that the noncommutative residue of the square of the inverse of the Dirac operator was proportioned to the Einstein-Hilbert action, which is called the Kastler-Kalau-Walze theorem now\cite{Co1,Co2}. Kastler gave a brute-force proof of this theorem\cite{Ka}. In the same time, Kalau-Walze and Ackermann proved this theorem by using the normal coordinates system and the heat kernel expansion, respectively\cite{KW,Ac}. The result of Connes was extended to the higher dimensional case \cite{U}. Wang generalized the Connes' results to the case of manifolds with boundary in \cite{Wa1,Wa2}, and proved the Kastler-Kalau-Walze type theorem for the Dirac operator and the signature operator on lower-dimensional manifolds with boundary. In \cite{Wa3}, for the Dirac operator, Wang computed $\widetilde{{\rm Wres}}[\pi^+D^{-1}\circ\pi^+D^{-1}]$, in these cases the boundary term vanished.
Most of the operators studied in the literature have the leading symbol $\sqrt{-1}c(\xi)$\cite{Wa4,Wa5,WW,WWW,Wa}. However, Wu and Wang studied operators with the leading symbol $-\widehat{c}(V)c(\xi)$. In \cite{WW2}, Wu and Wang gave the proof of Kastler-Kalau-Walze type theorems  of the operators $\sqrt{-1}\widehat{c}(V)(d+\delta)$ and $-\sqrt{-1}\widehat{c}(V)(d+\delta)$ on 3,4-dimensional oriented compact manifolds with or without boundary.

On the other hand, some preliminaries and lemmas about the Dirac operator $D$ and the $J$-twist are given in \cite{K}. In \cite{Chen1,Chen2}, the author got estimates on the higher eigenvalues of the Dirac operator on locally reducible Riemannian manifolds by the $J$-twist of the Dirac operator. It can be obtained by simple calculations that the leading symbol of the $J$-twist $D_{J}$ of the Dirac operator is not $\sqrt{-1}c(\xi)$.

The motivation of this paper is to prove a Kastler-Kalau-Walze type theorem for the $J$-twist $D_{J}$ of the Dirac operator. Unlike the case of the Dirac operator in \cite{Wa3}, we find the boundary term is not zero for the $J$-twist $D_{J}$ of the Dirac operator.

A brief description of the organization of this paper is as follows.
In Section 2, this paper will firstly introduce the basic notions of the almost product Riemannian manifold and the $J$-twist $D_{J}$ of the Dirac operator. We also give a Lichnerowicz type formula for the $J$-twist $D_{J}$ of the Dirac operator and a Kastler-Kalau-Walze type theorem for the $J$-twist $D_{J}$ of the Dirac operator on $n$-dimensional almost product Riemannian spin manifold without boundary.
In the next section, we calculate $\widetilde{{\rm Wres}}[\pi^+{{D}_{J}}^{-1}\circ\pi^+{{D}_{J}}^{-1}]$ for $4$-dimensional almost product Riemannian spin manifold with boundary.
In Section 4, we prove the Kastler-Kalau-Walze type theorem on 3-dimensional almost product Riemannian spin manifold with boundary for the $J$-twist $D_{J}$ of the Dirac operator.
%------------------------------------------------------------------------------------%

\vskip 1 true cm

\section{ The $J$-twist $D_{J}$ of the Dirac operator $D$ and its Lichnerowicz formula}

We give some definitions and basic notions which we will use in this paper.

Let $M$ be a smooth compact oriented Riemannian $n$-dimensional manifolds without boundary and $N$ be a vector bundle on $M$.
We say that $P$ is a differential operator of Laplace type, if it has locally the form
\begin{equation}
P=-(g^{ij}\partial_i\partial_j+A^i\partial_i+B),
\end{equation}
where $\partial_{i}$  is a natural local frame on $TM,$ $(g^{ij})_{1\leq i,j\leq n}$ is the inverse matrix associated to the metric
matrix  $(g_{ij})_{1\leq i,j\leq n}$ on $M,$ $A^{i}$ and $B$ are smooth sections of $\textrm{End}(N)$ on $M$ (endomorphism).
If $P$ satisfies the form (2.1), then there is a unique
connection $\nabla$ on $N$ and a unique endomorphism $E$ such that
 \begin{equation}
P=-[g^{ij}(\nabla_{\partial_{i}}\nabla_{\partial_{j}}- \nabla_{\nabla^{L}_{\partial_{i}}\partial_{j}})+E],
\end{equation}
where $\nabla^{L}$ is the Levi-Civita connection on $M$. Moreover
(with local frames of $T^{*}M$ and $N$), $\nabla_{\partial_{i}}=\partial_{i}+\omega_{i} $
and $E$ are related to $g^{ij}$, $A^{i}$ and $B$ through
 \begin{eqnarray}
&&\omega_{i}=\frac{1}{2}g_{ij}\big(A^{i}+g^{kl}\Gamma_{ kl}^{j} \texttt{id}\big),\\
&&E=B-g^{ij}\big(\partial_{i}(\omega_{j})+\omega_{i}\omega_{j}-\omega_{k}\Gamma_{ ij}^{k} \big),
\end{eqnarray}
where $\Gamma_{ kl}^{j}$ is the  Christoffel coefficient of $\nabla^{L}$.

Let $M$ be a $n$-dimensional ($n\geq 3$) oriented compact Riemannian manifold with a Riemannian metric $g^{M}$.
We recall that the Dirac operator $D$ is locally given as following:
\begin{equation}
D=\sum_{i, j=1}^{n}g^{ij}c(\partial_{i})\nabla_{\partial_{j}}^{S}=\sum_{i=1}^{n}c(e_{i})\nabla_{e_{i}}^{S},
\end{equation}
where $c(e_{i})$ be the Clifford action which satisfies the relation
\begin{align}
&c(e_{i})c(e_{j})+c(e_{j})c(e_{i})=-2g^{M}(e_{i}, e_{j})=-2\delta_i^j,
\end{align}
\begin{align}
&\nabla_{\partial_{j}}^{S}=\partial_{i}+\sigma_{i}
\end{align}
and
\begin{align}
&\sigma_{i}=\frac{1}{4}\sum_{j, k=1}^{n}\langle \nabla_{\partial_{i}}^{L}e_{j}, e_{k}\rangle c(e_{j})c(e_{k}).
\end{align}

Let $J$ be a $(1, 1)$-tensor field on $(M, g^M)$ such that $J^2=\texttt{id},$
\begin{align}
&g^M(J(X), J(Y))=g^M(X, Y),
\end{align}
for all vector fields $X,Y\in \Gamma(TM).$ Here $\texttt{id}$ stands for the identity map. $(M, g^M, J)$ is an almost product Riemannian manifold. We can also define on almost product Riemannian spin manifold the following $J$-twist $D_{J}$ of the Dirac operator $D$ by
\begin{align}
&D_{J}:=\sum_{i=1}^{n}c(e_{i})\nabla^{S}_{J(e_{i})}=\sum_{i=1}^{n}c[J(e_{i})]\nabla^{S}_{e_{i}}.
\end{align}

Let $\xi=\sum_{k}\xi_{j}dx_{j},$  $\nabla^L_{\partial_{i}}\partial_{j}=\sum_{k}\Gamma_{ij}^{k}\partial_{k},$  we denote that
\begin{align}
 \Gamma^{k}=g^{ij}\Gamma_{ij}^{k};\ \sigma^{j}=g^{ij}\sigma_{i};\ \partial^{j}=g^{ij}\partial_{i}.\nonumber
\end{align}

By (2.4), (2.5) in \cite{K}
\begin{align}
{D_{J}}^{2}=-\frac{1}{8}\sum_{i,j,k,l=1}^{n}R(J(e_{i}), J(e_{j}), e_{k}, e_{l})c(e_{i})c(e_{j})c(e_{k})c(e_{l})+\triangle+\sum_{\alpha,\beta=1}^{n}c[J(e_{\alpha})]c[(\nabla^{L}_{e_{\alpha}}J)e_{\beta}]\nabla^{S}_{e_{\beta}}.
\end{align}
By (2.8) in \cite{Wa}, we have
\begin{align}
\triangle=-g^{ij}\partial_{i}\partial_{j}-2\sigma^{j}\partial_{j}+\Gamma^{k}\partial_{k}-g^{ij}[\partial_{i}(\sigma_{j})+\sigma_{i}\sigma_{j}-\Gamma^{k}_{ij}\sigma_{k}]+\frac{1}{4}s,
\end{align}
where $s$ is the scalar curvature. We note that
\begin{align}
{D_{J}}^{2}&=-\frac{1}{8}\sum_{i,j,k,l=1}^{n}R(J(e_{i}), J(e_{j}), e_{k}, e_{l})c(e_{i})c(e_{j})c(e_{k})c(e_{l})-g^{ij}\partial_{i}\partial_{j}-2\sigma^{j}\partial_{j}+\Gamma^{k}\partial_{k}\\
&-g^{ij}[\partial_{i}(\sigma_{j})+\sigma_{i}\sigma_{j}-\Gamma^{k}_{ij}\sigma_{k}]+\frac{1}{4}s+\sum_{\alpha,\beta=1}^{n}c[J(e_{\alpha})]c[(\nabla^{L}_{e_{\alpha}}J)e_{\beta}]\sum_{k=1}^{n}\langle e_{\beta}, dx_{k}\rangle\nabla^{S}_{\partial_{k}}\nonumber\\
&=-\frac{1}{8}\sum_{i,j,k,l=1}^{n}R(J(e_{i}), J(e_{j}), e_{k}, e_{l})c(e_{i})c(e_{j})c(e_{k})c(e_{l})-g^{ij}\partial_{i}\partial_{j}-2\sigma^{j}\partial_{j}+\Gamma^{k}\partial_{k}\nonumber\\
&-g^{ij}[\partial_{i}(\sigma_{j})+\sigma_{i}\sigma_{j}-\Gamma^{k}_{ij}\sigma_{k}]+\frac{1}{4}s+\sum_{\alpha,k=1}^{n}c[J(e_{\alpha})]c[(\nabla^{L}_{e_{\alpha}}J)(dx_{k})^{*}]\nabla^{S}_{\partial_{k}},\nonumber
\end{align}
where $e_i^*=g^{M}(e_i,\cdot)$ and $\langle X, dx_{k}\rangle=g^{M}(X, (dx_{k})^{*}),$ for a vector field $X.$
Then we have
\begin{align}
(\omega_{i})_{{D_{J}}^{2}}&=\sigma_{i}-\frac{1}{2}\sum_{\alpha,j=1}^{n}g_{ij}c[J(e_{\alpha})]c[(\nabla^{L}_{e_{\alpha}}J)(dx_{j})^{*}],
\end{align}
\begin{align}
E_{{D_{J}}^{2}}&=\frac{1}{8}\sum_{i,j,k,l=1}^{n}R(J(e_{i}), J(e_{j}), e_{k}, e_{l})c(e_{i})c(e_{j})c(e_{k})c(e_{l})+\sum_{i,j=1}^{n}g^{ij}[\partial_{i}(\sigma_{j})+\sigma_{i}\sigma_{j}-\Gamma^{k}_{ij}\sigma_{k}]\\
&-\sum_{\alpha,j=1}^{n}c[J(e_{\alpha})]c[(\nabla^{L}_{e_{\alpha}}J)(dx_{j})^{*}]\sigma_{j}-\sum_{i,j=1}^{n}g^{ij}\big[\partial_{i}(\sigma_{j}-\frac{1}{2}\sum_{\nu,l=1}^{n}g_{jl}c[J(e_{\nu})]c[(\nabla^{L}_{e_{\nu}}J)(dx_{l})^{*}])\nonumber\\
&+(\sigma_{i}-\frac{1}{2}\sum_{\alpha,p=1}^{n}g_{ip}c[J(e_{\alpha})]c[(\nabla^{L}_{e_{\alpha}}J)(dx_{p})^{*}])(\sigma_{j}-\frac{1}{2}\sum_{\nu,l=1}^{n}g_{jl}c[J(e_{\nu})]c[(\nabla^{L}_{e_{\nu}}J)(dx_{l})^{*}])\nonumber\\
&-(\sigma_{k}-\frac{1}{2}\sum_{\mu,h=1}^{n}g_{kh}c[J(e_{\mu})]c[(\nabla^{L}_{e_{\mu}}J)(dx_{h})^{*}])\Gamma_{ij}^{k}\big]-\frac{1}{4}s.\nonumber
\end{align}

Since $E$ is globally defined on $M$, taking normal coordinates at $x_0$, we have $\sigma^{i}(x_0)=0$, $\Gamma^k(x_0)=0,$ $g^{ij}(x_0)=\delta^j_i,$ $\partial^{j}(x_0)=e_{j},$ $\partial^{j}[c(\partial_{j})](x_0)=0,$ $\nabla^{L}_{e_{j}}e_{k}(x_0)=0$ and $\nabla^{S}_{Y}(c(X))=c(\nabla^{L}_{Y}X),$ a simple calculation shows that
\begin{align}
&\frac{1}{2}\sum_{\nu,j=1}^{n}\partial^{j}(c[J(e_{\nu})]c[(\nabla^{L}_{e_{\nu}}J)(dx_{j})^{*}])(x_0)=\frac{1}{2}\sum_{\nu,j=1}^{n}\nabla_{e_{j}}^{S}(c[J(e_{\nu})]c[(\nabla^{L}_{e_{\nu}}J)e_{j}])(x_0)\\
&=\frac{1}{2}\sum_{\nu,j=1}^{n}c[\nabla_{e_{j}}^{L}(J)e_{\nu}]c[(\nabla^{L}_{e_{\nu}}J)e_{j}](x_0)+\frac{1}{2}\sum_{\nu,j=1}^{n}c[J(e_{\nu})]c[\nabla^{L}_{e_{j}}(\nabla^{L}_{e_{\nu}}(J)e_{j})](x_0)\nonumber\\
&=\frac{1}{2}\sum_{\nu,j=1}^{n}c[\nabla_{e_{j}}^{L}(J)e_{\nu}]c[(\nabla^{L}_{e_{\nu}}J)e_{j}](x_0)+\frac{1}{2}\sum_{\nu,j=1}^{n}c[J(e_{\nu})]c[(\nabla^{L}_{e_{j}}(\nabla^{L}_{e_{\nu}}(J)))e_{j}-(\nabla^{L}_{\nabla^{L}_{e_{j}}e_{\nu}}(J))e_{j}](x_0),\nonumber
\end{align}
so that
\begin{align}
E_{{D_{J}}^{2}}(x_0)=&\frac{1}{8}\sum_{i,j,k,l=1}^{n}R(J(e_{i}), J(e_{j}), e_{k}, e_{l})c(e_{i})c(e_{j})c(e_{k})c(e_{l})+\frac{1}{2}\sum_{\nu,j=1}^{n}c[\nabla_{e_{j}}^{L}(J)e_{\nu}]c[(\nabla^{L}_{e_{\nu}}J)e_{j}]\\
&+\frac{1}{2}\sum_{\nu,j=1}^{n}c[J(e_{\nu})]c[(\nabla^{L}_{e_{j}}(\nabla^{L}_{e_{\nu}}(J)))e_{j}-(\nabla^{L}_{\nabla^{L}_{e_{j}}e_{\nu}}(J))e_{j}]\nonumber\\
&-\frac{1}{4}\sum_{\alpha,\nu,j=1}^{n}c[J(e_{\alpha})]c[(\nabla^{L}_{e_{\alpha}}J)e_{j}]c[J(e_{\nu})]c[(\nabla^{L}_{e_{\nu}}J)e_{j}]-\frac{1}{4}s.\nonumber
\end{align}
We get the following Lichnerowicz formulas:
\begin{align}
{D_{J}}^{2}
=&-g^{ij}(\nabla_{\partial_{i}}\nabla_{\partial_{j}}-\nabla_{\nabla^{L}_{\partial_{i}}\partial_{j}})-\frac{1}{8}\sum_{i,j,k,l=1}^{n}R(J(e_{i}), J(e_{j}), e_{k}, e_{l})c(e_{i})c(e_{j})c(e_{k})c(e_{l})+\frac{1}{4}s\\
&-\frac{1}{2}\sum_{\nu,j=1}^{n}c[\nabla_{e_{j}}^{L}(J)e_{\nu}]c[(\nabla^{L}_{e_{\nu}}J)e_{j}]-\frac{1}{2}\sum_{\nu,j=1}^{n}c[J(e_{\nu})]c[(\nabla^{L}_{e_{j}}(\nabla^{L}_{e_{\nu}}(J)))e_{j}-(\nabla^{L}_{\nabla^{L}_{e_{j}}e_{\nu}}(J))e_{j}]\nonumber\\
&+\frac{1}{4}\sum_{\alpha,\nu,j=1}^{n}c[J(e_{\alpha})]c[(\nabla^{L}_{e_{\alpha}}J)e_{j}]c[J(e_{\nu})]c[(\nabla^{L}_{e_{\nu}}J)e_{j}].\nonumber
\end{align}

According to the detailed descriptions in \cite{Ac}, we know that the noncommutative residue of a generalized laplacian $\widetilde{\Delta}$ is expressed as
\begin{equation}
(n-2)\Phi_{2}(\widetilde{\Delta})=(4\pi)^{-\frac{n}{2}}\Gamma(\frac{n}{2})\widetilde{res}(\widetilde{\Delta}^{-\frac{n}{2}+1}),
\end{equation}
where $\Phi_{2}(\widetilde{\Delta})$ denotes the integral over the diagonal part of the second
coefficient of the heat kernel expansion of $\widetilde{\Delta}$.
Now let $\widetilde{\Delta}={D_{J}}^{2}$. Since ${D_{J}}^{2}$ is a generalized laplacian, we can suppose ${D_{J}}^{2}=\Delta-E$, then, we have
\begin{align}
{\rm Wres}({D_{J}}^{2})^{-\frac{n-2}{2}}
=\frac{(n-2)(4\pi)^{\frac{n}{2}}}{(\frac{n}{2}-1)!}\int_{M}{\rm tr}(\frac{1}{6}s+E_{{D_{J}}^{2}})d{\rm Vol_{M} },
\end{align}
where ${\rm Wres}$ denote the noncommutative residue, ${\rm tr}$ denote ${\rm trace}$. By
\begin{align}
{\rm tr}[c(X)c(Y)]=-g^{M}(X, Y){\rm tr}[\texttt{id}],
\end{align}
\begin{align}
{\rm tr}[c(X)c(Y)c(Z)c(W)]&=g^{M}(X, W)g^{M}(Y, Z){\rm tr}[\texttt{id}]-g^{M}(X, Z)g^{M}(Y, W){\rm tr}[\texttt{id}]\\
&+g^{M}(X, Y)g^{M}(Z, W){\rm tr}[\texttt{id}],\nonumber
\end{align}
where $X,Y,Z,W\in \Gamma(TM).$ We calculate that
\begin{align}
&\sum_{i,j,k,l=1}^{n}{\rm tr}[R(J(e_{i}), J(e_{j}), e_{k}, e_{l})c(e_{i})c(e_{j})c(e_{k})c(e_{l})]\\
&=\sum_{i,j,k,l=1}^{n}R(J(e_{i}), J(e_{j}), e_{k}, e_{l}){\rm tr}[c(e_{i})c(e_{j})c(e_{k})c(e_{l})]\nonumber\\
&=\sum_{i,j=1}^{n}R(J(e_{i}), J(e_{j}), e_{j}, e_{i}){\rm tr}[\texttt{id}]-\sum_{i,j=1}^{n}R(J(e_{i}), J(e_{j}), e_{i}, e_{j}){\rm tr}[\texttt{id}]\nonumber\\
&=2\sum_{i,j=1}^{n}R(J(e_{i}), J(e_{j}), e_{j}, e_{i}){\rm tr}[\texttt{id}],\nonumber
\end{align}
\begin{align}
&\sum_{\nu,j=1}^{n}{\rm tr}[c[\nabla_{e_{j}}^{L}(J)e_{\nu}]c[(\nabla^{L}_{e_{\nu}}J)e_{j}]]=-\sum_{\nu,j=1}^{n}g^{M}(\nabla_{e_{j}}^{L}(J)e_{\nu}, (\nabla^{L}_{e_{\nu}}J)e_{j}){\rm tr}[\texttt{id}],
\end{align}
\begin{align}
&\sum_{\nu,j=1}^{n}{\rm tr}[c[J(e_{\nu})]c[(\nabla^{L}_{e_{j}}(\nabla^{L}_{e_{\nu}}(J)))e_{j}-(\nabla^{L}_{\nabla^{L}_{e_{j}}e_{\nu}}(J))e_{j}]]\\
&=-\sum_{\nu,j=1}^{n}g^{M}(J(e_{\nu}), (\nabla^{L}_{e_{j}}(\nabla^{L}_{e_{\nu}}(J)))e_{j}-(\nabla^{L}_{\nabla^{L}_{e_{j}}e_{\nu}}(J))e_{j}){\rm tr}[\texttt{id}],\nonumber
\end{align}
\begin{align}
&\sum_{\alpha,\nu,j=1}^{n}{\rm tr}[c[J(e_{\alpha})]c[(\nabla^{L}_{e_{\alpha}}J)e_{j}]c[J(e_{\nu})]c[(\nabla^{L}_{e_{\nu}}J)e_{j}]]\\
&=\sum_{\alpha,\nu,j=1}^{n}g^{M}(J(e_{\alpha}), (\nabla^{L}_{e_{\nu}}J)e_{j})g^{M}((\nabla^{L}_{e_{\alpha}}J)e_{j}, J(e_{\nu})){\rm tr}[\texttt{id}]\nonumber\\
&-\sum_{\nu,j=1}^{n}g^{M}((\nabla^{L}_{e_{\nu}}J)e_{j}, (\nabla^{L}_{e_{\nu}}J)e_{j}){\rm tr}[\texttt{id}]\nonumber\\
&+\sum_{\alpha,\nu,j=1}^{n}g^{M}(J(e_{\alpha}), (\nabla^{L}_{e_{\alpha}}J)e_{j})g^{M}(J(e_{\nu}), (\nabla^{L}_{e_{\nu}}J)e_{j}){\rm tr}[\texttt{id}].\nonumber
\end{align}

By applying the formulae shown in (2.20)-(2.26), we get:
\begin{thm} If $M$ is a $n$-dimensional almost product Riemannian spin manifold without boundary, we have the following:
\begin{align}
{\rm Wres}({D_{J}}^{-n+2})
=\frac{(n-2)(4\pi)^{\frac{n}{2}}}{(\frac{n}{2}-1)!}\int_{M}2^{\frac{n}{2}}\Big(&\frac{1}{4}\sum_{i,j=1}^{n}R(J(e_{i}), J(e_{j}), e_{j}, e_{i})
-\frac{1}{2}\sum_{\nu,j=1}^{n}g^{M}(\nabla_{e_{j}}^{L}(J)e_{\nu}, (\nabla^{L}_{e_{\nu}}J)e_{j})\\
&-\frac{1}{2}\sum_{\nu,j=1}^{n}g^{M}(J(e_{\nu}), (\nabla^{L}_{e_{j}}(\nabla^{L}_{e_{\nu}}(J)))e_{j}-(\nabla^{L}_{\nabla^{L}_{e_{j}}e_{\nu}}(J))e_{j})\nonumber\\
&-\frac{1}{4}\sum_{\alpha,\nu,j=1}^{n}g^{M}(J(e_{\alpha}), (\nabla^{L}_{e_{\nu}}J)e_{j})g^{M}((\nabla^{L}_{e_{\alpha}}J)e_{j}, J(e_{\nu}))\nonumber\\
&-\frac{1}{4}\sum_{\alpha,\nu,j=1}^{n}g^{M}(J(e_{\alpha}), (\nabla^{L}_{e_{\alpha}}J)e_{j})g^{M}(J(e_{\nu}), (\nabla^{L}_{e_{\nu}}J)e_{j})\nonumber\\
&+\frac{1}{4}\sum_{\nu,j=1}^{n}g^{M}((\nabla^{L}_{e_{\nu}}J)e_{j}, (\nabla^{L}_{e_{\nu}}J)e_{j}))-\frac{1}{12}s\Big)d{\rm Vol_{M} },\nonumber
\end{align}
where $s$ is the scalar curvature.
\end{thm}

\section{ A Kastler-Kalau-Walze type theorem for $4$-dimensional manifolds with boundary}
\indent Firstly, we explain the basic notions of  Boutet de Monvel's calculus and the definition of the noncommutative residue for manifolds with boundary that will be used throughout the paper. For the details, see Ref.\cite{Wa3}.\\
\indent Let $U\subset M$ be a collar neighborhood of $\partial M$ which is diffeomorphic with $\partial M\times [0,1)$. By the definition of $h(x_n)\in C^{\infty}([0,1))$
and $h(x_n)>0$, there exists $\widehat{h}\in C^{\infty}((-\varepsilon,1))$ such that $\widehat{h}|_{[0,1)}=h$ and $\widehat{h}>0$ for some
sufficiently small $\varepsilon>0$. Then there exists a metric $g'$ on $\widetilde{M}=M\bigcup_{\partial M}\partial M\times
(-\varepsilon,0]$ which has the form on $U\bigcup_{\partial M}\partial M\times (-\varepsilon,0 ]$
\begin{equation}
g'=\frac{1}{\widehat{h}(x_{n})}g^{\partial M}+dx _{n}^{2} ,
\end{equation}
such that $g'|_{M}=g$. We fix a metric $g'$ on the $\widetilde{M}$ such that $g'|_{M}=g$.

We define the Fourier transformation $F'$  by
\begin{equation}
F':L^2({\bf R}_t)\rightarrow L^2({\bf R}_v);~F'(u)(v)=\int e^{-ivt}u(t)dt\\
\end{equation}
and let
\begin{equation}
r^{+}:C^\infty ({\bf R})\rightarrow C^\infty (\widetilde{{\bf R}^+});~ f\rightarrow f|\widetilde{{\bf R}^+};~
\widetilde{{\bf R}^+}=\{x\geq0;x\in {\bf R}\}.
\end{equation}
 where $\Phi({\bf R})$
denotes the Schwartz space and $\Phi(\widetilde{{\bf R}^+}) =r^+\Phi({\bf R})$, $\Phi(\widetilde{{\bf R}^-}) =r^-\Phi({\bf R})$. We define $H^+=F'(\Phi(\widetilde{{\bf R}^+}));~ H^-_0=F'(\Phi(\widetilde{{\bf R}^-}))$ which satisfies
$H^+\bot H^-_0$. We have the following
 property: $h\in H^+~(H^-_0)$ if and only if $h\in C^\infty({\bf R})$ which has an analytic extension to the lower (upper) complex
half-plane $\{{\rm Im}\xi<0\}~(\{{\rm Im}\xi>0\})$ such that for all nonnegative integer $l$,
 \begin{equation}
\frac{d^{l}h}{d\xi^l}(\xi)\sim\sum^{\infty}_{k=1}\frac{d^l}{d\xi^l}(\frac{c_k}{\xi^k}),
\end{equation}
as $|\xi|\rightarrow +\infty,{\rm Im}\xi\leq0~({\rm Im}\xi\geq0)$.\\
\indent Let $H'$ be the space of all polynomials and $H^-=H^-_0\bigoplus H';~H=H^+\bigoplus H^-.$ Denote by $\pi^+~(\pi^-)$ respectively the projection on $H^+~(H^-)$. For calculations, we take $H=\widetilde H=\{$rational functions having no poles on the real axis$\}$ ($\tilde{H}$ is a dense set in the topology of $H$). Then on $\tilde{H}$,
 \begin{equation}
\pi^+h(\xi_0)=\frac{1}{2\pi i}\lim_{u\rightarrow 0^{-}}\int_{\Gamma^+}\frac{h(\xi)}{\xi_0+iu-\xi}d\xi,
\end{equation}
where $\Gamma^+$ is a Jordan close curve
included ${\rm Im}(\xi)>0$ surrounding all the singularities of $h$ in the upper half-plane and
$\xi_0\in {\bf R}$. Similarly, define $\pi'$ on $\tilde{H}$,
\begin{equation}
\pi'h=\frac{1}{2\pi}\int_{\Gamma^+}h(\xi)d\xi.
\end{equation}
So, $\pi'(H^-)=0$. For $h\in H\bigcap L^1({\bf R})$, $\pi'h=\frac{1}{2\pi}\int_{{\bf R}}h(v)dv$ and for $h\in H^+\bigcap L^1({\bf R})$, $\pi'h=0$.

Let $M$ be a $n$-dimensional compact oriented manifold with boundary $\partial M$.
Denote by $\mathcal{B}$ Boutet de Monvel's algebra, we recall the main theorem in \cite{Wa3,FGLS}.
\begin{thm}\label{th:32}{\rm\cite{FGLS}}{\bf(Fedosov-Golse-Leichtnam-Schrohe)}
 Let $X$ and $\partial X$ be connected, ${\rm dim}X=n\geq3$,
 $A=\left(\begin{array}{lcr}\pi^+P+G &   K \\
T &  S    \end{array}\right)$ $\in \mathcal{B}$ , and denote by $p$, $b$ and $s$ the local symbols of $P,G$ and $S$ respectively.
 Define:
 \begin{align}
{\rm{\widetilde{Wres}}}(A)&=\int_X\int_{\bf S}{\rm{tr}}_E\left[p_{-n}(x,\xi)\right]\sigma(\xi)dx \\
&+2\pi\int_ {\partial X}\int_{\bf S'}\left\{{\rm tr}_E\left[({\rm{tr}}b_{-n})(x',\xi')\right]+{\rm{tr}}
_F\left[s_{1-n}(x',\xi')\right]\right\}\sigma(\xi')dx'.\nonumber
\end{align}
Then~~ a) ${\rm \widetilde{Wres}}([A,B])=0 $, for any
$A,B\in\mathcal{B}$;~~ b) It is a unique continuous trace on
$\mathcal{B}/\mathcal{B}^{-\infty}$.
\end{thm}

\begin{defn}{\rm\cite{Wa3} }
Lower dimensional volumes of spin manifolds with boundary are defined by
 \begin{equation}
{\rm Vol}^{(p_1,p_2)}_nM:= \widetilde{{\rm Wres}}[\pi^+D^{-p_1}\circ\pi^+D^{-p_2}].
\end{equation}
\end{defn}
We can get the spin structure on $\widetilde{M}$ by extending the spin structure on $M.$ Let $D$ be the Dirac operator associated to $g'$ on the spinors bundle $S(T\widetilde{M}).$ By \cite{Wa3}, we get
\begin{align}
\widetilde{{\rm Wres}}[\pi^+D^{-p_1}\circ\pi^+D^{-p_2}]=\int_M\int_{|\xi|=1}{\rm
trace}_{S(TM)}[\sigma_{-n}(D^{-p_1-p_2})]\sigma(\xi)dx+\int_{\partial M}\Phi
\end{align}
and
\begin{align}
\Phi&=\int_{|\xi'|=1}\int^{+\infty}_{-\infty}\sum^{\infty}_{j, k=0}\sum\frac{(-i)^{|\alpha|+j+k+1}}{\alpha!(j+k+1)!}
\times {\rm trace}_{S(TM)}[\partial^j_{x_n}\partial^\alpha_{\xi'}\partial^k_{\xi_n}\sigma^+_{r}(D^{-p_1})(x',0,\xi',\xi_n)
\\
&\times\partial^\alpha_{x'}\partial^{j+1}_{\xi_n}\partial^k_{x_n}\sigma_{l}(D^{-p_2})(x',0,\xi',\xi_n)]d\xi_n\sigma(\xi')dx',\nonumber
\end{align}
 where the sum is taken over $r+l-k-|\alpha|-j-1=-n,~~r\leq -p_1,l\leq -p_2$.

 Since $[\sigma_{-n}(D^{-p_1-p_2})]|_M$ has the same expression as $\sigma_{-n}(D^{-p_1-p_2})$ in the case of manifolds without
boundary, so locally we can compute the first term by \cite{Ka}, \cite{KW}, \cite{Wa3}, \cite{Po}.
%The following proposition is the motivation
%of the definition of lower dimensional volumes of spin manifolds with boundary \cite{Wa4}.

For any fixed point $x_0\in\partial M$, we choose the normal coordinates
$U$ of $x_0$ in $\partial M$ (not in $M$) and compute $\Phi(x_0)$ in the coordinates $\widetilde{U}=U\times [0,1)\subset M$ and the
metric $\frac{1}{h(x_n)}g^{\partial M}+dx_n^2.$ The dual metric of $g^M$ on $\widetilde{U}$ is ${h(x_n)}g^{\partial M}+dx_n^2.$  Write
$g^M_{ij}=g^M(\frac{\partial}{\partial x_i},\frac{\partial}{\partial x_j});~ g_M^{ij}=g^M(dx_i,dx_j)$, then

\begin{equation}
[g^M_{i,j}]= \left[\begin{array}{lcr}
  \frac{1}{h(x_n)}[g_{i,j}^{\partial M}]  & 0  \\
   0  &  1
\end{array}\right];~~~
[g_M^{i,j}]= \left[\begin{array}{lcr}
  h(x_n)[g^{i,j}_{\partial M}]  & 0  \\
   0  &  1
\end{array}\right]
\end{equation}
and
\begin{equation}
\partial_{x_s}g_{ij}^{\partial M}(x_0)=0, 1\leq i,j\leq n-1; ~~~g_{ij}^M(x_0)=\delta_{ij}.
\end{equation}
\indent $\{e_1, \cdots, e_n\}$ be an orthonormal frame field in $U$ about $g^{\partial M}$ which is parallel along geodesics and $e_i(x_0)=\frac{\partial}{\partial{x_i}}(x_0).$ We review the following three lemmas.
\begin{lem}{\rm \cite{Wa3}}\label{le:32}
With the metric $g^{M}$ on $M$ near the boundary
\begin{eqnarray}
\partial_{x_j}(|\xi|_{g^M}^2)(x_0)&=&\left\{
       \begin{array}{c}
        0,  ~~~~~~~~~~ ~~~~~~~~~~ ~~~~~~~~~~~~~{\rm if }~j<n, \\[2pt]
       h'(0)|\xi'|^{2}_{g^{\partial M}},~~~~~~~~~~~~~~~~~~~~{\rm if }~j=n;
       \end{array}
    \right. \\
\partial_{x_j}[c(\xi)](x_0)&=&\left\{
       \begin{array}{c}
      0,  ~~~~~~~~~~ ~~~~~~~~~~ ~~~~~~~~~~~~~{\rm if }~j<n,\\[2pt]
\partial x_{n}(c(\xi'))(x_{0}), ~~~~~~~~~~~~~~~~~{\rm if }~j=n,
       \end{array}
    \right.
\end{eqnarray}
where $\xi=\xi'+\xi_{n}dx_{n}$.
\end{lem}
\begin{lem}{\rm \cite{Wa3}}\label{le:32}With the metric $g^{M}$ on $M$ near the boundary
\begin{align}
\omega_{s,t}(e_i)(x_0)&=\left\{
       \begin{array}{c}
        \omega_{n,i}(e_i)(x_0)=\frac{1}{2}h'(0),  ~~~~~~~~~~ ~~~~~~~~~~~{\rm if }~s=n,t=i,i<n; \\[2pt]
       \omega_{i,n}(e_i)(x_0)=-\frac{1}{2}h'(0),~~~~~~~~~~~~~~~~~~~{\rm if }~s=i,t=n,i<n;\\[2pt]
    \omega_{s,t}(e_i)(x_0)=0,~~~~~~~~~~~~~~~~~~~~~~~~~~~other~cases,~~~~~~~~~\\[2pt]
       \end{array}
    \right.
\end{align}
where $(\omega_{s,t})$ denotes the connection matrix of Levi-Civita connection $\nabla^L$.
\end{lem}
\begin{lem}{\rm \cite{Wa3}}
\begin{align}
\Gamma_{st}^k(x_0)&=\left\{
       \begin{array}{c}
        \Gamma^n_{ii}(x_0)=\frac{1}{2}h'(0),~~~~~~~~~~ ~~~~~~~~~~~{\rm if }~s=t=i,k=n,i<n; \\[2pt]
        \Gamma^i_{ni}(x_0)=-\frac{1}{2}h'(0),~~~~~~~~~~~~~~~~~~~{\rm if }~s=n,t=i,k=i,i<n;\\[2pt]
        \Gamma^i_{in}(x_0)=-\frac{1}{2}h'(0),~~~~~~~~~~~~~~~~~~~{\rm if }~s=i,t=n,k=i,i<n,\\[2pt]
        \Gamma_{st}^i(x_0)=0,~~~~~~~~~~~~~~~~~~~~~~~~~~~other~cases.~~~~~~~~~
       \end{array}
    \right.
\end{align}
\end{lem}
\indent Similar to (3.9) and (3.10), we firstly compute
\begin{equation}
\widetilde{{\rm Wres}}[\pi^+{{D}_{J}}^{-1}\circ\pi^+{{D}_{J}}^{-1}]=\int_M\int_{|\xi|=1}{\rm
trace}_{S(TM)}[\sigma_{-4}({{D}_{J}}^{-2})]\sigma(\xi)dx+\int_{\partial M}\Psi,
\end{equation}
where
\begin{align}
\Psi &=\int_{|\xi'|=1}\int^{+\infty}_{-\infty}\sum^{\infty}_{j, k=0}\sum\frac{(-i)^{|\alpha|+j+k+1}}{\alpha!(j+k+1)!}
\times {\rm trace}_{S(TM)}[\partial^j_{x_n}\partial^\alpha_{\xi'}\partial^k_{\xi_n}\sigma^+_{r}({{D}_{J}}^{-1})\\
&(x',0,\xi',\xi_n)\times\partial^\alpha_{x'}\partial^{j+1}_{\xi_n}\partial^k_{x_n}\sigma_{l}({{D}_{J}}^{-1})(x',0,\xi',\xi_n)]d\xi_n\sigma(\xi')dx',\nonumber
\end{align}
the sum is taken over $r+l-k-j-|\alpha|-1=-4, r\leq -1, l\leq-1$.\\

\indent
Then we give the interior of $\widetilde{{\rm Wres}}[\pi^+{{D}_{J}}^{-1}\circ\pi^+{{D}_{J}}^{-1}]$,
\begin{align}
&\int_M\int_{|\xi|=1}{\rm trace}_{S(TM)}[\sigma_{-4}({{D}_{J}}^{-2})]\sigma(\xi)dx=32\pi^{2}\\
&\int_{M}\Big(\sum_{i,j=1}^{n}R(J(e_{i}), J(e_{j}), e_{j}, e_{i})
-2\sum_{\nu,j=1}^{n}g^{M}(\nabla_{e_{j}}^{L}(J)e_{\nu}, (\nabla^{L}_{e_{\nu}}J)e_{j})\\
&-2\sum_{\nu,j=1}^{n}g^{M}(J(e_{\nu}), (\nabla^{L}_{e_{j}}(\nabla^{L}_{e_{\nu}}(J)))e_{j}-(\nabla^{L}_{\nabla^{L}_{e_{j}}e_{\nu}}(J))e_{j})\nonumber\\
&-\sum_{\alpha,\nu,j=1}^{n}g^{M}(J(e_{\alpha}), (\nabla^{L}_{e_{\nu}}J)e_{j})g^{M}((\nabla^{L}_{e_{\alpha}}J)e_{j}, J(e_{\nu}))\nonumber\\
&-\sum_{\alpha,\nu,j=1}^{n}g^{M}(J(e_{\alpha}), (\nabla^{L}_{e_{\alpha}}J)e_{j})g^{M}(J(e_{\nu}), (\nabla^{L}_{e_{\nu}}J)e_{j})\nonumber\\
&+\sum_{\nu,j=1}^{n}g^{M}((\nabla^{L}_{e_{\nu}}J)e_{j}, (\nabla^{L}_{e_{\nu}}J)e_{j}))-\frac{1}{3}s\Big)d{\rm Vol_{M} },\nonumber
\end{align}
where $\Omega_{n}=\frac{2\pi^\frac{n}{2}}{\Gamma(\frac{n}{2})}.$

\indent We begin by computing $\int_{\partial M} \Psi$. Since, some operators have the following symbols.
\begin{lem} The following identities hold:
\begin{align}
\sigma_1({D}_{J})&=ic[J(\xi)];\\
\sigma_0({D}_{J})&=
-\frac{1}{4}\sum_{i,j,k=1}^{n}\omega_{j,k}(e_i)c[J(e_i)]c(e_j)c(e_k).
\end{align}
\end{lem}
\indent Write
 \begin{eqnarray}
D_x^{\alpha}&=(-i)^{|\alpha|}\partial_x^{\alpha};
~\sigma({D}_{J})=p_1+p_0;
~\sigma({{D}_{J}}^{-1})=\sum^{\infty}_{j=1}q_{-j}.
\end{eqnarray}

\indent By the composition formula of pseudodifferential operators, we have
\begin{align}
1=\sigma({D}_{J}\circ {{D}_{J}}^{-1})&=\sum_{\alpha}\frac{1}{\alpha!}\partial^{\alpha}_{\xi}[\sigma({D}_{J})]
{D}_x^{\alpha}[\sigma({{D}_{J}}^{-1})]\\
&=(p_1+p_0)(q_{-1}+q_{-2}+q_{-3}+\cdots)\nonumber\\
&~~~+\sum_j(\partial_{\xi_j}p_1+\partial_{\xi_j}p_0)(
D_{x_j}q_{-1}+D_{x_j}q_{-2}+D_{x_j}q_{-3}+\cdots)\nonumber\\
&=p_1q_{-1}+(p_1q_{-2}+p_0q_{-1}+\sum_j\partial_{\xi_j}p_1D_{x_j}q_{-1})+\cdots,\nonumber
\end{align}
so
\begin{equation}
q_{-1}=p_1^{-1};~q_{-2}=-p_1^{-1}[p_0p_1^{-1}+\sum_j\partial_{\xi_j}p_1D_{x_j}(p_1^{-1})].
\end{equation}
\begin{lem} The following identities hold:
\begin{align}
\sigma_{-1}({{D}_{J}}^{-1})&=\frac{ic[J(\xi)]}{|\xi|^2};\\
\sigma_{-2}({{D}_{J}}^{-1})&=\frac{c[J(\xi)]\sigma_{0}({D}_{J})c[J(\xi)]}{|\xi|^4}+\frac{c[J(\xi)]}{|\xi|^6}\sum_ {j=1}^{n} c[J(dx_j)]
\Big[\partial_{x_j}(c[J(\xi)])|\xi|^2-c[J(\xi)]\partial_{x_j}(|\xi|^2)\Big].
\end{align}
\end{lem}
\indent When $n=4$, then ${\rm tr}[{\rm \texttt{id}}]=4,$ since the sum is taken over $
r+l-k-j-|\alpha|-1=-4,~~r\leq -1,l\leq-1,$ then we have the following five cases:\\

\noindent  {\bf case a)~I)}~$r=-1,~l=-1,~k=j=0,~|\alpha|=1$\\

\noindent By applying the formula shown in (3.18), we can calculate
\begin{equation}
\Psi_1=-\int_{|\xi'|=1}\int^{+\infty}_{-\infty}\sum_{|\alpha|=1}
 {\rm trace}[\partial^\alpha_{\xi'}\pi^+_{\xi_n}\sigma_{-1}({{D}_{J}}^{-1})\times
 \partial^\alpha_{x'}\partial_{\xi_n}\sigma_{-1}({{D}_{J}}^{-1})](x_0)d\xi_n\sigma(\xi')dx'.
\end{equation}
By Lemma 3.3, for $i<n,$ then
\begin{align}
\partial_{x_i}\left(\frac{ic[J(\xi)]}{|\xi|^2}\right)(x_0)
=\frac{i\partial_{x_i}(c[J(\xi)])(x_0)}{|\xi|^2}-\frac{ic[J(\xi)]\partial_{x_i}(|\xi|^2)(x_0)}{|\xi|^4}
=\frac{i\partial_{x_i}(c[J(\xi)])(x_0)}{|\xi|^2},
\end{align}
where $J(dx_{p})=\sum^{n}_{h=1}a^{p}_{h}dx_{h}.$\\
It is clear that
\begin{align}
\partial_{x_i}\left(\frac{ic[J(\xi)]}{|\xi|^2}\right)(x_0)
&=\frac{i\partial_{x_i}(c[J(\sum^{n}_{p=1}\xi_{p}dx_{p})])(x_0)}{|\xi|^2}=\frac{i\sum^{n}_{p=1}\xi_{p}\partial_{x_i}(c(\sum^{n}_{h=1}a^{p}_{h}dx_{h}))(x_0)}{|\xi|^2}\\
&=\frac{i\sum^{n}_{p,h=1}\xi_{p}\partial_{x_i}(a^{p}_{h})c(dx_{h})(x_0)}{|\xi|^2}+\frac{i\sum^{n}_{p,h=1}\xi_{p}a^{p}_{h}\partial_{x_i}(c(dx_{h}))(x_0)}{|\xi|^2}\nonumber\\
&=\frac{i\sum^{n}_{p,h=1}\xi_{p}\partial_{x_i}(a^{p}_{h})c(dx_{h})(x_0)}{|\xi|^2}.\nonumber
\end{align}
Then we have
\begin{align}
\partial_{x_i}\left(\frac{ic[J(\xi)]}{|\xi|^2}\right)(x_0)|_{|\xi'|=1}
&=\frac{i\sum^{n}_{h=1}\sum^{n-1}_{p=1}\xi_{p}\partial_{x_i}(a^{p}_{h})c(dx_{h})(x_0)}{1+\xi_{n}^2}+\frac{i\sum^{n}_{h=1}\xi_{n}\partial_{x_i}(a^{n}_{h})c(dx_{h})(x_0)}{1+\xi_{n}^2}.
\end{align}
An easy computation shows that
\begin{align}
\partial_{\xi_{n}}\partial_{x_i}\left(\frac{ic[J(\xi)]}{|\xi|^2}\right)(x_0)|_{|\xi'|=1}
&=i\sum^{n}_{h=1}\sum^{n-1}_{p=1}\xi_{p}\partial_{x_i}(a^{p}_{h})c(dx_{h})\partial_{\xi_{n}}\left(\frac{1}{1+\xi_{n}^2}\right)\\
&+i\sum^{n}_{h=1}\partial_{x_i}(a^{n}_{h})c(dx_{h})\partial_{\xi_{n}}\left(\frac{\xi_{n}}{1+\xi_{n}^2}\right)\nonumber\\
&=-\frac{2i\xi_{n}}{(1+\xi_{n}^2)^2}\sum^{n}_{h=1}\sum^{n-1}_{p=1}\xi_{p}\partial_{x_i}(a^{p}_{h})c(dx_{h})\nonumber\\
&+\frac{i(1-\xi_{n}^2)}{(1+\xi_{n}^2)^2}\sum^{n}_{h=1}\partial_{x_i}(a^{n}_{h})c(dx_{h}).\nonumber
\end{align}
Likewise,
\begin{align}
\partial_{\xi_i}\left(\frac{ic[J(\xi)]}{|\xi|^2}\right)(x_0)
=\partial_{\xi_i}\left(\frac{i\sum^{n}_{q=1}\xi_{q}c[J(dx_q)]}{|\xi|^2}\right)(x_0),
\end{align}
\begin{align}
\partial_{\xi_i}\left(\frac{ic[J(\xi)]}{|\xi|^2}\right)(x_0)|_{|\xi'|=1}
=\frac{ic[J(dx_i)]}{1+\xi_{n}^2}-\frac{2i\sum^{n}_{q=1}\xi_{i}\xi_{q}c[J(dx_q)]}{(1+\xi_{n}^2)^2}.
\end{align}
Then we get
\begin{align}
\pi^+_{\xi_n}\partial_{\xi_i}\left(\frac{ic[J(\xi)]}{|\xi|^2}\right)|_{|\xi'|=1}
&=ic[J(dx_i)]\pi^+_{\xi_n}\left(\frac{1}{1+\xi_{n}^2}\right)-2i\xi_{i}c[J(dx_n)]\pi^+_{\xi_n}\left(\frac{\xi_{n}}{(1+\xi_{n}^2)^2}\right)\\
&-2i\sum^{n-1}_{q=1}\xi_{i}\xi_{q}c[J(dx_q)]\pi^+_{\xi_n}\left(\frac{1}{(1+\xi_{n}^2)^2}\right).\nonumber
\end{align}
By the Cauchy integral formula, we have
\begin{align}
\pi^+_{\xi_n}\left(\frac{1}{1+\xi_{n}^2}\right)(x_0)&=\frac{1}{2\pi i}{\rm lim}_{u\rightarrow
0^-}\int_{\Gamma^+}\frac{\frac{1}{(\eta_n+i)(\xi_n+iu-\eta_n)}}{(\eta_n-i)}d\eta_n\\
&=\left[\frac{1}{(\eta_n+i)(\xi_n-\eta_n)}\right]|_{\eta_n=i}\nonumber\\
&=\frac{1}{2i(\xi_n-i)},\nonumber
\end{align}
\begin{align}
\pi^+_{\xi_n}\left(\frac{1}{(1+\xi_{n}^2)^2}\right)(x_0)
&=\frac{1}{2\pi i}{\rm lim}_{u\rightarrow
0^-}\int_{\Gamma^+}\frac{\frac{1}{(\eta_n+i)^2(\xi_n+iu-\eta_n)}}{(\eta_n-i)^2}d\eta_n\\
&=\left[\frac{1}{(\eta_n+i)^2(\xi_n-\eta_n)}\right]^{(1)}|_{\eta_n=i}\nonumber\\
&=-\frac{i\xi_n+2}{4(\xi_n-i)^2}\nonumber
\end{align}
and
\begin{align}
\pi^+_{\xi_n}\left(\frac{\xi_n}{(1+\xi_{n}^2)^2}\right)(x_0)
&=\frac{1}{2\pi i}{\rm lim}_{u\rightarrow
0^-}\int_{\Gamma^+}\frac{\frac{\eta_n}{(\eta_n+i)^2(\xi_n+iu-\eta_n)}}{(\eta_n-i)^2}d\eta_n\\
&=\left[\frac{\eta_n}{(\eta_n+i)^2(\xi_n-\eta_n)}\right]^{(1)}|_{\eta_n=i}\nonumber\\
&=-\frac{i}{4(\xi_n-i)^2}.\nonumber
\end{align}
Means that
\begin{align}
\pi^+_{\xi_n}\partial_{\xi_i}\left(\frac{ic[J(\xi)]}{|\xi|^2}\right)(x_0)|_{|\xi'|=1}
&=\frac{1}{2(\xi_n-i)}c[J(dx_i)]-\frac{1}{2(\xi_n-i)^2}\xi_{i}c[J(dx_n)]\\
&+\frac{2i-\xi_n}{2(\xi_n-i)^2}\sum^{n-1}_{q=1}\xi_{i}\xi_{q}c[J(dx_q)].\nonumber
\end{align}
On account of the above result,
\begin{align}
&\sum_{|\alpha|=1}{\rm trace}[\partial^\alpha_{\xi'}\pi^+_{\xi_n}\sigma_{-1}({{D}_{J}}^{-1})\times\partial^\alpha_{x'}\partial_{\xi_n}\sigma_{-1}({{D}_{J}}^{-1})](x_0)|_{|\xi'|=1}\\
=&-\frac{i\xi_n}{(\xi_n-i)^3(\xi_n+i)^2}\sum_{h,\beta=1}^{n}\sum_{i,p=1}^{n-1}{\rm tr}[\xi_{p}a_{\beta}^{i}\partial_{x_i}(a_{h}^{p})c(dx_{\beta})c(dx_{h})]\nonumber\\
&+\frac{i(1-\xi^2_n)}{2(\xi_n-i)^3(\xi_n+i)^2}\sum_{h,\beta=1}^{n}\sum_{i=1}^{n-1}{\rm tr}[a_{\beta}^{i}\partial_{x_i}(a_{h}^{n})c(dx_{\beta})c(dx_{h})]\nonumber\\
&-\frac{i\xi_n(2i-\xi_n)}{(\xi_n-i)^4(\xi_n+i)^2}\sum_{h,\beta=1}^{n}\sum_{i,q,p=1}^{n-1}{\rm tr}[\xi_{i}\xi_{q}\xi_{p}a_{\beta}^{q}\partial_{x_i}(a_{h}^{p})c(dx_{\beta})c(dx_{h})]\nonumber\\
&+\frac{i(2i-\xi_n)(1-\xi_n^2)}{2(\xi_n-i)^4(\xi_n+i)^2}\sum_{h,\beta=1}^{n}\sum_{i,q=1}^{n-1}{\rm tr}[\xi_{i}\xi_{q}a_{\beta}^{q}\partial_{x_i}(a_{h}^{n})c(dx_{\beta})c(dx_{h})]\nonumber
\end{align}
\begin{align}
&+\frac{i\xi_n}{(\xi_n-i)^4(\xi_n+i)^2}\sum_{h,\beta=1}^{n}\sum_{i,p=1}^{n-1}{\rm tr}[\xi_{i}\xi_{p}a_{\beta}^{n}\partial_{x_i}(a_{h}^{p})c(dx_{\beta})c(dx_{h})]\nonumber\\
&-\frac{i(1-\xi_n^2)}{2(\xi_n-i)^4(\xi_n+i)^2}\sum_{h,\beta=1}^{n}\sum_{i=1}^{n-1}{\rm tr}[\xi_{i}a_{\beta}^{n}\partial_{x_i}(a_{h}^{n})c(dx_{\beta})c(dx_{h})].\nonumber
\end{align}
Since
\begin{align}
c(e_i)c(e_j)+c(e_j)c(e_i)=-2\delta_i^j,
\end{align}
then by the relation of the Clifford action and ${\rm tr}{AB}={\rm tr}{BA}$,  we have
\begin{align}
&\sum_{h,\beta=1}^{n}\sum_{i,p=1}^{n-1}{\rm tr}[\xi_{p}a_{\beta}^{i}\partial_{x_i}(a_{h}^{p})c(dx_{\beta})c(dx_{h})]=\sum_{h,\beta=1}^{n}\sum_{i,p=1}^{n-1}\xi_{p}a_{\beta}^{i}\partial_{x_i}(a_{h}^{p}){\rm tr}[c(dx_{\beta})c(dx_{h})]\\
&=-\sum_{h,\beta=1}^{n}\sum_{i,p=1}^{n-1}\xi_{p}a_{\beta}^{i}\partial_{x_i}(a_{h}^{p}){\rm tr}[c(dx_{h})c(dx_{\beta})]-2\sum_{\beta=1}^{n}\sum_{i,p=1}^{n-1}\xi_{p}a_{\beta}^{i}\partial_{x_i}(a_{\beta}^{p}){\rm tr}[\texttt{id}]\nonumber\\
&=-\sum_{h,\beta=1}^{n}\sum_{i,p=1}^{n-1}\xi_{p}a_{\beta}^{i}\partial_{x_i}(a_{h}^{p}){\rm tr}[c(dx_{\beta})c(dx_{h})]-2\sum_{\beta=1}^{n}\sum_{i,p=1}^{n-1}\xi_{p}a_{\beta}^{i}\partial_{x_i}(a_{\beta}^{p}){\rm tr}[\texttt{id}],\nonumber
\end{align}
that is
\begin{align}
\sum_{h,\beta=1}^{n}\sum_{i,p=1}^{n-1}{\rm tr}[\xi_{p}a_{\beta}^{i}\partial_{x_i}(a_{h}^{p})c(dx_{\beta})c(dx_{h})]=-\sum_{\beta=1}^{n}\sum_{i,p=1}^{n-1}\xi_{p}a_{\beta}^{i}\partial_{x_i}(a_{\beta}^{p}){\rm tr}[\texttt{id}].
\end{align}
Similarly, we have the following equalities:
\begin{align}
\sum_{h,\beta=1}^{n}\sum_{i=1}^{n-1}{\rm tr}[a_{\beta}^{i}\partial_{x_i}(a_{h}^{n})c(dx_{\beta})c(dx_{h})]=-\sum_{\beta=1}^{n}\sum_{i=1}^{n-1}a_{\beta}^{i}\partial_{x_i}(a_{\beta}^{n}){\rm tr}[\texttt{id}];
\end{align}
\begin{align}
\sum_{h,\beta=1}^{n}\sum_{i,q,p=1}^{n-1}{\rm tr}[\xi_{i}\xi_{q}\xi_{p}a_{\beta}^{q}\partial_{x_i}(a_{h}^{p})c(dx_{\beta})c(dx_{h})]=-\sum_{\beta=1}^{n}\sum_{i,q,p=1}^{n-1}\xi_{i}\xi_{q}\xi_{p}a_{\beta}^{q}\partial_{x_i}(a_{\beta}^{p})
{\rm tr}[\texttt{id}];
\end{align}
\begin{align}
\sum_{h,\beta=1}^{n}\sum_{i,q=1}^{n-1}{\rm tr}[\xi_{i}\xi_{q}a_{\beta}^{q}\partial_{x_i}(a_{h}^{n})c(dx_{\beta})c(dx_{h})]=-\sum_{\beta=1}^{n}\sum_{i,q=1}^{n-1}\xi_{i}\xi_{q}a_{\beta}^{q}\partial_{x_i}(a_{\beta}^{n}){\rm tr}[\texttt{id}];
\end{align}
\begin{align}
\sum_{h,\beta=1}^{n}\sum_{i,p=1}^{n-1}{\rm tr}[\xi_{i}\xi_{p}a_{\beta}^{n}\partial_{x_i}(a_{h}^{p})c(dx_{\beta})c(dx_{h})]=-\sum_{\beta=1}^{n}\sum_{i,p=1}^{n-1}\xi_{i}\xi_{p}a_{\beta}^{n}\partial_{x_i}(a_{\beta}^{p}){\rm tr}[\texttt{id}];
\end{align}
\begin{align}
\sum_{h,\beta=1}^{n}\sum_{i=1}^{n-1}{\rm tr}[\xi_{i}a_{\beta}^{n}\partial_{x_i}(a_{h}^{n})c(dx_{\beta})c(dx_{h})]=-\sum_{\beta=1}^{n}\sum_{i=1}^{n-1}\xi_{i}a_{\beta}^{n}\partial_{x_i}(a_{\beta}^{n}){\rm tr}[\texttt{id}].
\end{align}
We note that $i<n,$  $\int_{|\xi'|=1}{\{\xi_{i_1}\cdot\cdot\cdot\xi_{i_{2d+1}}}\}\sigma(\xi')=0,$ then we have
\begin{align}
\Psi_1
&=-\int_{|\xi'|=1}\int^{+\infty}_{-\infty}\sum_{|\alpha|=1}{\rm trace}[\partial^\alpha_{\xi'}\pi^+_{\xi_n}\sigma_{-1}({{D}_{J}}^{-1})\times \partial^\alpha_{x'}\partial_{\xi_n}\sigma_{-1}({{D}_{J}}^{-1})](x_0)d\xi_n\sigma(\xi')dx'\\
&=\int_{|\xi'|=1}\int^{+\infty}_{-\infty}\frac{i(1-\xi^2_n)}{2(\xi_n-i)^3(\xi_n+i)^2}\sum_{\beta=1}^{n}\sum_{i=1}^{n-1}a_{\beta}^{i}\partial_{x_i}(a_{\beta}^{n}){\rm tr}[\texttt{id}]d\xi_n\sigma(\xi')dx'\nonumber\\
&+\int_{|\xi'|=1}\int^{+\infty}_{-\infty}\frac{i(2i-\xi_n)(1-\xi_n^2)}{2(\xi_n-i)^4(\xi_n+i)^2}\sum_{\beta=1}^{n}\sum_{i,q=1}^{n-1}\xi_{i}\xi_{q}a_{\beta}^{q}\partial_{x_i}(a_{\beta}^{n}){\rm tr}[\texttt{id}]d\xi_n\sigma(\xi')dx'\nonumber\\
&+\int_{|\xi'|=1}\int^{+\infty}_{-\infty}\frac{i\xi_n}{(\xi_n-i)^4(\xi_n+i)^2}\sum_{\beta=1}^{n}\sum_{i,p=1}^{n-1}\xi_{i}\xi_{p}a_{\beta}^{n}\partial_{x_i}(a_{\beta}^{p}){\rm tr}[\texttt{id}]d\xi_n\sigma(\xi')dx'\nonumber.\nonumber
\end{align}
From \cite{Ka}, we have $\int_{|\xi'|=1}\xi_i\xi_j=\frac{4\pi}{3}\delta_i^j,$ then
\begin{align}
\Psi_1
&=\sum_{\beta=1}^{n}\sum_{i=1}^{n-1}a_{\beta}^{i}\partial_{x_i}(a_{\beta}^{n}){\rm tr}[\texttt{id}]\Omega_3\frac{2\pi i}{2!}\Big[\frac{i(1-\xi^2_n)}{2(\xi_n+i)^2}\Big]^{(2)}|_{\xi_n=i}dx'\\
&+\sum_{\beta=1}^{n}\sum_{i=1}^{n-1}a_{\beta}^{i}\partial_{x_i}(a_{\beta}^{n}){\rm tr}[\texttt{id}]\Omega_3\frac{4\pi}{3}\frac{2\pi i}{3!}\Big[\frac{i(2i-\xi_n)(1-\xi_n^2)}{2(\xi_n+i)^2}\Big]^{(3)}|_{\xi_n=i}dx'\nonumber\\
&+\sum_{\beta=1}^{n}\sum_{i=1}^{n-1}a_{\beta}^{n}\partial_{x_i}(a_{\beta}^{i}){\rm tr}[\texttt{id}]\Omega_3\frac{4\pi}{3}\frac{2\pi i}{3!}\Big[\frac{i\xi_n}{(\xi_n+i)^2}\Big]^{(3)}|_{\xi_n=i}dx'\nonumber\\
&=\sum_{\beta=1}^{n}\sum_{i=1}^{n-1}a_{\beta}^{i}\partial_{x_i}(a_{\beta}^{n}){\rm tr}[\texttt{id}]\Omega_3(-\frac{\pi}{8}+\frac{\pi^2}{3})dx'+\sum_{\beta=1}^{n}\sum_{i=1}^{n-1}a_{\beta}^{n}\partial_{x_i}(a_{\beta}^{i}){\rm tr}[\texttt{id}]\Omega_3(-\frac{\pi^2}{6})dx'.\nonumber
\end{align}

\noindent  {\bf case a)~II)}~$r=-1,~l=-1,~k=|\alpha|=0,~j=1$\\

\noindent Using (3.18), we get
\begin{equation}
\Psi_2=-\frac{1}{2}\int_{|\xi'|=1}\int^{+\infty}_{-\infty} {\rm
trace} [\partial_{x_n}\pi^+_{\xi_n}\sigma_{-1}({{D}_{J}}^{-1})\times
\partial_{\xi_n}^2\sigma_{-1}({{D}_{J}}^{-1})](x_0)d\xi_n\sigma(\xi')dx'.
\end{equation}
It is easy to check that
\begin{align}
\partial_{x_n}\left(\frac{ic[J(\xi)]}{|\xi|^2}\right)(x_0)
&=\frac{i\partial_{x_n}(c[J(\xi)])(x_0)}{|\xi|^2}-\frac{ic[J(\xi)]\partial_{x_n}(|\xi|^2)(x_0)}{|\xi|^4}\\
&=\frac{i\sum^{n}_{p=1}\xi_{p}\partial_{x_n}(c(\sum^{n}_{h=1}a^{p}_{h}dx_{h}))(x_0)}{|\xi|^2}-\frac{ih'(0)|\xi'|^2\sum^{n}_{p=1}\xi_{p}c(\sum^{n}_{h=1}a^{p}_{h}dx_{h})}{|\xi|^4}\nonumber\\
&=\frac{i\sum^{n}_{p,h=1}\xi_{p}\partial_{x_n}(a^{p}_{h})c(dx_{h})(x_0)}{|\xi|^2}+\frac{i\sum^{n}_{p=1}\sum^{n-1}_{h=1}\xi_{p}a^{p}_{h}\partial_{x_n}(c(dx_{h}))(x_0)}{|\xi|^2}\nonumber\\
&-\frac{ih'(0)|\xi'|^2\sum^{n}_{p,h=1}\xi_{p}a^{p}_{h}c(dx_{h})}{|\xi|^4}.\nonumber
\end{align}
Then
\begin{align}
&\pi^+_{\xi_n}\partial_{x_n}\left(\frac{ic[J(\xi)]}{|\xi|^2}\right)(x_0)|_{|\xi'|=1}\\
&=i\sum^{n}_{h=1}\sum^{n-1}_{p=1}\xi_{p}\partial_{x_n}(a^{p}_{h})c(dx_{h})\pi^+_{\xi_n}\left(\frac{1}{1+\xi_n^2}\right)+i\sum^{n}_{h=1}\partial_{x_n}(a^{n}_{h})c(dx_{h})\pi^+_{\xi_n}\left(\frac{\xi_{n}}{1+\xi_n^2}\right)\nonumber\\
&+i\sum^{n-1}_{p,h=1}\xi_{p}a^{p}_{h}\partial_{x_n}(c(dx_{h}))\pi^+_{\xi_n}\left(\frac{1}{1+\xi_n^2}\right)+i\sum^{n-1}_{h=1}a^{n}_{h}\partial_{x_n}(c(dx_{h}))\pi^+_{\xi_n}\left(\frac{\xi_{n}}{1+\xi_n^2}\right)\nonumber\\
&-ih'(0)\sum^{n}_{h=1}\sum^{n-1}_{p=1}\xi_{p}a^{p}_{h}c(dx_{h})\pi^+_{\xi_n}\left(\frac{1}{(1+\xi_n^2)^2}\right)-ih'(0)\sum^{n}_{h=1}a^{n}_{h}c(dx_{h})\pi^+_{\xi_n}\left(\frac{\xi_{n}}{(1+\xi_n^2)^2}\right)\nonumber\\
&=\frac{1}{2(\xi_n-i)}\sum^{n}_{h=1}\sum^{n-1}_{p=1}\xi_{p}\partial_{x_n}(a^{p}_{h})c(dx_{h})+\frac{i}{2(\xi_n-i)}\sum^{n}_{h=1}\partial_{x_n}(a^{n}_{h})c(dx_{h})\nonumber\\
&+\frac{1}{2(\xi_n-i)}\sum^{n-1}_{p,h=1}\xi_{p}a^{p}_{h}\partial_{x_n}(c(dx_{h}))+\frac{i}{2(\xi_n-i)}\sum^{n-1}_{h=1}a^{n}_{h}\partial_{x_n}(c(dx_{h}))\nonumber\\
&+\frac{2i-\xi_n}{4(\xi_n-i)^2}h'(0)\sum^{n}_{h=1}\sum^{n-1}_{p=1}\xi_{p}a^{p}_{h}c(dx_{h})-\frac{1}{4(\xi_n-i)^2}h'(0)\sum^{n}_{h=1}a^{n}_{h}c(dx_{h}),\nonumber
\end{align}
where $\sum_{h=1}^{n-1}\partial_{x_n}(c(dx_h))=\sum_{h=1}^{n-1}\frac{1}{2}h'(0)c(dx_h).$\\
By calculation, we have
\begin{align}
\partial_{\xi_n}\left(\frac{ic[J(\xi)]}{|\xi|^2}\right)(x_0)
=\partial_{\xi_n}\left(\frac{i\sum^{n}_{i=1}\xi_{i}c[J(dx_i)]}{|\xi|^2}\right)(x_0)
=\partial_{\xi_n}\left(\frac{i\sum^{n}_{i,\beta=1}\xi_{i}a_{\beta}^ic(dx_{\beta})}{|\xi|^2}\right)(x_0),
\end{align}
\begin{align}
\partial_{\xi_n}\left(\frac{ic[J(\xi)]}{|\xi|^2}\right)(x_0)|_{|\xi'|=1}
&=i\sum^{n}_{\beta=1}\sum^{n-1}_{i=1}\xi_{i}a_{\beta}^ic(dx_{\beta})\partial_{\xi_n}\left(\frac{1}{1+\xi_{n}^2}\right)+i\sum^{n}_{\beta=1}a_{\beta}^nc(dx_{\beta})\partial_{\xi_n}\left(\frac{\xi_{n}}{1+\xi_{n}^2}\right)\\
&=-\frac{2i\xi_n}{(1+\xi_{n}^2)^2}\sum^{n}_{\beta=1}\sum^{n-1}_{i=1}\xi_{i}a_{\beta}^ic(dx_{\beta})+\frac{i(1-\xi_n^2)}{(1+\xi_{n}^2)^2}\sum^{n}_{\beta=1}a_{\beta}^nc(dx_{\beta})\nonumber
\end{align}
and
\begin{align}
\partial_{\xi_n}^2\left(\frac{ic[J(\xi)]}{|\xi|^2}\right)(x_0)|_{|\xi'|=1}
&=\frac{2i(-1+3\xi_n^2)}{(1+\xi_{n}^2)^3}\sum^{n}_{\beta=1}\sum^{n-1}_{i=1}\xi_{i}a_{\beta}^ic(dx_{\beta})+\frac{2i\xi_n(-3+\xi_n^2)}{(1+\xi_{n}^2)^3}\sum^{n}_{\beta=1}a_{\beta}^nc(dx_{\beta}).
\end{align}
Similar to the formulae (3.42)-(3.48), we have
\begin{align}
&{\rm trace} [\partial_{x_n}\pi^+_{\xi_n}\sigma_{-1}({{D}_{J}}^{-1})\times
\partial_{\xi_n}^2\sigma_{-1}({{D}_{J}}^{-1})](x_0)|_{|\xi'|=1}\\
&=-\frac{i(-1+3\xi_n^2)}{(\xi_n-i)^4(\xi_n+i)^3}\sum_{\beta=1}^{n}\sum_{p,i=1}^{n-1}\xi_{p}\xi_{i}a_{\beta}^{i}\partial_{x_n}(a_{\beta}^{p}){\rm tr}[\texttt{id}]+\frac{-1+3\xi_n^2}{(\xi_n-i)^4(\xi_n+i)^3}\sum_{\beta=1}^{n}\sum_{i=1}^{n-1}\xi_{i}a_{\beta}^{i}\partial_{x_n}(a_{\beta}^{n}){\rm tr}[\texttt{id}]\nonumber\\
&-\frac{i\xi_n(-3+\xi_n^2)}{(\xi_n-i)^4(\xi_n+i)^3}\sum_{\beta=1}^{n}\sum_{p=1}^{n-1}\xi_{p}a_{\beta}^{n}\partial_{x_n}(a_{\beta}^{p}){\rm tr}[\texttt{id}]+\frac{\xi_n(-3+\xi_n^2)}{(\xi_n-i)^4(\xi_n+i)^3}\sum_{\beta=1}^{n}a_{\beta}^{n}\partial_{x_n}(a_{\beta}^{n}){\rm tr}[\texttt{id}]\nonumber\\
&-\frac{i(-1+3\xi_n^2)}{2(\xi_n-i)^4(\xi_n+i)^3}h'(0)\sum_{p,i,\beta=1}^{n-1}\xi_{p}\xi_{i}a_{\beta}^{p}a_{\beta}^{i}{\rm tr}[\texttt{id}]+\frac{-1+3\xi_n^2}{2(\xi_n-i)^4(\xi_n+i)^3}h'(0)\sum_{i,\beta=1}^{n-1}\xi_{i}a_{\beta}^{n}a_{\beta}^{i}{\rm tr}[\texttt{id}]\nonumber\\
&-\frac{i\xi_n(-3+\xi_n^2)}{2(\xi_n-i)^4(\xi_n+i)^3}h'(0)\sum_{p,\beta=1}^{n-1}\xi_{p}a_{\beta}^{p}a_{\beta}^{n}{\rm tr}[\texttt{id}]+\frac{\xi_n(-3+\xi_n^2)}{2(\xi_n-i)^4(\xi_n+i)^3}h'(0)\sum_{\beta=1}^{n-1}(a_{\beta}^{n})^2{\rm tr}[\texttt{id}]\nonumber\\
&-\frac{i(2i-\xi_n)(-1+3\xi_n^2)}{2(\xi_n-i)^5(\xi_n+i)^3}h'(0)\sum_{\beta=1}^{n}\sum_{p,i=1}^{n-1}\xi_{p}\xi_{i}a_{\beta}^{p}a_{\beta}^{i}{\rm tr}[\texttt{id}]+\frac{i(-1+3\xi_n^2)}{2(\xi_n-i)^5(\xi_n+i)^3}h'(0)\sum_{\beta=1}^{n}\sum_{i=1}^{n-1}\xi_{i}a_{\beta}^{n}a_{\beta}^{i}{\rm tr}[\texttt{id}]\nonumber\\
&-\frac{i\xi_n(2i-\xi_n)(-3+\xi_n^2)}{2(\xi_n-i)^5(\xi_n+i)^3}h'(0)\sum_{\beta=1}^{n}\sum_{p=1}^{n-1}\xi_{p}a_{\beta}^{p}a_{\beta}^{n}{\rm tr}[\texttt{id}]+\frac{i\xi_n(-3+\xi_n^2)}{2(\xi_n-i)^5(\xi_n+i)^3}h'(0)\sum_{\beta=1}^{n}(a_{\beta}^{n})^2{\rm tr}[\texttt{id}].\nonumber
\end{align}
Consequently,
\begin{align}
\Psi_2
&=-\frac{1}{2}\int_{|\xi'|=1}\int^{+\infty}_{-\infty} {\rm
trace} [\partial_{x_n}\pi^+_{\xi_n}\sigma_{-1}({{D}_{J}}^{-1})\times
\partial_{\xi_n}^2\sigma_{-1}({{D}_{J}}^{-1})](x_0)d\xi_n\sigma(\xi')dx'\\
&=\int_{|\xi'|=1}\int^{+\infty}_{-\infty}\frac{i(-1+3\xi_n^2)}{2(\xi_n-i)^4(\xi_n+i)^3}\sum_{\beta=1}^{n}\sum_{p,i=1}^{n-1}\xi_{p}\xi_{i}a_{\beta}^{i}\partial_{x_n}(a_{\beta}^{p}){\rm tr}[\texttt{id}]d\xi_n\sigma(\xi')dx'\nonumber
\end{align}
\begin{align}
&+\int_{|\xi'|=1}\int^{+\infty}_{-\infty}-\frac{\xi_n(-3+\xi_n^2)}{2(\xi_n-i)^4(\xi_n+i)^3}\sum_{\beta=1}^{n}a_{\beta}^{n}\partial_{x_n}(a_{\beta}^{n}){\rm tr}[\texttt{id}]d\xi_n\sigma(\xi')dx'\nonumber\\
&+\int_{|\xi'|=1}\int^{+\infty}_{-\infty}\frac{i(-1+3\xi_n^2)}{4(\xi_n-i)^4(\xi_n+i)^3}h'(0)\sum_{p,i,\beta=1}^{n-1}\xi_{p}\xi_{i}a_{\beta}^{p}a_{\beta}^{i}{\rm tr}[\texttt{id}]d\xi_n\sigma(\xi')dx'\nonumber\\
&+\int_{|\xi'|=1}\int^{+\infty}_{-\infty}-\frac{\xi_n(-3+\xi_n^2)}{4(\xi_n-i)^4(\xi_n+i)^3}h'(0)\sum_{\beta=1}^{n-1}(a_{\beta}^{n})^2{\rm tr}[\texttt{id}]d\xi_n\sigma(\xi')dx'\nonumber\\
&+\int_{|\xi'|=1}\int^{+\infty}_{-\infty}\frac{i(2i-\xi_n)(-1+3\xi_n^2)}{4(\xi_n-i)^5(\xi_n+i)^3}h'(0)\sum_{\beta=1}^{n}\sum_{p,i=1}^{n-1}\xi_{p}\xi_{i}a_{\beta}^{p}a_{\beta}^{i}{\rm tr}[\texttt{id}]d\xi_n\sigma(\xi')dx'\nonumber\\
&+\int_{|\xi'|=1}\int^{+\infty}_{-\infty}-\frac{i\xi_n(-3+\xi_n^2)}{4(\xi_n-i)^5(\xi_n+i)^3}h'(0)\sum_{\beta=1}^{n}(a_{\beta}^{n})^2{\rm tr}[\texttt{id}]d\xi_n\sigma(\xi')dx'.\nonumber
\end{align}
Similar calculations to (3.49), it is shown that
\begin{align}
\Psi_2
&=\sum_{\beta=1}^{n}\sum_{i=1}^{n-1}a_{\beta}^{i}\partial_{x_n}(a_{\beta}^{i}){\rm tr}[\texttt{id}]\Omega_3\frac{4\pi}{3}\frac{2\pi i}{3!}\Big[\frac{i(-1+3\xi_n^2)}{2(\xi_n+i)^3}\Big]^{(3)}|_{\xi_n=i}dx'\\
&+\sum_{\beta=1}^{n}a_{\beta}^{n}\partial_{x_n}(a_{\beta}^{n}){\rm tr}[\texttt{id}]\Omega_3\frac{2\pi i}{3!}\Big[-\frac{\xi_n(-3+\xi_n^2)}{2(\xi_n+i)^3}\Big]^{(3)}|_{\xi_n=i}dx'\nonumber\\
&+\sum_{i,\beta=1}^{n-1}(a_{\beta}^{i})^2{\rm tr}[\texttt{id}]\Omega_3h'(0)\frac{4\pi}{3}\frac{2\pi i}{3!}\Big[\frac{i(-1+3\xi_n^2)}{4(\xi_n+i)^3}\Big]^{(3)}|_{\xi_n=i}dx'\nonumber\\
&+\sum_{\beta=1}^{n-1}(a_{\beta}^{n})^2{\rm tr}[\texttt{id}]\Omega_3h'(0)\frac{2\pi i}{3!}\Big[-\frac{\xi_n(-3+\xi_n^2)}{4(\xi_n+i)^3}\Big]^{(3)}|_{\xi_n=i}dx'\nonumber\\
&+\sum_{\beta=1}^{n}\sum_{i=1}^{n-1}(a_{\beta}^{i})^2{\rm tr}[\texttt{id}]\Omega_3h'(0)\frac{4\pi}{3}\frac{2\pi i}{4!}\Big[\frac{i(2i-\xi_n)(-1+3\xi_n^2)}{4(\xi_n+i)^3}\Big]^{(4)}|_{\xi_n=i}dx'\nonumber\\
&+\sum_{\beta=1}^{n}(a_{\beta}^{n})^2{\rm tr}[\texttt{id}]\Omega_3h'(0)\frac{2\pi i}{4!}\Big[-\frac{i\xi_n(-3+\xi_n^2)}{4(\xi_n+i)^3}\Big]^{(4)}|_{\xi_n=i}dx'\nonumber\\
&=\sum_{\beta=1}^{n}\sum_{i=1}^{n-1}a_{\beta}^{i}\partial_{x_n}(a_{\beta}^{i}){\rm tr}[\texttt{id}]\Omega_3(\frac{\pi^2}{12})dx'+\sum_{\beta=1}^{n}a_{\beta}^{n}\partial_{x_n}(a_{\beta}^{n}){\rm tr}[\texttt{id}]\Omega_3(\frac{\pi}{16})dx'\nonumber\\
&+\sum_{i,\beta=1}^{n-1}(a_{\beta}^{i})^2{\rm tr}[\texttt{id}]\Omega_3h'(0)(\frac{\pi^2}{24})dx'+\sum_{\beta=1}^{n-1}(a_{\beta}^{n})^2{\rm tr}[\texttt{id}]\Omega_3h'(0)(\frac{\pi}{32})dx'\nonumber
\end{align}
\begin{align}
&+\sum_{\beta=1}^{n}\sum_{i=1}^{n-1}(a_{\beta}^{i})^2{\rm tr}[\texttt{id}]\Omega_3h'(0)(-\frac{5\pi^2}{48})dx'+\sum_{\beta=1}^{n}(a_{\beta}^{n})^2{\rm tr}[\texttt{id}]\Omega_3h'(0)(-\frac{3\pi}{64})dx'.\nonumber
\end{align}

\noindent  {\bf case a)~III)}~$r=-1,~l=-1,~j=|\alpha|=0,~k=1$\\

\noindent By (3.18), we calculate that
\begin{equation}
\Psi_3=-\frac{1}{2}\int_{|\xi'|=1}\int^{+\infty}_{-\infty}
{\rm trace} [\partial_{\xi_n}\pi^+_{\xi_n}\sigma_{-1}({{D}_{J}}^{-1})\times
\partial_{\xi_n}\partial_{x_n}\sigma_{-1}({{D}_{J}}^{-1})](x_0)d\xi_n\sigma(\xi')dx'.\\
\end{equation}
It is easily seen that
\begin{align}
\pi^+_{\xi_n}\partial_{\xi_n}\left(\frac{ic[J(\xi)]}{|\xi|^2}\right)(x_0)|_{|\xi'|=1}
&=-\frac{1}{2(\xi_{n}-i)^2}\sum^{n}_{\beta=1}\sum^{n-1}_{i=1}\xi_{i}a_{\beta}^ic(dx_{\beta})-\frac{i}{2(\xi_{n}-i)^2}\sum^{n}_{\beta=1}a_{\beta}^nc(dx_{\beta}),
\end{align}
\begin{align}
&\partial_{\xi_n}\partial_{x_n}\left(\frac{ic[J(\xi)]}{|\xi|^2}\right)(x_0)|_{|\xi'|=1}\\
&=-\frac{2i\xi_n}{(1+\xi_n^2)^2}\sum^{n}_{h=1}\sum^{n-1}_{p=1}\xi_{p}\partial_{x_n}(a^{p}_{h})c(dx_{h})+\frac{i(1-\xi_n^2)}{(1+\xi_n^2)^2}\sum^{n}_{h=1}\partial_{x_n}(a^{n}_{h})c(dx_{h})\nonumber\\
&-\frac{2i\xi_n}{(1+\xi_n^2)^2}\sum^{n-1}_{p,h=1}\xi_{p}a^{p}_{h}\partial_{x_n}(c(dx_{h}))+\frac{i(1-\xi_n^2)}{(1+\xi_n^2)^2}\sum^{n-1}_{h=1}a^{n}_{h}\partial_{x_n}(c(dx_{h}))\nonumber\\
&+\frac{4i\xi_n}{(1+\xi_n^2)^3}h'(0)\sum^{n}_{h=1}\sum^{n-1}_{p=1}\xi_{p}a^{p}_{h}c(dx_{h})-\frac{i(1-3\xi_n^2)}{(1+\xi_n^2)^3}h'(0)\sum^{n}_{h=1}a^{n}_{h}c(dx_{h}).\nonumber
\end{align}
We see at once that
\begin{align}
&{\rm trace} [\partial_{\xi_n}\pi^+_{\xi_n}\sigma_{-1}({{D}_{J}}^{-1})\times \partial_{\xi_n}\partial_{x_n}\sigma_{-1}({{D}_{J}}^{-1})](x_0)|_{|\xi'|=1}\\
&=-\frac{i\xi_n}{(\xi_n-i)^4(\xi_n+i)^2}\sum_{\beta=1}^{n}\sum_{i,p=1}^{n-1}\xi_{i}\xi_{p}a_{\beta}^{i}\partial_{x_n}(a_{\beta}^{p}){\rm tr}[\texttt{id}]+\frac{i(1-\xi_n^2)}{2(\xi_n-i)^4(\xi_n+i)^2}\sum_{\beta=1}^{n}\sum_{i=1}^{n-1}\xi_{i}a_{\beta}^{i}\partial_{x_n}(a_{\beta}^{n}){\rm tr}[\texttt{id}]\nonumber\\
&-\frac{i\xi_n}{2(\xi_n-i)^4(\xi_n+i)^2}h'(0)\sum_{i,\beta,p=1}^{n-1}\xi_{i}\xi_{p}a_{\beta}^{i}a_{\beta}^{p}{\rm tr}[\texttt{id}]+\frac{i(1-\xi_n^2)}{4(\xi_n-i)^4(\xi_n+i)^2}h'(0)\sum_{i,\beta=1}^{n-1}\xi_{i}a_{\beta}^{i}a_{\beta}^{n}{\rm tr}[\texttt{id}]\nonumber
\end{align}
\begin{align}
&+\frac{2i\xi_n}{(\xi_n-i)^5(\xi_n+i)^3}h'(0)\sum_{\beta=1}^{n}\sum_{i,p=1}^{n-1}\xi_{i}\xi_{p}a_{\beta}^{i}a_{\beta}^{p}{\rm tr}[\texttt{id}]-\frac{i(1-3\xi_n^2)}{2(\xi_n-i)^5(\xi_n+i)^3}h'(0)\sum_{\beta=1}^{n}\sum_{i=1}^{n-1}\xi_{i}a_{\beta}^{i}a_{\beta}^{n}{\rm tr}[\texttt{id}]\nonumber\\
&+\frac{\xi_n}{(\xi_n-i)^4(\xi_n+i)^2}\sum_{\beta=1}^{n}\sum_{p=1}^{n-1}\xi_{p}a_{\beta}^{n}\partial_{x_n}(a_{\beta}^{p}){\rm tr}[\texttt{id}]-\frac{1-\xi_n^2}{2(\xi_n-i)^4(\xi_n+i)^2}\sum_{\beta=1}^{n}a_{\beta}^{n}\partial_{x_n}(a_{\beta}^{n}){\rm tr}[\texttt{id}]\nonumber\\
&+\frac{\xi_n}{(\xi_n-i)^4(\xi_n+i)^2}h'(0)\sum_{\beta,p=1}^{n-1}\xi_{p}a_{\beta}^{n}a_{\beta}^{p}{\rm tr}[\texttt{id}]-\frac{1-\xi_n^2}{4(\xi_n-i)^4(\xi_n+i)^2}h'(0)\sum_{\beta=1}^{n-1}(a_{\beta}^{n})^2{\rm tr}[\texttt{id}]\nonumber\\
&-\frac{2\xi_n}{(\xi_n-i)^5(\xi_n+i)^3}h'(0)\sum_{\beta=1}^{n}\sum_{p=1}^{n-1}\xi_{p}a_{\beta}^{n}a_{\beta}^{p}{\rm tr}[\texttt{id}]+\frac{1-3\xi_n^2}{2(\xi_n-i)^5(\xi_n+i)^3}h'(0)\sum_{\beta=1}^{n}(a_{\beta}^{n})^2{\rm tr}[\texttt{id}].\nonumber
\end{align}
Hence, we have
\begin{align}
\Psi_3
&=-\frac{1}{2}\int_{|\xi'|=1}\int^{+\infty}_{-\infty}
{\rm trace} [\partial_{\xi_n}\pi^+_{\xi_n}\sigma_{-1}({{D}_{J}}^{-1})\times
\partial_{\xi_n}\partial_{x_n}\sigma_{-1}({{D}_{J}}^{-1})](x_0)d\xi_n\sigma(\xi')dx'\\
&=\int_{|\xi'|=1}\int^{+\infty}_{-\infty}\frac{i\xi_n}{2(\xi_n-i)^4(\xi_n+i)^2}\sum_{\beta=1}^{n}\sum_{i,p=1}^{n-1}\xi_{i}\xi_{p}a_{\beta}^{i}\partial_{x_n}(a_{\beta}^{p}){\rm tr}[\texttt{id}]d\xi_n\sigma(\xi')dx'\nonumber\\
&+\int_{|\xi'|=1}\int^{+\infty}_{-\infty}\frac{i\xi_n}{4(\xi_n-i)^4(\xi_n+i)^2}h'(0)\sum_{i,\beta,p=1}^{n-1}\xi_{i}\xi_{p}a_{\beta}^{i}a_{\beta}^{p}{\rm tr}[\texttt{id}]d\xi_n\sigma(\xi')dx'\nonumber\\
&+\int_{|\xi'|=1}\int^{+\infty}_{-\infty}-\frac{i\xi_n}{(\xi_n-i)^5(\xi_n+i)^3}h'(0)\sum_{\beta=1}^{n}\sum_{i,p=1}^{n-1}\xi_{i}\xi_{p}a_{\beta}^{i}a_{\beta}^{p}{\rm tr}[\texttt{id}]d\xi_n\sigma(\xi')dx'\nonumber\\
&+\int_{|\xi'|=1}\int^{+\infty}_{-\infty}\frac{1-\xi_n^2}{4(\xi_n-i)^4(\xi_n+i)^2}\sum_{\beta=1}^{n}a_{\beta}^{n}\partial_{x_n}(a_{\beta}^{n}){\rm tr}[\texttt{id}]d\xi_n\sigma(\xi')dx'\nonumber\\
&+\int_{|\xi'|=1}\int^{+\infty}_{-\infty}\frac{1-\xi_n^2}{8(\xi_n-i)^4(\xi_n+i)^2}h'(0)\sum_{\beta=1}^{n-1}(a_{\beta}^{n})^2{\rm tr}[\texttt{id}]d\xi_n\sigma(\xi')dx'\nonumber\\
&+\int_{|\xi'|=1}\int^{+\infty}_{-\infty}-\frac{1-3\xi_n^2}{4(\xi_n-i)^5(\xi_n+i)^3}h'(0)\sum_{\beta=1}^{n}(a_{\beta}^{n})^2{\rm tr}[\texttt{id}]d\xi_n\sigma(\xi')dx'.\nonumber
\end{align}
Computations show that
\begin{align}
\Psi_3
&=\sum_{\beta=1}^{n}\sum_{i=1}^{n-1}a_{\beta}^{i}\partial_{x_n}(a_{\beta}^{i}){\rm tr}[\texttt{id}]\Omega_3\frac{4\pi}{3}\frac{2\pi i}{3!}\Big[\frac{i\xi_n}{2(\xi_n+i)^2}\Big]^{(3)}|_{\xi_n=i}dx'\\
&+\sum_{i,\beta=1}^{n-1}(a_{\beta}^{i})^2{\rm tr}[\texttt{id}]\Omega_3h'(0)\frac{4\pi}{3}\frac{2\pi i}{3!}\Big[\frac{i\xi_n}{4(\xi_n+i)^2}\Big]^{(3)}|_{\xi_n=i}dx'\nonumber
\end{align}
\begin{align}
&+\sum_{\beta=1}^{n}\sum_{i=1}^{n-1}(a_{\beta}^{i})^2{\rm tr}[\texttt{id}]\Omega_3h'(0)\frac{4\pi}{3}\frac{2\pi i}{4!}\Big[-\frac{i\xi_n}{(\xi_n+i)^3}\Big]^{(4)}|_{\xi_n=i}dx'\nonumber\\
&+\sum_{\beta=1}^{n}a_{\beta}^{n}\partial_{x_n}(a_{\beta}^{n}){\rm tr}[\texttt{id}]\Omega_3\frac{2\pi i}{3!}\Big[\frac{1-\xi_n^2}{4(\xi_n+i)^2}\Big]^{(3)}|_{\xi_n=i}dx'\nonumber\\
&+\sum_{\beta=1}^{n-1}(a_{\beta}^{n})^2{\rm tr}[\texttt{id}]\Omega_3h'(0)\frac{2\pi i}{3!}\Big[\frac{1-\xi_n^2}{8(\xi_n+i)^2}\Big]^{(3)}|_{\xi_n=i}dx'\nonumber\\
&+\sum_{\beta=1}^{n}(a_{\beta}^{n})^2{\rm tr}[\texttt{id}]\Omega_3h'(0)\frac{2\pi i}{4!}\Big[-\frac{1-3\xi_n^2}{4(\xi_n+i)^3}\Big]^{(4)}|_{\xi_n=i}dx'\nonumber\\
&=\sum_{\beta=1}^{n}\sum_{i=1}^{n-1}a_{\beta}^{i}\partial_{x_n}(a_{\beta}^{i}){\rm tr}[\texttt{id}]\Omega_3(-\frac{\pi^2}{12})dx'\nonumber
+\sum_{i,\beta=1}^{n-1}(a_{\beta}^{i})^2{\rm tr}[\texttt{id}]\Omega_3h'(0)(-\frac{\pi^2}{24})dx'\nonumber\\
&+\sum_{\beta=1}^{n}\sum_{i=1}^{n-1}(a_{\beta}^{i})^2{\rm tr}[\texttt{id}]\Omega_3h'(0)(\frac{5\pi^2}{48})dx'\nonumber
+\sum_{\beta=1}^{n}a_{\beta}^{n}\partial_{x_n}(a_{\beta}^{n}){\rm tr}[\texttt{id}]\Omega_3(-\frac{\pi}{16})dx'\nonumber\\
&+\sum_{\beta=1}^{n-1}(a_{\beta}^{n})^2{\rm tr}[\texttt{id}]\Omega_3h'(0)(-\frac{\pi}{32})dx'\nonumber
+\sum_{\beta=1}^{n}(a_{\beta}^{n})^2{\rm tr}[\texttt{id}]\Omega_3h'(0)(\frac{3\pi}{64})dx'.\nonumber
\end{align}

In combination with the calculation,
\begin{equation}
\Psi_1+\Psi_2+\Psi_3=\sum_{\beta=1}^{n}\sum_{i=1}^{n-1}a_{\beta}^{i}\partial_{x_i}(a_{\beta}^{n}){\rm tr}[\texttt{id}]\Omega_3(-\frac{\pi}{8}+\frac{\pi^2}{3})dx'+\sum_{\beta=1}^{n}\sum_{i=1}^{n-1}a_{\beta}^{n}\partial_{x_i}(a_{\beta}^{i}){\rm tr}[\texttt{id}]\Omega_3(-\frac{\pi^2}{6})dx'.
\end{equation}

\noindent  {\bf case b)}~$r=-2,~l=-1,~k=j=|\alpha|=0$\\

\noindent Similarly, we get
\begin{align}
\Psi_4&=-i\int_{|\xi'|=1}\int^{+\infty}_{-\infty}{\rm trace} [\pi^+_{\xi_n}\sigma_{-2}({{D}_{J}}^{-1})\times
\partial_{\xi_n}\sigma_{-1}({{D}_{J}}^{-1})](x_0)d\xi_n\sigma(\xi')dx'.
\end{align}
We first compute\\
\begin{align}
\sigma_{-2}({{D}_{J}}^{-1})(x_0)&=\frac{c[J(\xi)]\sigma_{0}({D}_{J})(x_0)c[J(\xi)]}{|\xi|^4}+\frac{c[J(\xi)]}{|\xi|^6}\sum_{j=1}^{n} c[J(dx_j)]\Big[\sum_{p,h=1}^{n}\xi_p\partial_{x_j}(a_{h}^{p})c(dx_h)|\xi|^2\\
&+\sum_{p,h=1}^{n}\xi_pa_{h}^{p}\partial_{x_j}(c(dx_h))|\xi|^2-c[J(\xi)]\partial_{x_j}(|\xi|^2)\Big](x_0)\nonumber
\end{align}
\begin{align}
&=\frac{c[J(\xi)]\sigma_{0}({D}_{J})(x_0)c[J(\xi)]}{|\xi|^4}-\frac{c[J(\xi)]}{|\xi|^6}h'(0)|\xi'|^2c[J(dx_n)]c[J(\xi)]\nonumber\\
&+\frac{c[J(\xi)]}{|\xi|^4}\Big[\sum_{j,p,h=1}^{n}\xi_p\partial_{x_j}(a_{h}^{p})c[J(dx_j)]c(dx_h)+\sum_{p=1}^{n}\sum_{h=1}^{n-1}\xi_pa_{h}^{p}c[J(dx_n)]\partial_{x_n}(c(dx_h))\Big]
\nonumber
\end{align}
where
\begin{align}
\sigma_{0}({{D}_{J}})(x_0)
&=-\frac{1}{4}\sum_{i,j,k=1}^{n}\omega_{j,k}(e_i)(x_{0})c[J(e_i)]c(e_j)c(e_k)\\
&=-\frac{1}{4}h'(0)\sum_{i=1}^{n-1}c[J(e_i)]c(e_n)c(e_i)\nonumber\\
&=-\frac{1}{4}h'(0)\sum_{\mu=1}^{n}\sum_{\nu=1}^{n-1}a_{\nu}^{\mu}c(dx_{\mu})c(dx_{n})c(dx_{\nu}).\nonumber
\end{align}
It is obvious that
\begin{align}
&\pi^+_{\xi_n}\sigma_{-2}({{D}_{J}}^{-1})(x_0)|_{|\xi'|=1}=\pi^+_{\xi_n}\Big(\frac{c[J(\xi)]\sigma_{0}({D}_{J})(x_0)c[J(\xi)]}{(1+\xi_n^2)^2}\Big)-h'(0)\pi^+_{\xi_n}\Big(\frac{c[J(\xi)]}{(1+\xi_n^2)^3}c[J(dx_n)]c[J(\xi)]\Big)\\
&+\pi^+_{\xi_n}\Big(\frac{c[J(\xi)]}{(1+\xi_n^2)^2}\Big[\sum_{j,p,h=1}^{n}\xi_p\partial_{x_j}(a_{h}^{p})c[J(dx_j)]c(dx_h)+\sum_{p=1}^{n}\sum_{h=1}^{n-1}\xi_pa_{h}^{p}c[J(dx_n)]\partial_{x_n}(c(dx_h))\Big]\Big).\nonumber
\end{align}
For the sake of convenience in writing, we denote
\begin{align}
A_1(x_0)&=\frac{c[J(\xi)]\sigma_{0}({D}_{J})(x_0)c[J(\xi)]}{(1+\xi_n^2)^2};\\
A_2(x_0)&=\frac{c[J(\xi)]}{(1+\xi_n^2)^2}\Big[\sum_{j,p,h=1}^{n}\xi_p\partial_{x_j}(a_{h}^{p})c[J(dx_j)]c(dx_h)+\sum_{p=1}^{n}\sum_{h=1}^{n-1}\xi_pa_{h}^{p}c[J(dx_n)]\partial_{x_n}(c(dx_h))\Big];\\
A_3(x_0)&=\frac{c[J(\xi)]}{(1+\xi_n^2)^3}c[J(dx_n)]c[J(\xi)],
\end{align}
means that
\begin{align}
&\pi^+_{\xi_n}\sigma_{-2}({{D}_{J}}^{-1})(x_0)|_{|\xi'|=1}=\pi^+_{\xi_n}(A_1(x_0))+\pi^+_{\xi_n}(A_2(x_0))-h'(0)\pi^+_{\xi_n}(A_3(x_0)).
\end{align}
By computation, we have
\begin{align}
\pi^+_{\xi_n}(A_1(x_0))&=\pi^+_{\xi_n}\Big(\frac{\sum_{q,\alpha,l,\gamma=1}^{n}\xi_{q}\xi_{\alpha}a_{l}^{q}a_{\gamma}^{\alpha}c(dx_l)\sigma_{0}({D}_{J})(x_0)c(dx_{\gamma})}{(1+\xi_n^2)^2}\Big)\\
&=-\frac{i\xi_n}{4(\xi_n-i)^2}\sum_{l,\gamma=1}^{n}a_{l}^{n}a_{\gamma}^{n}c(dx_l)\sigma_{0}({D}_{J})(x_0)c(dx_{\gamma})\nonumber\\
&-\frac{i}{4(\xi_n-i)^2}\sum_{l,\gamma=1}^{n}\sum_{q=1}^{n-1}\xi_{q}a_{l}^{q}a_{\gamma}^{n}c(dx_l)\sigma_{0}({D}_{J})(x_0)c(dx_{\gamma})\nonumber\\
&-\frac{i}{4(\xi_n-i)^2}\sum_{l,\gamma=1}^{n}\sum_{\alpha=1}^{n-1}\xi_{\alpha}a_{l}^{n}a_{\gamma}^{\alpha}c(dx_l)\sigma_{0}({D}_{J})(x_0)c(dx_{\gamma})\nonumber\\
&-\frac{i\xi_n+2}{4(\xi_n-i)^2}\sum_{l,\gamma=1}^{n}\sum_{q,\alpha=1}^{n-1}\xi_{q}\xi_{\alpha}a_{l}^{q}a_{\gamma}^{\alpha}c(dx_l)\sigma_{0}({D}_{J})(x_0)c(dx_{\gamma}).\nonumber
\end{align}
Thus, we have
\begin{align}
&{\rm trace} [\pi^+_{\xi_n}(A_1(x_0)) \times \partial_{\xi_n}\sigma_{-1}({{D}_{J}}^{-1})](x_0)|_{|\xi'|=1}\\
&=\frac{\xi_n^2}{8(\xi_n-i)^4(\xi_n+i)^2}h'(0)\sum_{l,\gamma,\mu,\beta=1}^{n}\sum_{\nu,i=1}^{n-1}{\rm tr}[\xi_{i}a_{l}^{n}a_{\gamma}^{n}a_{\nu}^{\mu}a_{\beta}^{i}c(dx_{l})c(dx_{\mu})c(dx_{n})c(dx_{\nu})c(dx_{\gamma})c(dx_{\beta})]\nonumber\\
&-\frac{\xi_n(1-\xi_n^2)}{16(\xi_n-i)^4(\xi_n+i)^2}h'(0)\sum_{l,\gamma,\mu,\beta=1}^{n}\sum_{\nu=1}^{n-1}{\rm tr}[a_{l}^{n}a_{\gamma}^{n}a_{\nu}^{\mu}a_{\beta}^{n}c(dx_{l})c(dx_{\mu})c(dx_{n})c(dx_{\nu})c(dx_{\gamma})c(dx_{\beta})]\nonumber\\
&+\frac{\xi_n}{8(\xi_n-i)^4(\xi_n+i)^2}h'(0)\sum_{l,\gamma,\mu,\beta=1}^{n}\sum_{q,\nu,i=1}^{n-1}{\rm tr}[\xi_{q}\xi_{i}a_{l}^{q}a_{\gamma}^{n}a_{\nu}^{\mu}a_{\beta}^{i}c(dx_{l})c(dx_{\mu})c(dx_{n})c(dx_{\nu})c(dx_{\gamma})c(dx_{\beta})]\nonumber\\
&-\frac{1-\xi_n^2}{16(\xi_n-i)^4(\xi_n+i)^2}h'(0)\sum_{l,\gamma,\mu,\beta=1}^{n}\sum_{q,\nu=1}^{n-1}{\rm tr}[\xi_{q}a_{l}^{q}a_{\gamma}^{n}a_{\nu}^{\mu}a_{\beta}^{n}c(dx_{l})c(dx_{\mu})c(dx_{n})c(dx_{\nu})c(dx_{\gamma})c(dx_{\beta})]\nonumber\\
&+\frac{\xi_n}{8(\xi_n-i)^4(\xi_n+i)^2}h'(0)\sum_{l,\gamma,\mu,\beta=1}^{n}\sum_{\alpha,\nu,i=1}^{n-1}{\rm tr}[\xi_{\alpha}\xi_{i}a_{l}^{n}a_{\gamma}^{\alpha}a_{\nu}^{\mu}a_{\beta}^{i}c(dx_{l})c(dx_{\mu})c(dx_{n})c(dx_{\nu})c(dx_{\gamma})c(dx_{\beta})]\nonumber\\
&-\frac{1-\xi_n^2}{16(\xi_n-i)^4(\xi_n+i)^2}h'(0)\sum_{l,\gamma,\mu,\beta=1}^{n}\sum_{\alpha,\nu=1}^{n-1}{\rm tr}[\xi_{\alpha}a_{l}^{n}a_{\gamma}^{\alpha}a_{\nu}^{\mu}a_{\beta}^{n}c(dx_{l})c(dx_{\mu})c(dx_{n})c(dx_{\nu})c(dx_{\gamma})c(dx_{\beta})]\nonumber
\end{align}
\begin{align}
&-\frac{i\xi_n(i\xi_n+2)}{8(\xi_n-i)^4(\xi_n+i)^2}h'(0)\sum_{l,\gamma,\mu,\beta=1}^{n}\sum_{q,\alpha,\nu,i=1}^{n-1}{\rm tr}[\xi_{q}\xi_{\alpha}\xi_{i}a_{l}^{q}a_{\gamma}^{\alpha}a_{\nu}^{\mu}a_{\beta}^{i}c(dx_{l})c(dx_{\mu})c(dx_{n})c(dx_{\nu})c(dx_{\gamma})c(dx_{\beta})]\nonumber\\
&+\frac{i(i\xi_n+2)(1-\xi_n^2)}{16(\xi_n-i)^4(\xi_n+i)^2}h'(0)\sum_{l,\gamma,\mu,\beta=1}^{n}\sum_{q,\alpha,\nu=1}^{n-1}{\rm tr}[\xi_{q}\xi_{\alpha}a_{l}^{q}a_{\gamma}^{\alpha}a_{\nu}^{\mu}a_{\beta}^{n}c(dx_{l})c(dx_{\mu})c(dx_{n})c(dx_{\nu})c(dx_{\gamma})c(dx_{\beta})].\nonumber
\end{align}
By the relation of the Clifford action and ${\rm tr}{AB}={\rm tr}{BA}$,  we have the equality:\\
\begin{align}
&\sum_{l,\gamma,\mu,\beta=1}^{n}\sum_{\nu=1}^{n-1}{\rm tr}[c(dx_{l})c(dx_{\mu})c(dx_{n})c(dx_{\nu})c(dx_{\gamma})c(dx_{\beta})]\\
&=\sum_{l=1}^{n}\sum_{\nu=1}^{n-1}\delta_{\beta}^{l}\delta_{\nu}^{\mu}\delta_{\gamma}^{n}{\rm tr}[\texttt{id}]-\sum_{l=1}^{n}\sum_{\nu=1}^{n-1}\delta_{\beta}^{l}\delta_{n}^{\mu}\delta_{\gamma}^{\nu}{\rm tr}[\texttt{id}]-\sum_{l=1}^{n}\sum_{\nu=1}^{n-1}\delta_{\gamma}^{l}\delta_{\nu}^{\mu}\delta_{\beta}^{n}{\rm tr}[\texttt{id}]\nonumber\\
&+\sum_{l=1}^{n}\sum_{\nu=1}^{n-1}\delta_{\gamma}^{l}\delta_{n}^{\mu}\delta_{\beta}^{\nu}{\rm tr}[\texttt{id}]-\sum_{\mu=1}^{n}\sum_{\nu=1}^{n-1}\delta_{\nu}^{l}\delta_{\beta}^{\mu}\delta_{\gamma}^{n}{\rm tr}[\texttt{id}]+\sum_{\mu=1}^{n}\sum_{\nu=1}^{n-1}\delta_{\nu}^{l}\delta_{\gamma}^{\mu}\delta_{\beta}^{n}{\rm tr}[\texttt{id}]\nonumber\\
&-\sum_{\gamma=1}^{n}\sum_{\nu=1}^{n-1}\delta_{\nu}^{l}\delta_{n}^{\mu}\delta_{\beta}^{\gamma}{\rm tr}[\texttt{id}]+\sum_{\mu=1}^{n}\sum_{\nu=1}^{n-1}\delta_{n}^{l}\delta_{\beta}^{\mu}\delta_{\gamma}^{\nu}{\rm tr}[\texttt{id}]-\sum_{\mu=1}^{n}\sum_{\nu=1}^{n-1}\delta_{n}^{l}\delta_{\gamma}^{\mu}\delta_{\beta}^{\nu}{\rm tr}[\texttt{id}]\nonumber\\
&+\sum_{\gamma=1}^{n}\sum_{\nu=1}^{n-1}\delta_{n}^{l}\delta_{\nu}^{\mu}\delta_{\beta}^{\gamma}{\rm tr}[\texttt{id}]-\sum_{l=1}^{n}\sum_{\nu=1}^{n-1}\delta_{\mu}^{l}\delta_{\beta}^{n}\delta_{\gamma}^{\nu}{\rm tr}[\texttt{id}]+\sum_{l=1}^{n}\sum_{\nu=1}^{n-1}\delta_{\mu}^{l}\delta_{\gamma}^{n}\delta_{\beta}^{\nu}{\rm tr}[\texttt{id}].\nonumber
\end{align}
By $\int_{|\xi'|=1}{\{\xi_{i_1}\cdot\cdot\cdot\xi_{i_{2d+1}}}\}\sigma(\xi')=0,$  $\int_{|\xi'|=1}\xi_i\xi_j=\frac{4\pi}{3}\delta_i^j,$ then we have
\begin{align}
&-i\int_{|\xi'|=1}\int^{+\infty}_{-\infty}{\rm trace} [\pi^+_{\xi_n}(A_1(x_0)) \times \partial_{\xi_n}\sigma_{-1}({{D}_{J}}^{-1})](x_0)d\xi_n\sigma(\xi')dx'\\
&=\Omega_3\int_{\Gamma^{+}}\frac{i\xi_n(1-\xi_n^2)}{16(\xi_n-i)^4(\xi_n+i)^2}h'(0)\sum_{l=1}^{n}\sum_{\nu=1}^{n-1}(-(a_{\nu}^{n})^2(a_{l}^{n})^2+(a_{l}^{n})^2a_{\nu}^{\nu}a_{n}^{n}){\rm tr}[\texttt{id}]d\xi_ndx'\nonumber\\
&+\frac{8\pi}{3}\Omega_3\int_{\Gamma^{+}}\frac{-i\xi_n}{8(\xi_n-i)^4(\xi_n+i)^2}h'(0)\sum_{l=1}^{n}\sum_{\nu,i=1}^{n-1}(-(a_{\nu}^{n})^2(a_{l}^{i})^2+(a_{l}^{i})^2a_{\nu}^{\nu}a_{n}^{n}){\rm tr}[\texttt{id}]d\xi_ndx'\nonumber\\
&+\frac{4\pi}{3}\Omega_3\int_{\Gamma^{+}}\frac{(i\xi_n+2)(1-\xi_n^2)}{16(\xi_n-i)^4(\xi_n+i)^2}h'(0)\sum_{l=1}^{n}\sum_{\nu,i=1}^{n-1}((a_{\nu}^{n})^2(a_{l}^{i})^2-(a_{l}^{i})^2a_{\nu}^{\nu}a_{n}^{n}-2a_{\nu}^{i}a_{l}^{i}a_{\nu}^{n}a_{l}^{n}+2a_{i}^{i}a_{\nu}^{\nu})\nonumber\\
&{\rm tr}[\texttt{id}]d\xi_ndx'.\nonumber
\end{align}
It follows immediately that
\begin{align}
&-i\int_{|\xi'|=1}\int^{+\infty}_{-\infty}{\rm trace} [\pi^+_{\xi_n}(A_1(x_0)) \times \partial_{\xi_n}\sigma_{-1}({{D}_{J}}^{-1})](x_0)d\xi_n\sigma(\xi')dx'\\
&=\sum_{l=1}^{n}\sum_{\nu=1}^{n-1}(-(a_{\nu}^{n})^2(a_{l}^{n})^2+(a_{l}^{n})^2a_{\nu}^{\nu}a_{n}^{n}){\rm tr}[\texttt{id}]\Omega_3h'(0)\frac{2\pi i}{3!}\Big[\frac{i\xi_n(1-\xi_n^2)}{16(\xi_n+i)^2}\Big]^{(3)}|_{\xi_n=i}dx'\nonumber\\
&+\sum_{l=1}^{n}\sum_{\nu,i=1}^{n-1}(-(a_{\nu}^{n})^2(a_{l}^{i})^2+(a_{l}^{i})^2a_{\nu}^{\nu}a_{n}^{n}){\rm tr}[\texttt{id}]\Omega_3h'(0)\frac{8\pi}{3}\frac{2\pi i}{3!}\Big[\frac{-i\xi_n}{8(\xi_n+i)^2}\Big]^{(3)}|_{\xi_n=i}dx'\nonumber\\
&+\sum_{l=1}^{n}\sum_{\nu,i=1}^{n-1}((a_{\nu}^{n})^2(a_{l}^{i})^2-(a_{l}^{i})^2a_{\nu}^{\nu}a_{n}^{n}-2a_{\nu}^{i}a_{l}^{i}a_{\nu}^{n}a_{l}^{n}+2a_{i}^{i}a_{\nu}^{\nu}){\rm tr}[\texttt{id}]\Omega_3h'(0)\frac{4\pi}{3}\frac{2\pi i}{3!}\nonumber\\
&\Big[\frac{(i\xi_n+2)(1-\xi_n^2)}{16(\xi_n+i)^2}\Big]^{(3)}|_{\xi_n=i}dx'\nonumber\\
&=\sum_{l=1}^{n}\sum_{\nu,i=1}^{n-1}(-2(a_{\nu}^{n})^2(a_{l}^{i})^2+2(a_{l}^{i})^2a_{\nu}^{\nu}a_{n}^{n}+2a_{\nu}^{i}a_{l}^{i}a_{\nu}^{n}a_{l}^{n}-2a_{i}^{i}a_{\nu}^{\nu}){\rm tr}[\texttt{id}]\Omega_3h'(0)(\frac{\pi^2}{24})dx'.\nonumber
\end{align}
Therefore
\begin{align}
\pi^+_{\xi_n}(A_2(x_0))&=\pi^+_{\xi_n}\Big(\frac{\sum_{q,l,j,p,h,y=1}^{n}\xi_{q}\xi_{p}a_{l}^{q}a_{y}^{j}\partial_{x_j}(a_{h}^{p})c(dx_{l})c(dx_{y})c(dx_{h})}{(1+\xi_n^2)^2}\Big)\\
&+\pi^+_{\xi_n}\Big(\frac{\frac{1}{2}h'(0)\sum_{q,l,p,z=1}^{n}\sum_{h=1}^{n-1}\xi_{q}\xi_{p}a_{l}^{q}a_{h}^{p}a_{z}^{n}c(dx_{l})c(dx_{z})c(dx_{h})}{(1+\xi_n^2)^2}\Big)\nonumber\\
&=-\frac{i\xi_n}{4(\xi_n-i)^2}\sum_{l,j,h,y=1}^{n}a_{l}^{n}a_{y}^{j}\partial_{x_j}(a_{h}^{n})c(dx_{l})c(dx_{y})c(dx_{h})\nonumber\\
&-\frac{i}{4(\xi_n-i)^2}\sum_{l,j,h,y=1}^{n}\sum_{q=1}^{n-1}\xi_{q}a_{l}^{q}a_{y}^{j}\partial_{x_j}(a_{h}^{n})c(dx_{l})c(dx_{y})c(dx_{h})\nonumber\\
&-\frac{i}{4(\xi_n-i)^2}\sum_{l,j,h,y=1}^{n}\sum_{p=1}^{n-1}\xi_{p}a_{l}^{n}a_{y}^{j}\partial_{x_j}(a_{h}^{p})c(dx_{l})c(dx_{y})c(dx_{h})\nonumber\\
&-\frac{i\xi_n+2}{4(\xi_n-i)^2}\sum_{l,j,h,y=1}^{n}\sum_{q,p=1}^{n-1}\xi_{q}\xi_{p}a_{l}^{q}a_{y}^{j}\partial_{x_j}(a_{h}^{p})c(dx_{l})c(dx_{y})c(dx_{h})\nonumber\\
&-\frac{i\xi_n}{8(\xi_n-i)^2}h'(0)\sum_{l,z=1}^{n}\sum_{h=1}^{n-1}a_{l}^{n}a_{h}^{n}a_{z}^{n}c(dx_{l})c(dx_{z})c(dx_{h})\nonumber
\end{align}
\begin{align}
&-\frac{i}{8(\xi_n-i)^2}h'(0)\sum_{l,z=1}^{n}\sum_{q,h=1}^{n-1}\xi_{q}a_{l}^{q}a_{h}^{n}a_{z}^{n}c(dx_{l})c(dx_{z})c(dx_{h})\nonumber\\
&-\frac{i}{8(\xi_n-i)^2}h'(0)\sum_{l,z=1}^{n}\sum_{p,h=1}^{n-1}\xi_{p}a_{l}^{n}a_{h}^{p}a_{z}^{n}c(dx_{l})c(dx_{z})c(dx_{h})\nonumber\\
&-\frac{i\xi_n+2}{8(\xi_n-i)^2}h'(0)\sum_{l,z=1}^{n}\sum_{q,p,h=1}^{n-1}\xi_{q}\xi_{p}a_{l}^{q}a_{h}^{p}a_{z}^{n}c(dx_{l})c(dx_{z})c(dx_{h}).\nonumber
\end{align}
Clearly,
\begin{align}
&{\rm trace} [\pi^+_{\xi_n}(A_2(x_0)) \times \partial_{\xi_n}\sigma_{-1}({{D}_{J}}^{-1})](x_0)|_{|\xi'|=1}\\
&=-\frac{\xi_n^2}{2(\xi_n-i)^4(\xi_n+i)^2}\sum_{l,j,h,y,\beta=1}^{n}\sum_{i=1}^{n-1}{\rm tr}[\xi_{i}a_{l}^{n}a_{y}^{j}a_{\beta}^{i}\partial_{x_j}(a_{h}^{n})c(dx_{l})c(dx_{y})c(dx_{h})c(dx_{\beta})]\nonumber\\
&+\frac{\xi_n(1-\xi_n^2)}{4(\xi_n-i)^4(\xi_n+i)^2}\sum_{l,j,h,y,\beta=1}^{n}{\rm tr}[a_{l}^{n}a_{y}^{j}a_{\beta}^{n}\partial_{x_j}(a_{h}^{n})c(dx_{l})c(dx_{y})c(dx_{h})c(dx_{\beta})]\nonumber\\
&-\frac{\xi_n}{2(\xi_n-i)^4(\xi_n+i)^2}\sum_{l,j,h,y,\beta=1}^{n}\sum_{q,i=1}^{n-1}{\rm tr}[\xi_{q}\xi_{i}a_{l}^{q}a_{y}^{j}a_{\beta}^{i}\partial_{x_j}(a_{h}^{n})c(dx_{l})c(dx_{y})c(dx_{h})c(dx_{\beta})]\nonumber\\
&+\frac{1-\xi_n^2}{4(\xi_n-i)^4(\xi_n+i)^2}\sum_{l,j,h,y,\beta=1}^{n}\sum_{q=1}^{n-1}{\rm tr}[\xi_{q}a_{l}^{q}a_{y}^{j}a_{\beta}^{n}\partial_{x_j}(a_{h}^{n})c(dx_{l})c(dx_{y})c(dx_{h})c(dx_{\beta})]\nonumber\\
&-\frac{\xi_n}{2(\xi_n-i)^4(\xi_n+i)^2}\sum_{l,j,h,y,\beta=1}^{n}\sum_{p,i=1}^{n-1}{\rm tr}[\xi_{p}\xi_{i}a_{l}^{n}a_{y}^{j}a_{\beta}^{i}\partial_{x_j}(a_{h}^{p})c(dx_{l})c(dx_{y})c(dx_{h})c(dx_{\beta})]\nonumber\\
&+\frac{1-\xi_n^2}{4(\xi_n-i)^4(\xi_n+i)^2}\sum_{l,j,h,y,\beta=1}^{n}\sum_{p=1}^{n-1}{\rm tr}[\xi_{p}a_{l}^{n}a_{y}^{j}a_{\beta}^{n}\partial_{x_j}(a_{h}^{p})c(dx_{l})c(dx_{y})c(dx_{h})c(dx_{\beta})]\nonumber\\
&+\frac{i\xi_n(i\xi_n+2)}{2(\xi_n-i)^4(\xi_n+i)^2}\sum_{l,j,h,y,\beta=1}^{n}\sum_{q,p,i=1}^{n-1}{\rm tr}[\xi_{q}\xi_{p}\xi_{i}a_{l}^{q}a_{y}^{j}a_{\beta}^{i}\partial_{x_j}(a_{h}^{p})c(dx_{l})c(dx_{y})c(dx_{h})c(dx_{\beta})]\nonumber\\
&-\frac{i(i\xi_n+2)(1-\xi_n^2)}{4(\xi_n-i)^4(\xi_n+i)^2}\sum_{l,j,h,y,\beta=1}^{n}\sum_{q,p=1}^{n-1}{\rm tr}[\xi_{q}\xi_{p}a_{l}^{q}a_{y}^{j}a_{\beta}^{n}\partial_{x_j}(a_{h}^{p})c(dx_{l})c(dx_{y})c(dx_{h})c(dx_{\beta})]\nonumber
\end{align}
\begin{align}
&-\frac{\xi_n^2}{4(\xi_n-i)^4(\xi_n+i)^2}h'(0)\sum_{l,z,\beta=1}^{n}\sum_{h,i=1}^{n-1}{\rm tr}[\xi_{i}a_{l}^{n}a_{h}^{n}a_{z}^{n}a_{\beta}^{i}c(dx_{l})c(dx_{z})c(dx_{h})c(dx_{\beta})]\nonumber\\
&+\frac{\xi_n(1-\xi_n^2)}{8(\xi_n-i)^4(\xi_n+i)^2}h'(0)\sum_{l,z,\beta=1}^{n}\sum_{h=1}^{n-1}{\rm tr}[a_{l}^{n}a_{h}^{n}a_{z}^{n}a_{\beta}^{n}c(dx_{l})c(dx_{z})c(dx_{h})c(dx_{\beta})]\nonumber\\
&-\frac{\xi_n}{4(\xi_n-i)^4(\xi_n+i)^2}h'(0)\sum_{l,z,\beta=1}^{n}\sum_{q,h,i=1}^{n-1}{\rm tr}[\xi_{q}\xi_{i}a_{l}^{q}a_{h}^{n}a_{z}^{n}a_{\beta}^{i}c(dx_{l})c(dx_{z})c(dx_{h})c(dx_{\beta})]\nonumber\\
&+\frac{1-\xi_n^2}{8(\xi_n-i)^4(\xi_n+i)^2}h'(0)\sum_{l,z,\beta=1}^{n}\sum_{q,h=1}^{n-1}{\rm tr}[\xi_{q}a_{l}^{q}a_{h}^{n}a_{z}^{n}a_{\beta}^{n}c(dx_{l})c(dx_{z})c(dx_{h})c(dx_{\beta})]\nonumber\\
&-\frac{\xi_n}{4(\xi_n-i)^4(\xi_n+i)^2}h'(0)\sum_{l,z,\beta=1}^{n}\sum_{p,h,i=1}^{n-1}{\rm tr}[\xi_{q}\xi_{i}a_{l}^{n}a_{h}^{p}a_{z}^{n}a_{\beta}^{i}c(dx_{l})c(dx_{z})c(dx_{h})c(dx_{\beta})]\nonumber\\
&+\frac{1-\xi_n^2}{8(\xi_n-i)^4(\xi_n+i)^2}h'(0)\sum_{l,z,\beta=1}^{n}\sum_{p,h=1}^{n-1}{\rm tr}[\xi_{p}a_{l}^{n}a_{h}^{p}a_{z}^{n}a_{\beta}^{n}c(dx_{l})c(dx_{z})c(dx_{h})c(dx_{\beta})]\nonumber\\
&+\frac{i\xi_n(i\xi_n+2)}{4(\xi_n-i)^4(\xi_n+i)^2}h'(0)\sum_{l,z,\beta=1}^{n}\sum_{q,p,h,i=1}^{n-1}{\rm tr}[\xi_{q}\xi_{p}\xi_{i}a_{l}^{q}a_{h}^{p}a_{z}^{n}a_{\beta}^{i}c(dx_{l})c(dx_{z})c(dx_{h})c(dx_{\beta})]\nonumber\\
&-\frac{i(i\xi_n+2)(1-\xi_n^2)}{8(\xi_n-i)^4(\xi_n+i)^2}h'(0)\sum_{l,z,\beta=1}^{n}\sum_{q,p,h=1}^{n-1}{\rm tr}[\xi_{q}\xi_{p}a_{l}^{q}a_{h}^{p}a_{z}^{n}a_{\beta}^{n}c(dx_{l})c(dx_{z})c(dx_{h})c(dx_{\beta})].\nonumber
\end{align}
Similarly, we have
\begin{align}
&\sum_{l,y,h,\beta=1}^{n}{\rm tr}[c(dx_{l})c(dx_{y})c(dx_{h})c(dx_{\beta})]=\sum_{\beta,y=1}^{n}\delta_{\beta}^{l}\delta_{h}^{y}{\rm tr}[\texttt{id}]-\sum_{\beta,l=1}^{n}\delta_{h}^{l}\delta_{\beta}^{y}{\rm tr}[\texttt{id}]+\sum_{\beta,l=1}^{n}\delta_{y}^{l}\delta_{\beta}^{h}{\rm tr}[\texttt{id}],
\end{align}
\begin{align}
&\sum_{l,z,\beta=1}^{n}\sum_{h=1}^{n-1}{\rm tr}[c(dx_{l})c(dx_{z})c(dx_{h})c(dx_{\beta})]=\sum_{\beta=1}^{n}\sum_{z=1}^{n-1}\delta_{\beta}^{l}\delta_{h}^{z}{\rm tr}[\texttt{id}]-\sum_{\beta=1}^{n}\sum_{l=1}^{n-1}\delta_{h}^{l}\delta_{\beta}^{z}{\rm tr}[\texttt{id}]+\sum_{l=1}^{n}\sum_{\beta=1}^{n-1}\delta_{z}^{l}\delta_{\beta}^{h}{\rm tr}[\texttt{id}].
\end{align}
Then
\begin{align}
&-i\int_{|\xi'|=1}\int^{+\infty}_{-\infty}{\rm trace} [\pi^+_{\xi_n}(A_2(x_0)) \times \partial_{\xi_n}\sigma_{-1}({{D}_{J}}^{-1})](x_0)d\xi_n\sigma(\xi')dx'
\end{align}
\begin{align}
&=\Omega_3\int_{\Gamma^{+}}\frac{-i\xi_n(1-\xi_n^2)}{4(\xi_n-i)^4(\xi_n+i)^2}\sum_{l,j,\beta=1}^{n}\left((a_{\beta}^{n})^2a_{l}^{j}\partial_{x_j}(a_{l}^{n})-a_{l}^{n}a_{\beta}^{j}a_{\beta}^{n}\partial_{x_j}(a_{l}^{n})+a_{l}^{n}a_{l}^{j}a_{\beta}^{n}\partial_{x_j}(a_{\beta}^{n})\right)\nonumber\\
&{\rm tr}[\texttt{id}]d\xi_ndx'\nonumber\\
&+\frac{4\pi}{3}\Omega_3\int_{\Gamma^{+}}\frac{i\xi_n}{2(\xi_n-i)^4(\xi_n+i)^2}\sum_{l,j,\beta=1}^{n}\sum_{i=1}^{n-1}\left((a_{\beta}^{i})^2a_{l}^{j}\partial_{x_j}(a_{l}^{n})-a_{l}^{i}a_{\beta}^{j}a_{\beta}^{i}\partial_{x_j}(a_{l}^{n})+a_{l}^{i}a_{l}^{j}a_{\beta}^{i}\partial_{x_j}(a_{\beta}^{n})\right)\nonumber\\
&{\rm tr}[\texttt{id}]d\xi_ndx'\nonumber\\
&+\frac{4\pi}{3}\Omega_3\int_{\Gamma^{+}}\frac{i\xi_n}{2(\xi_n-i)^4(\xi_n+i)^2}\sum_{l,j,\beta=1}^{n}\sum_{i=1}^{n-1}\left(a_{\beta}^{n}a_{l}^{j}a_{\beta}^{i}\partial_{x_j}(a_{l}^{i})-a_{l}^{n}a_{\beta}^{j}a_{\beta}^{i}\partial_{x_j}(a_{l}^{i})+a_{l}^{n}a_{l}^{j}a_{\beta}^{i}\partial_{x_j}(a_{\beta}^{i})\right)\nonumber\\
&{\rm tr}[\texttt{id}]d\xi_ndx'\nonumber\\
&+\frac{4\pi}{3}\Omega_3\int_{\Gamma^{+}}\frac{-(i\xi_n+2)(1-\xi_n^2)}{4(\xi_n-i)^4(\xi_n+i)^2}\sum_{l,j,\beta=1}^{n}\sum_{i=1}^{n-1}\left(a_{\beta}^{i}a_{l}^{j}a_{\beta}^{n}\partial_{x_j}(a_{l}^{i})-a_{l}^{i}a_{\beta}^{j}a_{\beta}^{n}\partial_{x_j}(a_{l}^{i})+a_{l}^{i}a_{l}^{j}a_{\beta}^{n}\partial_{x_j}(a_{\beta}^{i})\right)\nonumber\\
&{\rm tr}[\texttt{id}]d\xi_ndx'\nonumber\\
&+\Omega_3\int_{\Gamma^{+}}\frac{-i\xi_n(1-\xi_n^2)}{8(\xi_n-i)^4(\xi_n+i)^2}h'(0)\sum_{l=1}^{n}\sum_{\nu=1}^{n-1}(a_{\nu}^{n})^2(a_{l}^{n})^2{\rm tr}[\texttt{id}]d\xi_ndx'\nonumber\\
&+\frac{4\pi}{3}\Omega_3\int_{\Gamma^{+}}\frac{i\xi_n}{4(\xi_n-i)^4(\xi_n+i)^2}h'(0)\sum_{l=1}^{n}\sum_{\nu,i=1}^{n-1}(a_{\nu}^{n})^2(a_{l}^{i})^2{\rm tr}[\texttt{id}]d\xi_ndx'\nonumber\\
&+\frac{4\pi}{3}\Omega_3\int_{\Gamma^{+}}\frac{i\xi_n}{4(\xi_n-i)^4(\xi_n+i)^2}h'(0)\sum_{l=1}^{n}\sum_{\nu,i=1}^{n-1}(a_{\nu}^{i})^2(a_{l}^{n})^2{\rm tr}[\texttt{id}]d\xi_ndx'\nonumber\\
&+\frac{4\pi}{3}\Omega_3\int_{\Gamma^{+}}\frac{-(i\xi_n+2)(1-\xi_n^2)}{8(\xi_n-i)^4(\xi_n+i)^2}h'(0)\sum_{l=1}^{n}\sum_{\nu,i=1}^{n-1}\left(2a_{\nu}^{i}a_{l}^{i}a_{\nu}^{n}a_{l}^{n}-(a_{\nu}^{i})^2(a_{l}^{n})^2\right){\rm tr}[\texttt{id}]d\xi_ndx'.\nonumber
\end{align}
A simple calculation shows that
\begin{align}
&-i\int_{|\xi'|=1}\int^{+\infty}_{-\infty}{\rm trace} [\pi^+_{\xi_n}(A_2(x_0)) \times \partial_{\xi_n}\sigma_{-1}({{D}_{J}}^{-1})](x_0)d\xi_n\sigma(\xi')dx'\\
&=\sum_{l,j,\beta=1}^{n}\sum_{i=1}^{n-1}\left((a_{\beta}^{i})^2a_{l}^{j}\partial_{x_j}(a_{l}^{n})-a_{l}^{i}a_{\beta}^{j}a_{\beta}^{i}\partial_{x_j}(a_{l}^{n})+a_{l}^{i}a_{l}^{j}a_{\beta}^{i}\partial_{x_j}(a_{\beta}^{n})\right){\rm tr}[\texttt{id}]\Omega_3(-\frac{\pi^2}{12})dx'\nonumber\\
&+\sum_{l,j,\beta=1}^{n}\sum_{i=1}^{n-1}\left(a_{\beta}^{n}a_{l}^{j}a_{\beta}^{i}\partial_{x_j}(a_{l}^{i})-a_{l}^{n}a_{\beta}^{j}a_{\beta}^{i}\partial_{x_j}(a_{l}^{i})+a_{l}^{n}a_{l}^{j}a_{\beta}^{i}\partial_{x_j}(a_{\beta}^{i})\right){\rm tr}[\texttt{id}]\Omega_3(-\frac{\pi^2}{12})dx'\nonumber\\
&+\sum_{l,j,\beta=1}^{n}\sum_{i=1}^{n-1}\left(a_{\beta}^{i}a_{l}^{j}a_{\beta}^{n}\partial_{x_j}(a_{l}^{i})-a_{l}^{i}a_{\beta}^{j}a_{\beta}^{n}\partial_{x_j}(a_{l}^{i})+a_{l}^{i}a_{l}^{j}a_{\beta}^{n}\partial_{x_j}(a_{\beta}^{i})\right){\rm tr}[\texttt{id}]\Omega_3(\frac{\pi^2}{6})dx'\nonumber
\end{align}
\begin{align}
&+\sum_{l=1}^{n}\sum_{\nu,i=1}^{n-1}(a_{\nu}^{n})^2(a_{l}^{i})^2{\rm tr}[\texttt{id}]\Omega_3h'(0)(-\frac{\pi^2}{24})dx'+\sum_{l=1}^{n}\sum_{\nu,i=1}^{n-1}(a_{\nu}^{i})^2(a_{l}^{n})^2{\rm tr}[\texttt{id}]\Omega_3h'(0)(-\frac{\pi^2}{24})dx'\nonumber\\
&+\sum_{l=1}^{n}\sum_{\nu,i=1}^{n-1}\left(2a_{\nu}^{i}a_{l}^{i}a_{\nu}^{n}a_{l}^{n}-(a_{\nu}^{i})^2(a_{l}^{n})^2\right){\rm tr}[\texttt{id}]\Omega_3h'(0)(\frac{\pi^2}{12})dx'.\nonumber
\end{align}
Since
\begin{align}
-h'(0)\pi^+_{\xi_n}(A_3(x_0))&=-h'(0)\pi^+_{\xi_n}\Big(\frac{\sum_{q,l,w,\alpha,\gamma=1}^{n}\xi_{q}\xi_{\alpha}a_{l}^{q}a_{w}^{n}a_{\gamma}^{\alpha}c(dx_{l})c(dx_{w})c(dx_{\gamma})}{(1+\xi_n^2)^3}\Big)\\
&=\frac{3\xi_n+i\xi_n^2}{16(\xi_n-i)^3}h'(0)\sum_{l,w,\gamma=1}^{n}a_{l}^{n}a_{w}^{n}a_{\gamma}^{n}c(dx_{l})c(dx_{w})c(dx_{\gamma})\nonumber\\
&+\frac{i\xi_n+3}{16(\xi_n-i)^3}h'(0)\sum_{l,w,\gamma=1}^{n}\sum_{q=1}^{n-1}\xi_{q}a_{l}^{q}a_{w}^{n}a_{\gamma}^{n}c(dx_{l})c(dx_{w})c(dx_{\gamma})\nonumber\\
&+\frac{i\xi_n+3}{16(\xi_n-i)^3}h'(0)\sum_{l,w,\gamma=1}^{n}\sum_{\alpha=1}^{n-1}\xi_{\alpha}a_{l}^{n}a_{w}^{n}a_{\gamma}^{\alpha}c(dx_{l})c(dx_{w})c(dx_{\gamma})\nonumber\\
&+\frac{-8i+9\xi_n+3i\xi_n^2}{16(\xi_n-i)^3}h'(0)\sum_{l,w,\gamma=1}^{n}\sum_{q,\alpha=1}^{n-1}\xi_{q}\xi_{\alpha}a_{l}^{q}a_{w}^{n}a_{\gamma}^{\alpha}c(dx_{l})c(dx_{w})c(dx_{\gamma}),\nonumber
\end{align}
\begin{align}
&{\rm trace} [-h'(0)\pi^+_{\xi_n}(A_3(x_0)) \times \partial_{\xi_n}\sigma_{-1}({{D}_{J}}^{-1})](x_0)|_{|\xi'|=1}\\
&=-\frac{i\xi_n(3\xi_n+i\xi_n^2)}{8(\xi_n-i)^5(\xi_n+i)^2}h'(0)\sum_{l,w,\gamma,\beta=1}^{n}\sum_{i=1}^{n-1}{\rm tr}[\xi_{i}a_{l}^{n}a_{w}^{n}a_{\gamma}^{n}a_{\beta}^{i}c(dx_{l})c(dx_{w})c(dx_{\gamma})c(dx_{\beta})]\nonumber\\
&+\frac{i(3\xi_n+i\xi_n^2)(1-\xi_n^2)}{16(\xi_n-i)^5(\xi_n+i)^2}h'(0)\sum_{l,w,\gamma,\beta=1}^{n}{\rm tr}[a_{l}^{n}a_{w}^{n}a_{\gamma}^{n}a_{\beta}^{n}c(dx_{l})c(dx_{w})c(dx_{\gamma})c(dx_{\beta})]\nonumber\\
&-\frac{i\xi_n(i\xi_n+3)}{8(\xi_n-i)^5(\xi_n+i)^2}h'(0)\sum_{l,w,\gamma,\beta=1}^{n}\sum_{q,i=1}^{n-1}{\rm tr}[\xi_{q}\xi_{i}a_{l}^{q}a_{w}^{n}a_{\gamma}^{n}a_{\beta}^{i}c(dx_{l})c(dx_{w})c(dx_{\gamma})c(dx_{\beta})]\nonumber\\
&+\frac{i(i\xi_n+3)(1-\xi_n^2)}{16(\xi_n-i)^5(\xi_n+i)^2}h'(0)\sum_{l,w,\gamma,\beta=1}^{n}\sum_{q=1}^{n-1}{\rm tr}[\xi_{q}a_{l}^{q}a_{w}^{n}a_{\gamma}^{n}a_{\beta}^{n}c(dx_{l})c(dx_{w})c(dx_{\gamma})c(dx_{\beta})]\nonumber\\
&-\frac{i\xi_n(i\xi_n+3)}{8(\xi_n-i)^5(\xi_n+i)^2}h'(0)\sum_{l,w,\gamma,\beta=1}^{n}\sum_{\alpha,i=1}^{n-1}{\rm tr}[\xi_{\alpha}\xi_{i}a_{l}^{n}a_{w}^{n}a_{\gamma}^{\alpha}a_{\beta}^{i}c(dx_{l})c(dx_{w})c(dx_{\gamma})c(dx_{\beta})]\nonumber
\end{align}
\begin{align}
&+\frac{i(i\xi_n+3)(1-\xi_n^2)}{16(\xi_n-i)^5(\xi_n+i)^2}h'(0)\sum_{l,w,\gamma,\beta=1}^{n}\sum_{\alpha=1}^{n-1}{\rm tr}[\xi_{\alpha}a_{l}^{n}a_{w}^{n}a_{\gamma}^{\alpha}a_{\beta}^{n}c(dx_{l})c(dx_{w})c(dx_{\gamma})c(dx_{\beta})]\nonumber\\
&-\frac{i\xi_n(-8i+9\xi_n+3i\xi_n^2)}{8(\xi_n-i)^5(\xi_n+i)^2}h'(0)\sum_{l,w,\gamma,\beta=1}^{n}\sum_{q,\alpha,i=1}^{n-1}{\rm tr}[\xi_{q}\xi_{\alpha}\xi_{i}a_{l}^{q}a_{w}^{n}a_{\gamma}^{\alpha}a_{\beta}^{i}c(dx_{l})c(dx_{w})c(dx_{\gamma})c(dx_{\beta})]\nonumber\\
&+\frac{i(1-\xi_n^2)(-8i+9\xi_n+3i\xi_n^2)}{16(\xi_n-i)^5(\xi_n+i)^2}h'(0)\sum_{l,w,\gamma,\beta=1}^{n}\sum_{q,\alpha=1}^{n-1}{\rm tr}[\xi_{q}\xi_{\alpha}a_{l}^{q}a_{w}^{n}a_{\gamma}^{\alpha}a_{\beta}^{n}c(dx_{l})c(dx_{w})c(dx_{\gamma})c(dx_{\beta})],\nonumber
\end{align}
then, we have
\begin{align}
&-i\int_{|\xi'|=1}\int^{+\infty}_{-\infty}{\rm trace} [-h'(0)\pi^+_{\xi_n}(A_3(x_0)) \times \partial_{\xi_n}\sigma_{-1}({{D}_{J}}^{-1})](x_0)d\xi_n\sigma(\xi')dx'\\
&=\Omega_3\int_{\Gamma^{+}}\frac{(3\xi_n+i\xi_n^2)(1-\xi_n^2)}{16(\xi_n-i)^5(\xi_n+i)^2}h'(0)\sum_{\beta,l=1}^{n}(a_{\beta}^{n})^2(a_{l}^{n})^2{\rm tr}[\texttt{id}]d\xi_ndx'\nonumber\\
&+\frac{8\pi}{3}\Omega_3\int_{\Gamma^{+}}\frac{-\xi_n(i\xi_n+3)}{8(\xi_n-i)^5(\xi_n+i)^2}h'(0)\sum_{\beta,l=1}^{n}\sum_{i=1}^{n-1}(a_{\beta}^{i})^2(a_{l}^{n})^2{\rm tr}[\texttt{id}]d\xi_ndx'\nonumber\\
&+\frac{4\pi}{3}\Omega_3\int_{\Gamma^{+}}\frac{(-8i+9\xi_n+3i\xi_n^2)(1-\xi_n^2)}{16(\xi_n-i)^5(\xi_n+i)^2}h'(0)\sum_{\beta,l=1}^{n}\sum_{i=1}^{n-1}\left(2a_{l}^{i}a_{\beta}^{i}a_{l}^{n}a_{\beta}^{n}-(a_{l}^{i})^2(a_{\beta}^{n})^2\right){\rm tr}[\texttt{id}]d\xi_ndx'\nonumber\\
&=\sum_{\beta,l=1}^{n}(a_{\beta}^{n})^2(a_{l}^{n})^2{\rm tr}[\texttt{id}]\Omega_3h'(0)(-\frac{\pi}{128})dx'+\sum_{\beta,l=1}^{n}\sum_{i=1}^{n-1}(a_{\beta}^{i})^2(a_{l}^{n})^2{\rm tr}[\texttt{id}]\Omega_3h'(0)(\frac{5\pi^2}{48})dx'\nonumber\\
&+\sum_{\beta,l=1}^{n}\sum_{i=1}^{n-1}\left(2a_{l}^{i}a_{\beta}^{i}a_{l}^{n}a_{\beta}^{n}-(a_{l}^{i})^2(a_{\beta}^{n})^2\right){\rm tr}[\texttt{id}]\Omega_3h'(0)(-\frac{5\pi^2}{32})dx'.\nonumber
\end{align}
Hence, we have
\begin{align}
\Psi_4&=\sum_{l=1}^{n}\sum_{\nu,i=1}^{n-1}(-2(a_{\nu}^{n})^2(a_{l}^{i})^2+2(a_{l}^{i})^2a_{\nu}^{\nu}a_{n}^{n}+2a_{\nu}^{i}a_{l}^{i}a_{\nu}^{n}a_{l}^{n}-2a_{i}^{i}a_{\nu}^{\nu}){\rm tr}[\texttt{id}]\Omega_3h'(0)(\frac{\pi^2}{24})dx'\\
&+\sum_{l,j,\beta=1}^{n}\sum_{i=1}^{n-1}\left((a_{\beta}^{i})^2a_{l}^{j}\partial_{x_j}(a_{l}^{n})-a_{l}^{i}a_{\beta}^{j}a_{\beta}^{i}\partial_{x_j}(a_{l}^{n})+a_{l}^{i}a_{l}^{j}a_{\beta}^{i}\partial_{x_j}(a_{\beta}^{n})\right){\rm tr}[\texttt{id}]\Omega_3(-\frac{\pi^2}{12})dx'\nonumber
\end{align}
\begin{align}
&+\sum_{l,j,\beta=1}^{n}\sum_{i=1}^{n-1}\left(a_{\beta}^{n}a_{l}^{j}a_{\beta}^{i}\partial_{x_j}(a_{l}^{i})-a_{l}^{n}a_{\beta}^{j}a_{\beta}^{i}\partial_{x_j}(a_{l}^{i})+a_{l}^{n}a_{l}^{j}a_{\beta}^{i}\partial_{x_j}(a_{\beta}^{i})\right){\rm tr}[\texttt{id}]\Omega_3(-\frac{\pi^2}{12})dx'\nonumber\\
&+\sum_{l,j,\beta=1}^{n}\sum_{i=1}^{n-1}\left(a_{\beta}^{i}a_{l}^{j}a_{\beta}^{n}\partial_{x_j}(a_{l}^{i})-a_{l}^{i}a_{\beta}^{j}a_{\beta}^{n}\partial_{x_j}(a_{l}^{i})+a_{l}^{i}a_{l}^{j}a_{\beta}^{n}\partial_{x_j}(a_{\beta}^{i})\right){\rm tr}[\texttt{id}]\Omega_3(\frac{\pi^2}{6})dx'\nonumber\\
&+\sum_{l=1}^{n}\sum_{\nu,i=1}^{n-1}(a_{\nu}^{n})^2(a_{l}^{i})^2{\rm tr}[\texttt{id}]\Omega_3h'(0)(-\frac{\pi^2}{24})dx'+\sum_{l=1}^{n}\sum_{\nu,i=1}^{n-1}(a_{\nu}^{i})^2(a_{l}^{n})^2{\rm tr}[\texttt{id}]\Omega_3h'(0)(-\frac{\pi^2}{24})dx'\nonumber\\
&+\sum_{l=1}^{n}\sum_{\nu,i=1}^{n-1}\left(2a_{\nu}^{i}a_{l}^{i}a_{\nu}^{n}a_{l}^{n}-(a_{\nu}^{i})^2(a_{l}^{n})^2\right){\rm tr}[\texttt{id}]\Omega_3h'(0)(\frac{\pi^2}{12})dx'\nonumber\\
&+\sum_{\beta,l=1}^{n}(a_{\beta}^{n})^2(a_{l}^{n})^2{\rm tr}[\texttt{id}]\Omega_3h'(0)(-\frac{\pi}{128})dx'+\sum_{\beta,l=1}^{n}\sum_{i=1}^{n-1}(a_{\beta}^{i})^2(a_{l}^{n})^2{\rm tr}[\texttt{id}]\Omega_3h'(0)(\frac{5\pi^2}{48})dx'\nonumber\\
&+\sum_{\beta,l=1}^{n}\sum_{i=1}^{n-1}\left(2a_{l}^{i}a_{\beta}^{i}a_{l}^{n}a_{\beta}^{n}-(a_{l}^{i})^2(a_{\beta}^{n})^2\right){\rm tr}[\texttt{id}]\Omega_3h'(0)(-\frac{5\pi^2}{32})dx'.\nonumber
\end{align}

\noindent {\bf  case c)}~$r=-1,~l=-2,~k=j=|\alpha|=0$\\

\noindent We calculate
\begin{align}
\Psi_5=-i\int_{|\xi'|=1}\int^{+\infty}_{-\infty}{\rm trace} [\pi^+_{\xi_n}\sigma_{-1}({{D}_{J}}^{-1})\times
\partial_{\xi_n}\sigma_{-2}({{D}_{J}}^{-1})](x_0)d\xi_n\sigma(\xi')dx'.
\end{align}
It is evident that
\begin{align}
\pi^+_{\xi_n}\left(\frac{ic[J(\xi)]}{|\xi|^2}\right)(x_0)
=\pi^+_{\xi_n}\left(\frac{i\sum^{n}_{i,\beta=1}\xi_{i}a_{\beta}^{i}c(dx_{\beta})}{|\xi|^2}\right)(x_0),
\end{align}
\begin{align}
\pi^+_{\xi_n}\left(\frac{ic[J(\xi)]}{|\xi|^2}\right)(x_0)|_{|\xi'|=1}
=\frac{1}{2(\xi_{n}-i)}\sum^{n}_{\beta=1}\sum^{n-1}_{i=1}\xi_{i}a_{\beta}^{i}c(dx_{\beta})+\frac{i}{2(\xi_{n}-i)}\sum^{n}_{\beta=1}a_{\beta}^{n}c(dx_{\beta}).
\end{align}
Since
\begin{align}
&\partial_{\xi_n}\sigma_{-2}({{D}_{J}}^{-1})(x_0)|_{|\xi'|=1}=\partial_{\xi_n}\Big(\frac{c[J(\xi)]\sigma_{0}({D}_{J})(x_0)c[J(\xi)]}{(1+\xi_n^2)^2}\Big)-h'(0)\partial_{\xi_n}\Big(\frac{c[J(\xi)]}{(1+\xi_n^2)^3}c[J(dx_n)]c[J(\xi)]\Big)\\
&+\partial_{\xi_n}\Big(\frac{c[J(\xi)]}{(1+\xi_n^2)^2}\Big[\sum_{j,p,h=1}^{n}\xi_p\partial_{x_j}(a_{h}^{p})c[J(dx_j)]c(dx_h)+\sum_{p=1}^{n}\sum_{h=1}^{n-1}\xi_pa_{h}^{p}c[J(dx_n)]\partial_{x_n}(c(dx_h))\Big]\Big),\nonumber
\end{align}
then
\begin{align}
&\partial_{\xi_n}\sigma_{-2}({{D}_{J}}^{-1})(x_0)|_{|\xi'|=1}=\partial_{\xi_n}(A_1(x_0))+\partial_{\xi_n}(A_2(x_0))-h'(0)\partial_{\xi_n}(A_3(x_0)).
\end{align}
By computation, we have
\begin{align}
\partial_{\xi_n}(A_1(x_0))&=\partial_{\xi_n}\Big(\frac{\sum_{q,\alpha,l,\gamma=1}^{n}\xi_{q}\xi_{\alpha}a_{l}^{q}a_{\gamma}^{\alpha}c(dx_l)\sigma_{0}({D}_{J})(x_0)c(dx_{\gamma})}{(1+\xi_n^2)^2}\Big)\\
&=-\frac{2\xi_n(-1+\xi_n^2)}{(1+\xi_n^2)^3}\sum_{l,\gamma=1}^{n}a_{l}^{n}a_{\gamma}^{n}c(dx_l)\sigma_{0}({D}_{J})(x_0)c(dx_{\gamma})\nonumber\\
&+\frac{1-3\xi_n^2}{(1+\xi_n^2)^3}\sum_{l,\gamma=1}^{n}\sum_{q=1}^{n-1}\xi_{q}a_{l}^{q}a_{\gamma}^{n}c(dx_l)\sigma_{0}({D}_{J})(x_0)c(dx_{\gamma})\nonumber\\
&+\frac{1-3\xi_n^2}{(1+\xi_n^2)^3}\sum_{l,\gamma=1}^{n}\sum_{\alpha=1}^{n-1}\xi_{\alpha}a_{l}^{n}a_{\gamma}^{\alpha}c(dx_l)\sigma_{0}({D}_{J})(x_0)c(dx_{\gamma})\nonumber\\
&-\frac{4\xi_n}{(1+\xi_n^2)^3}\sum_{l,\gamma=1}^{n}\sum_{q,\alpha=1}^{n-1}\xi_{q}\xi_{\alpha}a_{l}^{q}a_{\gamma}^{\alpha}c(dx_l)\sigma_{0}({D}_{J})(x_0)c(dx_{\gamma}).\nonumber
\end{align}
By (3.92) and (3.95), we have
\begin{align}
&{\rm trace} [\pi^+_{\xi_n}\sigma_{-1}({{D}_{J}}^{-1}) \times \partial_{\xi_n}(A_1(x_0))](x_0)|_{|\xi'|=1}\\
&=\frac{\xi_n}{2(\xi_n-i)^4(\xi_n+i)^3}h'(0)\sum_{\beta,l,\gamma,\mu=1}^{n}\sum_{i,q,\alpha,\nu=1}^{n-1}{\rm tr}[\xi_{i}\xi_{q}\xi_{\alpha}a_{\beta}^{i}a_{l}^{q}a_{\gamma}^{\alpha}a_{\nu}^{\mu}c(dx_{\beta})c(dx_{l})c(dx_{\mu})c(dx_{n})c(dx_{\nu})c(dx_{\gamma})]\nonumber\\
&+\frac{i\xi_n}{2(\xi_n-i)^4(\xi_n+i)^3}h'(0)\sum_{\beta,l,\gamma,\mu=1}^{n}\sum_{q,\alpha,\nu=1}^{n-1}{\rm tr}[\xi_{q}\xi_{\alpha}a_{\beta}^{n}a_{l}^{q}a_{\gamma}^{\alpha}a_{\nu}^{\mu}c(dx_{\beta})c(dx_{l})c(dx_{\mu})c(dx_{n})c(dx_{\nu})c(dx_{\gamma})]\nonumber\\
&-\frac{1-3\xi_n^2}{8(\xi_n-i)^4(\xi_n+i)^3}h'(0)\sum_{\beta,l,\gamma,\mu=1}^{n}\sum_{i,\alpha,\nu=1}^{n-1}{\rm tr}[\xi_{i}\xi_{\alpha}a_{\beta}^{i}a_{l}^{n}a_{\gamma}^{\alpha}a_{\nu}^{\mu}c(dx_{\beta})c(dx_{l})c(dx_{\mu})c(dx_{n})c(dx_{\nu})c(dx_{\gamma})]\nonumber\\
&-\frac{i(1-3\xi_n^2)}{8(\xi_n-i)^4(\xi_n+i)^3}h'(0)\sum_{\beta,l,\gamma,\mu=1}^{n}\sum_{\alpha,\nu=1}^{n-1}{\rm tr}[\xi_{\alpha}a_{\beta}^{n}a_{l}^{n}a_{\gamma}^{\alpha}a_{\nu}^{\mu}c(dx_{\beta})c(dx_{l})c(dx_{\mu})c(dx_{n})c(dx_{\nu})c(dx_{\gamma})]\nonumber\\
&-\frac{1-3\xi_n^2}{8(\xi_n-i)^4(\xi_n+i)^3}h'(0)\sum_{\beta,l,\gamma,\mu=1}^{n}\sum_{i,q,\nu=1}^{n-1}{\rm tr}[\xi_{i}\xi_{q}a_{\beta}^{i}a_{l}^{q}a_{\gamma}^{n}a_{\nu}^{\mu}c(dx_{\beta})c(dx_{l})c(dx_{\mu})c(dx_{n})c(dx_{\nu})c(dx_{\gamma})]\nonumber\\
&-\frac{i(1-3\xi_n^2)}{8(\xi_n-i)^4(\xi_n+i)^3}h'(0)\sum_{\beta,l,\gamma,\mu=1}^{n}\sum_{q,\nu=1}^{n-1}{\rm tr}[\xi_{q}a_{\beta}^{n}a_{l}^{q}a_{\gamma}^{n}a_{\nu}^{\mu}c(dx_{\beta})c(dx_{l})c(dx_{\mu})c(dx_{n})c(dx_{\nu})c(dx_{\gamma})]\nonumber
\end{align}
\begin{align}
&+\frac{\xi_n(-1+\xi_n^2)}{4(\xi_n-i)^4(\xi_n+i)^3}h'(0)\sum_{\beta,l,\gamma,\mu=1}^{n}\sum_{i,\nu=1}^{n-1}{\rm tr}[\xi_{i}a_{\beta}^{i}a_{l}^{n}a_{\gamma}^{n}a_{\nu}^{\mu}c(dx_{\beta})c(dx_{l})c(dx_{\mu})c(dx_{n})c(dx_{\nu})c(dx_{\gamma})]\nonumber\\
&+\frac{i\xi_n(-1+\xi_n^2)}{4(\xi_n-i)^4(\xi_n+i)^3}h'(0)\sum_{\beta,l,\gamma,\mu=1}^{n}\sum_{\nu=1}^{n-1}{\rm tr}[a_{\beta}^{n}a_{l}^{n}a_{\gamma}^{n}a_{\nu}^{\mu}c(dx_{\beta})c(dx_{l})c(dx_{\mu})c(dx_{n})c(dx_{\nu})c(dx_{\gamma})],\nonumber
\end{align}
it is shown that
\begin{align}
&-i\int_{|\xi'|=1}\int^{+\infty}_{-\infty}{\rm trace} [\pi^+_{\xi_n}\sigma_{-1}({{D}_{J}}^{-1}) \times \partial_{\xi_n}(A_1(x_0))](x_0)d\xi_n\sigma(\xi')dx'\\
&=\Omega_3\int_{\Gamma^{+}}\frac{\xi_n(-1+\xi_n^2)}{4(\xi_n-i)^4(\xi_n+i)^3}h'(0)\sum_{l=1}^{n}\sum_{\nu=1}^{n-1}(-(a_{\nu}^{n})^2(a_{l}^{n})^2+(a_{l}^{n})^2a_{\nu}^{\nu}a_{n}^{n}){\rm tr}[\texttt{id}]d\xi_ndx'\nonumber\\
&+\frac{8\pi}{3}\Omega_3\int_{\Gamma^{+}}\frac{i(1-3\xi_n^2)}{8(\xi_n-i)^4(\xi_n+i)^3}h'(0)\sum_{l=1}^{n}\sum_{\nu,i=1}^{n-1}(-(a_{\nu}^{n})^2(a_{l}^{i})^2+(a_{l}^{i})^2a_{\nu}^{\nu}a_{n}^{n}){\rm tr}[\texttt{id}]d\xi_ndx'\nonumber\\
&+\frac{4\pi}{3}\Omega_3\int_{\Gamma^{+}}\frac{\xi_n}{2(\xi_n-i)^4(\xi_n+i)^3}h'(0)\sum_{l=1}^{n}\sum_{\nu,i=1}^{n-1}((a_{\nu}^{n})^2(a_{l}^{i})^2-(a_{l}^{i})^2a_{\nu}^{\nu}a_{n}^{n}-2a_{\nu}^{i}a_{l}^{i}a_{\nu}^{n}a_{l}^{n}+2a_{i}^{i}a_{\nu}^{\nu})\nonumber\\
&{\rm tr}[\texttt{id}]d\xi_ndx'\nonumber\\
&=\sum_{l=1}^{n}\sum_{\nu,i=1}^{n-1}(-2(a_{\nu}^{n})^2(a_{l}^{i})^2+2(a_{l}^{i})^2a_{\nu}^{\nu}a_{n}^{n}+2a_{\nu}^{i}a_{l}^{i}a_{\nu}^{n}a_{l}^{n}-2a_{i}^{i}a_{\nu}^{\nu}){\rm tr}[\texttt{id}]\Omega_3h'(0)(-\frac{\pi^2}{24})dx'.\nonumber
\end{align}
Similarly to (3.95) and (3.96), we have
\begin{align}
\partial_{\xi_n}(A_2(x_0))&=\partial_{\xi_n}\Big(\frac{\sum_{q,l,j,p,h,y=1}^{n}\xi_{q}\xi_{p}a_{l}^{q}a_{y}^{j}\partial_{x_j}(a_{h}^{p})c(dx_{l})c(dx_{y})c(dx_{h})}{(1+\xi_n^2)^2}\Big)\\
&+\partial_{\xi_n}\Big(\frac{\frac{1}{2}h'(0)\sum_{q,l,p,z=1}^{n}\sum_{h=1}^{n-1}\xi_{q}\xi_{p}a_{l}^{q}a_{h}^{p}a_{z}^{n}c(dx_{l})c(dx_{z})c(dx_{h})}{(1+\xi_n^2)^2}\Big)\nonumber\\
&=-\frac{2\xi_n(-1+\xi_n^2)}{(1+\xi_n^2)^3}\sum_{l,j,h,y=1}^{n}a_{l}^{n}a_{y}^{j}\partial_{x_j}(a_{h}^{n})c(dx_{l})c(dx_{y})c(dx_{h})\nonumber\\
&+\frac{1-3\xi_n^2}{(1+\xi_n^2)^3}\sum_{l,j,h,y=1}^{n}\sum_{q=1}^{n-1}\xi_{q}a_{l}^{q}a_{y}^{j}\partial_{x_j}(a_{h}^{n})c(dx_{l})c(dx_{y})c(dx_{h})\nonumber\\
&+\frac{1-3\xi_n^2}{(1+\xi_n^2)^3}\sum_{l,j,h,y=1}^{n}\sum_{p=1}^{n-1}\xi_{p}a_{l}^{n}a_{y}^{j}\partial_{x_j}(a_{h}^{p})c(dx_{l})c(dx_{y})c(dx_{h})\nonumber
\end{align}
\begin{align}
&-\frac{4\xi_n}{(1+\xi_n^2)^3}\sum_{l,j,h,y=1}^{n}\sum_{q,p=1}^{n-1}\xi_{q}\xi_{p}a_{l}^{q}a_{y}^{j}\partial_{x_j}(a_{h}^{p})c(dx_{l})c(dx_{y})c(dx_{h})\nonumber\\
&-\frac{\xi_n(-1+\xi_n^2)}{(1+\xi_n^2)^3}h'(0)\sum_{l,z=1}^{n}\sum_{h=1}^{n-1}a_{l}^{n}a_{h}^{n}a_{z}^{n}c(dx_{l})c(dx_{z})c(dx_{h})\nonumber\\
&+\frac{1-3\xi_n^2}{2(1+\xi_n^2)^3}h'(0)\sum_{l,z=1}^{n}\sum_{q,h=1}^{n-1}\xi_{q}a_{l}^{q}a_{h}^{n}a_{z}^{n}c(dx_{l})c(dx_{z})c(dx_{h})\nonumber\\
&+\frac{1-3\xi_n^2}{2(1+\xi_n^2)^3}h'(0)\sum_{l,z=1}^{n}\sum_{p,h=1}^{n-1}\xi_{p}a_{l}^{n}a_{h}^{p}a_{z}^{n}c(dx_{l})c(dx_{z})c(dx_{h})\nonumber\\
&-\frac{2\xi_n}{(1+\xi_n^2)^3}h'(0)\sum_{l,z=1}^{n}\sum_{q,p,h=1}^{n-1}\xi_{q}\xi_{p}a_{l}^{q}a_{h}^{p}a_{z}^{n}c(dx_{l})c(dx_{z})c(dx_{h}),\nonumber
\end{align}
\begin{align}
&{\rm trace} [\pi^+_{\xi_n}\sigma_{-1}({{D}_{J}}^{-1}) \times \partial_{\xi_n}(A_2(x_0))](x_0)|_{|\xi'|=1}\\
&=-\frac{\xi_n(-1+\xi_n^2)}{(\xi_n-i)^4(\xi_n+i)^3}\sum_{\beta,l,j,h,y=1}^{n}\sum_{i=1}^{n-1}{\rm tr}[\xi_{i}a_{\beta}^{i}a_{l}^{n}a_{y}^{j}\partial_{x_j}(a_{h}^{n})c(dx_{\beta})c(dx_{l})c(dx_{y})c(dx_{h})]\nonumber\\
&-\frac{i\xi_n(-1+\xi_n^2)}{(\xi_n-i)^4(\xi_n+i)^3}\sum_{\beta,l,j,h,y=1}^{n}{\rm tr}[a_{\beta}^{n}a_{l}^{n}a_{y}^{j}\partial_{x_j}(a_{h}^{n})c(dx_{\beta})c(dx_{l})c(dx_{y})c(dx_{h})]\nonumber\\
&+\frac{1-3\xi_n^2}{2(\xi_n-i)^4(\xi_n+i)^3}\sum_{\beta,l,j,h,y=1}^{n}\sum_{i,q=1}^{n-1}{\rm tr}[\xi_{i}\xi_{q}a_{\beta}^{i}a_{l}^{q}a_{y}^{j}\partial_{x_j}(a_{h}^{n})c(dx_{\beta})c(dx_{l})c(dx_{y})c(dx_{h})]\nonumber\\
&+\frac{i(1-3\xi_n^2)}{2(\xi_n-i)^4(\xi_n+i)^3}\sum_{\beta,l,j,h,y=1}^{n}\sum_{q=1}^{n-1}{\rm tr}[\xi_{q}a_{\beta}^{n}a_{l}^{q}a_{y}^{j}\partial_{x_j}(a_{h}^{n})c(dx_{\beta})c(dx_{l})c(dx_{y})c(dx_{h})]\nonumber\\
&+\frac{1-3\xi_n^2}{2(\xi_n-i)^4(\xi_n+i)^3}\sum_{\beta,l,j,h,y=1}^{n}\sum_{i,p=1}^{n-1}{\rm tr}[\xi_{i}\xi_{p}a_{\beta}^{i}a_{l}^{n}a_{y}^{j}\partial_{x_j}(a_{h}^{p})c(dx_{\beta})c(dx_{l})c(dx_{y})c(dx_{h})]\nonumber\\
&+\frac{i(1-3\xi_n^2)}{2(\xi_n-i)^4(\xi_n+i)^3}\sum_{\beta,l,j,h,y=1}^{n}\sum_{p=1}^{n-1}{\rm tr}[\xi_{p}a_{\beta}^{n}a_{l}^{n}a_{y}^{j}\partial_{x_j}(a_{h}^{p})c(dx_{\beta})c(dx_{l})c(dx_{y})c(dx_{h})]\nonumber\\
&-\frac{2\xi_n}{(\xi_n-i)^4(\xi_n+i)^3}\sum_{\beta,l,j,h,y=1}^{n}\sum_{i,q,p=1}^{n-1}{\rm tr}[\xi_{i}\xi_{q}\xi_{p}a_{\beta}^{i}a_{l}^{q}a_{y}^{j}\partial_{x_j}(a_{h}^{p})c(dx_{\beta})c(dx_{l})c(dx_{y})c(dx_{h})]\nonumber
\end{align}
\begin{align}
&-\frac{2i\xi_n}{(\xi_n-i)^4(\xi_n+i)^3}\sum_{\beta,l,j,h,y=1}^{n}\sum_{q,p=1}^{n-1}{\rm tr}[\xi_{q}\xi_{p}a_{\beta}^{n}a_{l}^{q}a_{y}^{j}\partial_{x_j}(a_{h}^{p})c(dx_{\beta})c(dx_{l})c(dx_{y})c(dx_{h})]\nonumber\\
&-\frac{\xi_n(-1+\xi_n^2)}{2(\xi_n-i)^4(\xi_n+i)^3}h'(0)\sum_{\beta,l,z=1}^{n}\sum_{i,h=1}^{n-1}{\rm tr}[\xi_{i}a_{\beta}^{i}a_{l}^{n}a_{h}^{n}a_{z}^{n}c(dx_{\beta})c(dx_{l})c(dx_{z})c(dx_{h})]\nonumber\\
&-\frac{i\xi_n(-1+\xi_n^2)}{2(\xi_n-i)^4(\xi_n+i)^3}h'(0)\sum_{\beta,l,z=1}^{n}\sum_{h=1}^{n-1}{\rm tr}[a_{\beta}^{n}a_{l}^{n}a_{h}^{n}a_{z}^{n}c(dx_{\beta})c(dx_{l})c(dx_{z})c(dx_{h})]\nonumber\\
&+\frac{1-3\xi_n^2}{4(\xi_n-i)^4(\xi_n+i)^3}h'(0)\sum_{\beta,l,z=1}^{n}\sum_{i,q,h=1}^{n-1}{\rm tr}[\xi_{i}\xi_{q}a_{\beta}^{i}a_{l}^{q}a_{h}^{n}a_{z}^{n}c(dx_{\beta})c(dx_{l})c(dx_{z})c(dx_{h})]\nonumber\\
&+\frac{i(1-3\xi_n^2)}{4(\xi_n-i)^4(\xi_n+i)^3}h'(0)\sum_{\beta,l,z=1}^{n}\sum_{q,h=1}^{n-1}{\rm tr}[\xi_{q}a_{\beta}^{n}a_{l}^{q}a_{h}^{n}a_{z}^{n}c(dx_{\beta})c(dx_{l})c(dx_{z})c(dx_{h})]\nonumber\\
&+\frac{1-3\xi_n^2}{4(\xi_n-i)^4(\xi_n+i)^3}h'(0)\sum_{\beta,l,z=1}^{n}\sum_{i,p,h=1}^{n-1}{\rm tr}[\xi_{i}\xi_{p}a_{\beta}^{i}a_{l}^{n}a_{h}^{p}a_{z}^{n}c(dx_{\beta})c(dx_{l})c(dx_{z})c(dx_{h})]\nonumber\\
&+\frac{i(1-3\xi_n^2)}{4(\xi_n-i)^4(\xi_n+i)^3}h'(0)\sum_{\beta,l,z=1}^{n}\sum_{p,h=1}^{n-1}{\rm tr}[\xi_{p}a_{\beta}^{n}a_{l}^{n}a_{h}^{p}a_{z}^{n}c(dx_{\beta})c(dx_{l})c(dx_{z})c(dx_{h})]\nonumber\\
&-\frac{\xi_n}{(\xi_n-i)^4(\xi_n+i)^3}h'(0)\sum_{\beta,l,z=1}^{n}\sum_{i,q,p,h=1}^{n-1}{\rm tr}[\xi_{i}\xi_{q}\xi_{p}a_{\beta}^{i}a_{l}^{q}a_{h}^{p}a_{z}^{n}c(dx_{\beta})c(dx_{l})c(dx_{z})c(dx_{h})]\nonumber\\
&-\frac{i\xi_n}{(\xi_n-i)^4(\xi_n+i)^3}h'(0)\sum_{\beta,l,z=1}^{n}\sum_{q,p,h=1}^{n-1}{\rm tr}[\xi_{q}\xi_{p}a_{\beta}^{n}a_{l}^{q}a_{h}^{p}a_{z}^{n}c(dx_{\beta})c(dx_{l})c(dx_{z})c(dx_{h})],\nonumber
\end{align}
then
\begin{align}
&-i\int_{|\xi'|=1}\int^{+\infty}_{-\infty}{\rm trace} [\pi^+_{\xi_n}\sigma_{-1}({{D}_{J}}^{-1}) \times \partial_{\xi_n}(A_2(x_0))](x_0)d\xi_n\sigma(\xi')dx'\\
&=\Omega_3\int_{\Gamma^{+}}\frac{-\xi_n(-1+\xi_n^2)}{(\xi_n-i)^4(\xi_n+i)^3}\sum_{\beta,l,j=1}^{n}\left((a_{\beta}^{n})^2a_{l}^{j}\partial_{x_j}(a_{l}^{n})-a_{l}^{n}a_{\beta}^{j}a_{\beta}^{n}\partial_{x_j}(a_{l}^{n})+a_{l}^{n}a_{l}^{j}a_{\beta}^{n}\partial_{x_j}(a_{\beta}^{n})\right)\nonumber\\
&{\rm tr}[\texttt{id}]d\xi_ndx'\nonumber
\end{align}
\begin{align}
&+\frac{4\pi}{3}\Omega_3\int_{\Gamma^{+}}\frac{-i(1-3\xi_n^2)}{2(\xi_n-i)^4(\xi_n+i)^3}\sum_{\beta,l,j=1}^{n}\sum_{i=1}^{n-1}\left((a_{\beta}^{i})^2a_{l}^{j}\partial_{x_j}(a_{l}^{n})-a_{l}^{i}a_{\beta}^{j}a_{\beta}^{i}\partial_{x_j}(a_{l}^{n})+a_{l}^{i}a_{l}^{j}a_{\beta}^{i}\partial_{x_j}(a_{\beta}^{n})\right)\nonumber\\
&{\rm tr}[\texttt{id}]d\xi_ndx'\nonumber\\
&+\frac{4\pi}{3}\Omega_3\int_{\Gamma^{+}}\frac{-i(1-3\xi_n^2)}{2(\xi_n-i)^4(\xi_n+i)^3}\sum_{\beta,l,j=1}^{n}\sum_{i=1}^{n-1}\left(a_{\beta}^{n}a_{l}^{j}a_{\beta}^{i}\partial_{x_j}(a_{l}^{i})-a_{l}^{n}a_{\beta}^{j}a_{\beta}^{i}\partial_{x_j}(a_{l}^{i})+a_{l}^{n}a_{l}^{j}a_{\beta}^{i}\partial_{x_j}(a_{\beta}^{i})\right)\nonumber\\
&{\rm tr}[\texttt{id}]d\xi_ndx'\nonumber\\
&+\frac{4\pi}{3}\Omega_3\int_{\Gamma^{+}}\frac{-2\xi_n}{(\xi_n-i)^4(\xi_n+i)^3}\sum_{\beta,l,j=1}^{n}\sum_{i=1}^{n-1}\left(a_{\beta}^{i}a_{l}^{j}a_{\beta}^{n}\partial_{x_j}(a_{l}^{i})-a_{l}^{i}a_{\beta}^{j}a_{\beta}^{n}\partial_{x_j}(a_{l}^{i})+a_{l}^{i}a_{l}^{j}a_{\beta}^{n}\partial_{x_j}(a_{\beta}^{i})\right)\nonumber\\
&{\rm tr}[\texttt{id}]d\xi_ndx'\nonumber\\
&+\Omega_3\int_{\Gamma^{+}}\frac{-\xi_n(-1+\xi_n^2)}{2(\xi_n-i)^4(\xi_n+i)^3}h'(0)\sum_{l=1}^{n}\sum_{\nu=1}^{n-1}(a_{\nu}^{n})^2(a_{l}^{n})^2{\rm tr}[\texttt{id}]d\xi_ndx'\nonumber\\
&+\frac{4\pi}{3}\Omega_3\int_{\Gamma^{+}}\frac{-i(1-3\xi_n^2)}{4(\xi_n-i)^4(\xi_n+i)^3}h'(0)\sum_{l=1}^{n}\sum_{\nu,i=1}^{n-1}(a_{\nu}^{n})^2(a_{l}^{i})^2{\rm tr}[\texttt{id}]d\xi_ndx'\nonumber\\
&+\frac{4\pi}{3}\Omega_3\int_{\Gamma^{+}}\frac{-i(1-3\xi_n^2)}{4(\xi_n-i)^4(\xi_n+i)^3}h'(0)\sum_{l=1}^{n}\sum_{\nu,i=1}^{n-1}(a_{\nu}^{i})^2(a_{l}^{n})^2{\rm tr}[\texttt{id}]d\xi_ndx'\nonumber\\
&+\frac{4\pi}{3}\Omega_3\int_{\Gamma^{+}}\frac{-\xi_n}{(\xi_n-i)^4(\xi_n+i)^3}h'(0)\sum_{l=1}^{n}\sum_{\nu,i=1}^{n-1}\left(2a_{\nu}^{i}a_{l}^{i}a_{\nu}^{n}a_{l}^{n}-(a_{\nu}^{i})^2(a_{l}^{n})^2\right){\rm tr}[\texttt{id}]d\xi_ndx'.\nonumber
\end{align}
We get
\begin{align}
&-i\int_{|\xi'|=1}\int^{+\infty}_{-\infty}{\rm trace} [\pi^+_{\xi_n}\sigma_{-1}({{D}_{J}}^{-1}) \times \partial_{\xi_n}(A_2(x_0))](x_0)d\xi_n\sigma(\xi')dx'\\
&=\sum_{l,j,\beta=1}^{n}\sum_{i=1}^{n-1}\left((a_{\beta}^{i})^2a_{l}^{j}\partial_{x_j}(a_{l}^{n})-a_{l}^{i}a_{\beta}^{j}a_{\beta}^{i}\partial_{x_j}(a_{l}^{n})+a_{l}^{i}a_{l}^{j}a_{\beta}^{i}\partial_{x_j}(a_{\beta}^{n})\right){\rm tr}[\texttt{id}]\Omega_3(\frac{\pi^2}{12})dx'\nonumber\\
&+\sum_{l,j,\beta=1}^{n}\sum_{i=1}^{n-1}\left(a_{\beta}^{n}a_{l}^{j}a_{\beta}^{i}\partial_{x_j}(a_{l}^{i})-a_{l}^{n}a_{\beta}^{j}a_{\beta}^{i}\partial_{x_j}(a_{l}^{i})+a_{l}^{n}a_{l}^{j}a_{\beta}^{i}\partial_{x_j}(a_{\beta}^{i})\right){\rm tr}[\texttt{id}]\Omega_3(\frac{\pi^2}{12})dx'\nonumber\\
&+\sum_{l,j,\beta=1}^{n}\sum_{i=1}^{n-1}\left(a_{\beta}^{i}a_{l}^{j}a_{\beta}^{n}\partial_{x_j}(a_{l}^{i})-a_{l}^{i}a_{\beta}^{j}a_{\beta}^{n}\partial_{x_j}(a_{l}^{i})+a_{l}^{i}a_{l}^{j}a_{\beta}^{n}\partial_{x_j}(a_{\beta}^{i})\right){\rm tr}[\texttt{id}]\Omega_3(-\frac{\pi^2}{6})dx'\nonumber
\end{align}
\begin{align}
&+\sum_{l=1}^{n}\sum_{\nu,i=1}^{n-1}(a_{\nu}^{n})^2(a_{l}^{i})^2{\rm tr}[\texttt{id}]\Omega_3h'(0)(\frac{\pi^2}{24})dx'+\sum_{l=1}^{n}\sum_{\nu,i=1}^{n-1}(a_{\nu}^{i})^2(a_{l}^{n})^2{\rm tr}[\texttt{id}]\Omega_3h'(0)(\frac{\pi^2}{24})dx'\nonumber\\
&+\sum_{l=1}^{n}\sum_{\nu,i=1}^{n-1}\left(2a_{\nu}^{i}a_{l}^{i}a_{\nu}^{n}a_{l}^{n}-(a_{\nu}^{i})^2(a_{l}^{n})^2\right){\rm tr}[\texttt{id}]\Omega_3h'(0)(-\frac{\pi^2}{12})dx'.\nonumber
\end{align}
Likewise,
\begin{align}
-h'(0)\partial_{\xi_n}(A_3(x_0))&=-h'(0)\partial_{\xi_n}\Big(\frac{\sum_{q,l,w,\alpha,\gamma=1}^{n}\xi_{q}\xi_{\alpha}a_{l}^{q}a_{w}^{n}a_{\gamma}^{\alpha}c(dx_{l})c(dx_{w})c(dx_{\gamma})}{(1+\xi_n^2)^3}\Big)\\
&=-\frac{2\xi_n-4\xi_n^3}{(1+\xi_n^2)^4}h'(0)\sum_{l,w,\gamma=1}^{n}a_{l}^{n}a_{w}^{n}a_{\gamma}^{n}c(dx_{l})c(dx_{w})c(dx_{\gamma})\nonumber\\
&-\frac{1-5\xi_n^2}{(1+\xi_n^2)^4}h'(0)\sum_{l,w,\gamma=1}^{n}\sum_{q=1}^{n-1}\xi_{q}a_{l}^{q}a_{w}^{n}a_{\gamma}^{n}c(dx_{l})c(dx_{w})c(dx_{\gamma})\nonumber\\
&-\frac{1-5\xi_n^2}{(1+\xi_n^2)^4}h'(0)\sum_{l,w,\gamma=1}^{n}\sum_{\alpha=1}^{n-1}\xi_{\alpha}a_{l}^{n}a_{w}^{n}a_{\gamma}^{\alpha}c(dx_{l})c(dx_{w})c(dx_{\gamma})\nonumber\\
&+\frac{6\xi_n}{(1+\xi_n^2)^4}h'(0)\sum_{l,w,\gamma=1}^{n}\sum_{q,\alpha=1}^{n-1}\xi_{q}\xi_{\alpha}a_{l}^{q}a_{w}^{n}a_{\gamma}^{\alpha}c(dx_{l})c(dx_{w})c(dx_{\gamma}).\nonumber
\end{align}
Observing (3.92) and (3.102), we have
\begin{align}
&{\rm trace} [\pi^+_{\xi_n}\sigma_{-1}({{D}_{J}}^{-1}) \times -h'(0)\partial_{\xi_n}(A_3(x_0))](x_0)|_{|\xi'|=1}\\
&=-\frac{\xi_n-2\xi_n^3}{(\xi_n-i)^5(\xi_n+i)^4}h'(0)\sum_{\beta,l,w,\gamma=1}^{n}\sum_{i=1}^{n-1}{\rm tr}[\xi_{i}a_{\beta}^{i}a_{l}^{n}a_{w}^{n}a_{\gamma}^{n}c(dx_{\beta})c(dx_{l})c(dx_{w})c(dx_{\gamma})]\nonumber\\
&-\frac{i(\xi_n-2\xi_n^3)}{(\xi_n-i)^5(\xi_n+i)^4}h'(0)\sum_{\beta,l,w,\gamma=1}^{n}{\rm tr}[a_{\beta}^{n}a_{l}^{n}a_{w}^{n}a_{\gamma}^{n}c(dx_{\beta})c(dx_{l})c(dx_{w})c(dx_{\gamma})]\nonumber\\
&-\frac{1-5\xi_n^2}{2(\xi_n-i)^5(\xi_n+i)^4}h'(0)\sum_{\beta,l,w,\gamma=1}^{n}\sum_{i,q=1}^{n-1}{\rm tr}[\xi_{i}\xi_{q}a_{\beta}^{i}a_{l}^{q}a_{w}^{n}a_{\gamma}^{n}c(dx_{\beta})c(dx_{l})c(dx_{w})c(dx_{\gamma})]\nonumber\\
&-\frac{i(1-5\xi_n^2)}{2(\xi_n-i)^5(\xi_n+i)^4}h'(0)\sum_{\beta,l,w,\gamma=1}^{n}\sum_{q=1}^{n-1}{\rm tr}[\xi_{q}a_{\beta}^{n}a_{l}^{q}a_{w}^{n}a_{\gamma}^{n}c(dx_{\beta})c(dx_{l})c(dx_{w})c(dx_{\gamma})]\nonumber
\end{align}
\begin{align}
&-\frac{1-5\xi_n^2}{2(\xi_n-i)^5(\xi_n+i)^4}h'(0)\sum_{\beta,l,w,\gamma=1}^{n}\sum_{i,\alpha=1}^{n-1}{\rm tr}[\xi_{i}\xi_{\alpha}a_{\beta}^{i}a_{l}^{n}a_{w}^{n}a_{\gamma}^{\alpha}c(dx_{\beta})c(dx_{l})c(dx_{w})c(dx_{\gamma})]\nonumber\\
&-\frac{i(1-5\xi_n^2)}{2(\xi_n-i)^5(\xi_n+i)^4}h'(0)\sum_{\beta,l,w,\gamma=1}^{n}\sum_{\alpha=1}^{n-1}{\rm tr}[\xi_{\alpha}a_{\beta}^{n}a_{l}^{n}a_{w}^{n}a_{\gamma}^{\alpha}c(dx_{\beta})c(dx_{l})c(dx_{w})c(dx_{\gamma})]\nonumber\\
&+\frac{3\xi_n}{(\xi_n-i)^5(\xi_n+i)^4}h'(0)\sum_{\beta,l,w,\gamma=1}^{n}\sum_{i,q,\alpha=1}^{n-1}{\rm tr}[\xi_{i}\xi_{q}\xi_{\alpha}a_{\beta}^{i}a_{l}^{q}a_{w}^{n}a_{\gamma}^{\alpha}c(dx_{\beta})c(dx_{l})c(dx_{w})c(dx_{\gamma})]\nonumber\\
&+\frac{3i\xi_n}{(\xi_n-i)^5(\xi_n+i)^4}h'(0)\sum_{\beta,l,w,\gamma=1}^{n}\sum_{q,\alpha=1}^{n-1}{\rm tr}[\xi_{q}\xi_{\alpha}a_{\beta}^{n}a_{l}^{q}a_{w}^{n}a_{\gamma}^{\alpha}c(dx_{\beta})c(dx_{l})c(dx_{w})c(dx_{\gamma})].\nonumber
\end{align}
By $\int_{|\xi'|=1}{\{\xi_{i_1}\cdot\cdot\cdot\xi_{i_{2d+1}}}\}\sigma(\xi')=0$ and (3.103), we have
\begin{align}
&-i\int_{|\xi'|=1}\int^{+\infty}_{-\infty}{\rm tr} [\pi^+_{\xi_n}\sigma_{-1}({{D}_{J}}^{-1}) \times -h'(0)\partial_{\xi_n}(A_3(x_0))](x_0)d\xi_n\sigma(\xi')dx'\\
&=\Omega_3\int_{\Gamma^{+}}\frac{-(\xi_n-2\xi_n^3)}{(\xi_n-i)^5(\xi_n+i)^4}h'(0)\sum_{\beta,l=1}^{n}(a_{\beta}^{n})^2(a_{l}^{n})^2{\rm tr}[\texttt{id}]d\xi_ndx'\nonumber\\
&+\frac{8\pi}{3}\Omega_3\int_{\Gamma^{+}}\frac{i(1-5\xi_n^2)}{2(\xi_n-i)^5(\xi_n+i)^4}h'(0)\sum_{\beta,l=1}^{n}\sum_{i=1}^{n-1}(a_{\beta}^{i})^2(a_{l}^{n})^2{\rm tr}[\texttt{id}]d\xi_ndx'\nonumber\\
&+\frac{4\pi}{3}\Omega_3\int_{\Gamma^{+}}\frac{3\xi_n}{(\xi_n-i)^5(\xi_n+i)^4}h'(0)\sum_{\beta,l=1}^{n}\sum_{i=1}^{n-1}\left(2a_{l}^{i}a_{\beta}^{i}a_{l}^{n}a_{\beta}^{n}-(a_{l}^{i})^2(a_{\beta}^{n})^2\right){\rm tr}[\texttt{id}]d\xi_ndx'\nonumber\\
&=\sum_{\beta,l=1}^{n}(a_{\beta}^{n})^2(a_{l}^{n})^2{\rm tr}[\texttt{id}]\Omega_3h'(0)(\frac{\pi}{128})dx'+\sum_{\beta,l=1}^{n}\sum_{i=1}^{n-1}(a_{\beta}^{i})^2(a_{l}^{n})^2{\rm tr}[\texttt{id}]\Omega_3h'(0)(-\frac{5\pi^2}{48})dx'\nonumber\\
&+\sum_{\beta,l=1}^{n}\sum_{i=1}^{n-1}\left(2a_{l}^{i}a_{\beta}^{i}a_{l}^{n}a_{\beta}^{n}-(a_{l}^{i})^2(a_{\beta}^{n})^2\right){\rm tr}[\texttt{id}]\Omega_3h'(0)(\frac{5\pi^2}{32})dx'.\nonumber
\end{align}
Then,
\begin{align}
\Psi_5&=\sum_{l=1}^{n}\sum_{\nu,i=1}^{n-1}(-2(a_{\nu}^{n})^2(a_{l}^{i})^2+2(a_{l}^{i})^2a_{\nu}^{\nu}a_{n}^{n}+2a_{\nu}^{i}a_{l}^{i}a_{\nu}^{n}a_{l}^{n}-2a_{i}^{i}a_{\nu}^{\nu}){\rm tr}[\texttt{id}]\Omega_3h'(0)(-\frac{\pi^2}{24})dx'
\end{align}
\begin{align}
&+\sum_{l,j,\beta=1}^{n}\sum_{i=1}^{n-1}\left((a_{\beta}^{i})^2a_{l}^{j}\partial_{x_j}(a_{l}^{n})-a_{l}^{i}a_{\beta}^{j}a_{\beta}^{i}\partial_{x_j}(a_{l}^{n})+a_{l}^{i}a_{l}^{j}a_{\beta}^{i}\partial_{x_j}(a_{\beta}^{n})\right){\rm tr}[\texttt{id}]\Omega_3(\frac{\pi^2}{12})dx'\nonumber\\
&+\sum_{l,j,\beta=1}^{n}\sum_{i=1}^{n-1}\left(a_{\beta}^{n}a_{l}^{j}a_{\beta}^{i}\partial_{x_j}(a_{l}^{i})-a_{l}^{n}a_{\beta}^{j}a_{\beta}^{i}\partial_{x_j}(a_{l}^{i})+a_{l}^{n}a_{l}^{j}a_{\beta}^{i}\partial_{x_j}(a_{\beta}^{i})\right){\rm tr}[\texttt{id}]\Omega_3(\frac{\pi^2}{12})dx'\nonumber\\
&+\sum_{l,j,\beta=1}^{n}\sum_{i=1}^{n-1}\left(a_{\beta}^{i}a_{l}^{j}a_{\beta}^{n}\partial_{x_j}(a_{l}^{i})-a_{l}^{i}a_{\beta}^{j}a_{\beta}^{n}\partial_{x_j}(a_{l}^{i})+a_{l}^{i}a_{l}^{j}a_{\beta}^{n}\partial_{x_j}(a_{\beta}^{i})\right){\rm tr}[\texttt{id}]\Omega_3(-\frac{\pi^2}{6})dx'\nonumber\\
&+\sum_{l=1}^{n}\sum_{\nu,i=1}^{n-1}(a_{\nu}^{n})^2(a_{l}^{i})^2{\rm tr}[\texttt{id}]\Omega_3h'(0)(\frac{\pi^2}{24})dx'+\sum_{l=1}^{n}\sum_{\nu,i=1}^{n-1}(a_{\nu}^{i})^2(a_{l}^{n})^2{\rm tr}[\texttt{id}]\Omega_3h'(0)(\frac{\pi^2}{24})dx'\nonumber\\
&+\sum_{l=1}^{n}\sum_{\nu,i=1}^{n-1}\left(2a_{\nu}^{i}a_{l}^{i}a_{\nu}^{n}a_{l}^{n}-(a_{\nu}^{i})^2(a_{l}^{n})^2\right){\rm tr}[\texttt{id}]\Omega_3h'(0)(-\frac{\pi^2}{12})dx'\nonumber\\
&+\sum_{\beta,l=1}^{n}(a_{\beta}^{n})^2(a_{l}^{n})^2{\rm tr}[\texttt{id}]\Omega_3h'(0)(\frac{\pi}{128})dx'+\sum_{\beta,l=1}^{n}\sum_{i=1}^{n-1}(a_{\beta}^{i})^2(a_{l}^{n})^2{\rm tr}[\texttt{id}]\Omega_3h'(0)(-\frac{5\pi^2}{48})dx'\nonumber\\
&+\sum_{\beta,l=1}^{n}\sum_{i=1}^{n-1}\left(2a_{l}^{i}a_{\beta}^{i}a_{l}^{n}a_{\beta}^{n}-(a_{l}^{i})^2(a_{\beta}^{n})^2\right){\rm tr}[\texttt{id}]\Omega_3h'(0)(\frac{5\pi^2}{32})dx'.\nonumber
\end{align}

In summary,
\begin{align}
\Psi&=\Psi_1+\Psi_2+\Psi_3+\Psi_4+\Psi_5\\
&=\sum_{\beta=1}^{n}\sum_{i=1}^{n-1}a_{\beta}^{i}\partial_{x_i}(a_{\beta}^{n}){\rm tr}[\texttt{id}]\Omega_3(-\frac{\pi}{8}+\frac{\pi^2}{3})dx'+\sum_{\beta=1}^{n}\sum_{i=1}^{n-1}a_{\beta}^{n}\partial_{x_i}(a_{\beta}^{i}){\rm tr}[\texttt{id}]\Omega_3(-\frac{\pi^2}{6})dx'.\nonumber
\end{align}
\begin{lem}
\begin{align}
\sum_{\beta=1}^{n}\sum_{i=1}^{n-1}a_{\beta}^{i}\partial_{x_i}(a_{\beta}^{n})=\sum_{\beta=1}^{n}\langle\nabla_{J(e_{\beta})}^{L}(Je_{n}), e_{\beta}\rangle-\sum_{\beta=1}^{n}g^{M}\left(J(\frac{\partial}{\partial{x_{n}}}), \frac{\partial}{\partial{x_{n}}}\right)\langle\nabla_{J(e_{\beta})}^{L}(\frac{\partial}{\partial{x_{n}}}), e_{\beta}\rangle,
\end{align}
\begin{align}
\sum_{\beta=1}^{n}\sum_{i=1}^{n-1}a_{\beta}^{n}\partial_{x_i}(a_{\beta}^{i})=-\sum_{\beta=1}^{n}\langle\nabla_{J(e_{\beta})}^{L}(Je_{n}), e_{\beta}\rangle+\sum_{\beta=1}^{n}g^{M}\left(J(\frac{\partial}{\partial{x_{n}}}), \frac{\partial}{\partial{x_{n}}}\right)\langle\nabla_{J(e_{\beta})}^{L}(\frac{\partial}{\partial{x_{n}}}), e_{\beta}\rangle.
\end{align}
\end{lem}
\begin{proof}
We can certainly assume that $J(\frac{\partial}{\partial{x_{l}}})=\sum^{n}_{q=1}\widetilde{a}^{q}_{l}\frac{\partial}{\partial{x_{q}}},$ since
\begin{align}
\langle \frac{\partial}{\partial{x_{l}}}, J(dx_{p})\rangle&=\langle J(\frac{\partial}{\partial{x_{l}}}), dx_{p}\rangle,
\end{align}
then we have $a^{q}_{h}=\widetilde{a}^{q}_{h}.$
As
\begin{align}
J^2=\texttt{id}, J^{T}J=\texttt{id},
\end{align}
we have $a^{p}_{l}a^{j}_{p}=\delta_l^j, a^{p}_{l}=a^{l}_{p}.$
We can get
\begin{align}
&\sum_{\beta=1}^{n}\sum_{i=1}^{n-1}a_{\beta}^{i}\partial_{x_i}(a_{\beta}^{n})=\sum_{\beta,i=1}^{n}a_{\beta}^{i}\partial_{x_i}(a_{\beta}^{n})-\sum_{\beta=1}^{n}a_{\beta}^{n}\partial_{x_n}(a_{\beta}^{n})\\
&=\sum_{\beta,i=1}^{n}a_{\beta}^{i}\partial_{x_i}(a_{\beta}^{n})-\frac{1}{2}\partial_{x_n}(\delta_n^n)=\sum_{\beta,i=1}^{n}a_{\beta}^{i}\partial_{x_i}(a_{\beta}^{n}).\nonumber
\end{align}
We have
\begin{align}
&\sum_{\beta=1}^{n}\langle\nabla_{J(e_{\beta})}^{L}(Je_{n}), e_{\beta}\rangle(x_0)=\sum_{\beta=1}^{n}\langle\nabla_{J(e_{\beta})}^{L}(\sum_{h=1}^{n}a_{h}^{n}\frac{\partial}{\partial{x_{h}}}), e_{\beta}\rangle(x_0)\\
&=\sum_{\beta,h=1}^{n}\langle J(e_{\beta})(a_{h}^{n})\frac{\partial}{\partial{x_{h}}}, e_{\beta}\rangle(x_0)+\sum_{\beta,h=1}^{n}a_{h}^{n}\langle \nabla_{J(e_{\beta})}^{L}(\frac{\partial}{\partial{x_{h}}}), e_{\beta}\rangle(x_0)\nonumber\\
&=\sum_{\beta,i=1}^{n}a_{\beta}^{i}\frac{\partial}{\partial{x_{\beta}}}(a_{\beta}^{n})(x_0)+\sum_{\beta,h=1}^{n}a_{h}^{n}\langle \nabla_{J(e_{\beta})}^{L}(\frac{\partial}{\partial{x_{h}}}), e_{\beta}\rangle(x_0).\nonumber
\end{align}
Let us compute
\begin{align}
&\sum_{\beta,h=1}^{n}a_{h}^{n}\langle \nabla_{J(e_{\beta})}^{L}(\frac{\partial}{\partial{x_{h}}}), e_{\beta}\rangle(x_0)
=\sum_{\beta=1}^{n}\sum_{h=1}^{n-1}a_{h}^{n}\langle \nabla_{e_{\beta}}^{L}(\frac{\partial}{\partial{x_{h}}}), J(e_{\beta})\rangle(x_0)+\sum_{\beta=1}^{n}a_{n}^{n}\langle \nabla_{J(e_{\beta})}^{L}(\frac{\partial}{\partial{x_{n}}}), e_{\beta}\rangle(x_0)\\
&=\sum_{\beta,\alpha=1}^{n}\sum_{h=1}^{n-1}a_{h}^{n}a_{\beta}^{\alpha}\langle \nabla_{\frac{\partial}{\partial{x_{\beta}}}}^{L}(\frac{\partial}{\partial{x_{h}}}), \frac{\partial}{\partial{x_{\alpha}}}\rangle(x_0)
+\sum_{\beta=1}^{n}g^{M}\left(J(\frac{\partial}{\partial{x_{n}}}), \frac{\partial}{\partial{x_{n}}}\right)\langle \nabla_{J(e_{\beta})}^{L}(\frac{\partial}{\partial{x_{n}}}), e_{\beta}\rangle(x_0)\nonumber\\
&=\sum_{\beta,\alpha=1}^{n}\sum_{h=1}^{n-1}a_{h}^{n}a_{\beta}^{\alpha}\Gamma_{\beta h}^{\alpha}(x_0)
+\sum_{\beta=1}^{n}g^{M}\left(J(\frac{\partial}{\partial{x_{n}}}), \frac{\partial}{\partial{x_{n}}}\right)\langle \nabla_{J(e_{\beta})}^{L}(\frac{\partial}{\partial{x_{n}}}), e_{\beta}\rangle(x_0)\nonumber\\
&=\sum_{\beta=1}^{n}g^{M}\left(J(\frac{\partial}{\partial{x_{n}}}), \frac{\partial}{\partial{x_{n}}}\right)\langle \nabla_{J(e_{\beta})}^{L}(\frac{\partial}{\partial{x_{n}}}), e_{\beta}\rangle(x_0),\nonumber
\end{align}
where we used
$\nabla_{\frac{\partial}{\partial{x_{\beta}}}}^{L}(\frac{\partial}{\partial{x_{h}}})=\sum_{k=1}^{n}\Gamma_{\beta h}^{k}\frac{\partial}{\partial{x_{k}}},$ $\sum_{\beta,\alpha=1}^{n}\sum_{h=1}^{n-1}a_{h}^{n}a_{\beta}^{\alpha}\Gamma_{\beta h}^{\alpha}(x_0)=0$ and Lemma A.$2$ in \cite{Wa3}.
We thus get
\begin{align}
&\sum_{\beta=1}^{n}\sum_{i=1}^{n-1}a_{\beta}^{i}\partial_{x_i}(a_{\beta}^{n})=\sum_{\beta=1}^{n}\langle\nabla_{J(e_{\beta})}^{L}(Je_{n}),e_{\beta}\rangle-\sum_{\beta=1}^{n}g^{M}\left(J(\frac{\partial}{\partial{x_{n}}}), \frac{\partial}{\partial{x_{n}}}\right)\langle \nabla_{J(e_{\beta})}^{L}(\frac{\partial}{\partial{x_{n}}}), e_{\beta}\rangle.
\end{align}
We next show that
\begin{align}
0=\sum_{\beta=1}^{n}\sum_{i=1}^{n-1}\partial_{x_i}(a_{\beta}^{i}a_{\beta}^{n})=\sum_{\beta=1}^{n}\sum_{i=1}^{n-1}a_{\beta}^{i}\partial_{x_i}(a_{\beta}^{n})+\sum_{\beta=1}^{n}\sum_{i=1}^{n-1}a_{\beta}^{n}\partial_{x_i}(a_{\beta}^{i}),
\end{align}
that is,
\begin{align}
\sum_{\beta=1}^{n}\sum_{i=1}^{n-1}a_{\beta}^{i}\partial_{x_i}(a_{\beta}^{n})=-\sum_{\beta=1}^{n}\sum_{i=1}^{n-1}a_{\beta}^{n}\partial_{x_i}(a_{\beta}^{i}).
\end{align}
\end{proof}

Applying (3.106), (3.107) and (3.108), we can assert that:
\begin{thm}
Let $M$ be a $4$-dimensional almost product Riemannian spin manifold with the boundary $\partial M$ and the metric
$g^M$ as above, ${{D}_{J}}$ be the $J$-twist of the Dirac operator on $\widetilde{M}$, then
\begin{align}
\widetilde{{\rm Wres}}[\pi^+{{D}_{J}}^{-1}\circ\pi^+{{D}_{J}}^{-1}]=&\int_{M}32\pi^{2}\Big(\sum_{i,j=1}^{4}R(J(e_{i}), J(e_{j}), e_{j}, e_{i})-2\sum_{\nu,j=1}^{4}g^{M}(\nabla_{e_{j}}^{L}(J)e_{\nu}, (\nabla^{L}_{e_{\nu}}J)e_{j})\\
&-2\sum_{\nu,j=1}^{4}g^{M}(J(e_{\nu}), (\nabla^{L}_{e_{j}}(\nabla^{L}_{e_{\nu}}(J)))e_{j}-(\nabla^{L}_{\nabla^{L}_{e_{j}}e_{\nu}}(J))e_{j})\nonumber\\
&-\sum_{\alpha,\nu,j=1}^{4}g^{M}(J(e_{\alpha}), (\nabla^{L}_{e_{\nu}}J)e_{j})g^{M}((\nabla^{L}_{e_{\alpha}}J)e_{j}, J(e_{\nu}))\nonumber\\
&-\sum_{\alpha,\nu,j=1}^{4}g^{M}(J(e_{\alpha}), (\nabla^{L}_{e_{\alpha}}J)e_{j})g^{M}(J(e_{\nu}), (\nabla^{L}_{e_{\nu}}J)e_{j})\nonumber\\
&+\sum_{\nu,j=1}^{4}g^{M}((\nabla^{L}_{e_{\nu}}J)e_{j}, (\nabla^{L}_{e_{\nu}}J)e_{j}))-\frac{1}{3}s\Big)d{\rm Vol_{M} }\nonumber\\
&+\int_{\partial M}(-\frac{\pi}{2}+2\pi^2)\Big(\sum_{\beta=1}^{4}\langle\nabla_{J(e_{\beta})}^{L}(Je_{4}), e_{\beta}\rangle\nonumber\\
&-\sum_{\beta=1}^{4}g^{M}\left(J(\frac{\partial}{\partial{x_{4}}}), \frac{\partial}{\partial{x_{4}}}\right)\langle\nabla_{J(e_{\beta})}^{L}(\frac{\partial}{\partial{x_{4}}}), e_{\beta}\rangle\Big)\Omega_3d{\rm Vol_{\partial M}},\nonumber
\end{align}
where $s$ is the scalar curvature.
\end{thm}
\section{ A Kastler-Kalau-Walze type theorem for $3$-dimensional manifolds with boundary}
For an odd-dimensional manifolds with boundary, as in Sections 5, 6 and 7 in \cite{Wa1}, we have the formula
\begin{align}
&\widetilde{{\rm Wres}}[(\pi^+{{D}_{J}}^{-1})^2]=\int_{\partial M}\Phi.
\end{align}
When $n=3,$ then in (3.10) $r+l-k-|\alpha|-j-1=-3,~~r,l\leq -1,$ then
\begin{align}
\Phi&=\int_{|\xi'|=1}\int^{+\infty}_{-\infty} {\rm trace}_{S(TM)}[\sigma^+_{-1}(D_{J}^{-1})(x',0,\xi',\xi_n)\times\partial_{\xi_n}\sigma_{-1}(D_{J}^{-1})(x',0,\xi',\xi_n)]d\xi_n\sigma(\xi')dx'.
\end{align}
We have
\begin{align}
\pi^+_{\xi_n}\left(\frac{ic[J(\xi)]}{|\xi|^2}\right)(x_0)|_{|\xi'|=1}
=\frac{1}{2(\xi_{n}-i)}\sum^{n}_{h=1}\sum^{n-1}_{p=1}\xi_{p}a_{h}^{p}c(dx_{h})+\frac{i}{2(\xi_{n}-i)}\sum^{n}_{h=1}a_{h}^{n}c(dx_{h}),
\end{align}
\begin{align}
\partial_{\xi_n}\left(\frac{ic[J(\xi)]}{|\xi|^2}\right)(x_0)|_{|\xi'|=1}
&=-\frac{2i\xi_n}{(1+\xi_{n}^2)^2}\sum^{n}_{\beta=1}\sum^{n-1}_{i=1}\xi_{i}a_{\beta}^ic(dx_{\beta})+\frac{i(1-\xi_n^2)}{(1+\xi_{n}^2)^2}\sum^{n}_{\beta=1}a_{\beta}^nc(dx_{\beta}).
\end{align}
For $n=3,$ we take the coordinates in Section 2. Locally $S(TM)|_{\widetilde{U}}\cong\widetilde{U}\times\wedge_{\rm{C}}^{even}(2).$ Let ${\widetilde{f_{1}},\widetilde{f_{2}}}$ be an orthonormal basis $\wedge_{\rm{C}}^{even}(2)$ and we will compute the trace under this basis.
We get ${\rm tr}[\texttt{id}]=2.$\\
Then
\begin{align}
&{\rm trace}_{S(TM)}[\sigma^+_{-1}(D_{J}^{-1})\times\partial_{\xi_n}\sigma_{-1}(D_{J}^{-1})](x_0)|_{|\xi'|=1}\\
&=-\frac{i\xi_n}{(\xi_n-i)^3(\xi_n+i)^2}\sum_{h,\beta=1}^{n}\sum_{p,i=1}^{n-1}{\rm tr}[\xi_{p}\xi_{i}a_{h}^{p}a_{\beta}^{i}c(dx_{h})c(dx_{\beta})]\nonumber\\
&+\frac{i(1-\xi_n^2)}{2(\xi_n-i)^3(\xi_n+i)^2}\sum_{h,\beta=1}^{n}\sum_{p=1}^{n-1}{\rm tr}[\xi_{p}a_{h}^{p}a_{\beta}^{n}c(dx_{h})c(dx_{\beta})]\nonumber\\
&+\frac{\xi_n}{(\xi_n-i)^3(\xi_n+i)^2}\sum_{h,\beta=1}^{n}\sum_{i=1}^{n-1}{\rm tr}[\xi_{i}a_{h}^{n}a_{\beta}^{i}c(dx_{h})c(dx_{\beta})]\nonumber\\
&-\frac{1-\xi_n^2}{2(\xi_n-i)^3(\xi_n+i)^2}\sum_{h,\beta=1}^{n}{\rm tr}[a_{h}^{n}a_{\beta}^{n}c(dx_{h})c(dx_{\beta})].\nonumber
\end{align}
By the Cauchy integral formula, we get
\begin{align}
\Phi&=\sum_{\beta=1}^{n}\sum_{i=1}^{n-1}(a_{\beta}^{i})^2{\rm tr}[\texttt{id}]\Omega_2(\frac{i\pi^2}{6})dx'+\sum_{\beta=1}^{n}(a_{\beta}^{n})^2{\rm tr}[\texttt{id}]\Omega_2(\frac{i\pi}{8})dx'.
\end{align}
\begin{thm}
Let $M$ be a $3$-dimensional almost product Riemannian spin manifold with the boundary $\partial M$ and the metric
$g^M$ as above, ${{D}_{J}}$ be the $J$-twist of the Dirac operator on $\widetilde{M}$, then
\begin{align}
\widetilde{{\rm Wres}}[(\pi^+{{D}_{J}}^{-1})^2]
&=\int_{\partial M}\Big(\sum_{\beta=1}^{3}\sum_{i=1}^{2}\frac{i\pi^2}{3}{\langle J(e_{\beta}), e_{i}\rangle}^2+\sum_{\beta=1}^{3}\frac{i\pi}{4}{\langle J(e_{\beta}), e_{3}\rangle}^2\Big)\Omega_2d{\rm Vol_{\partial M}}\\
&=\int_{\partial M}\Big(\sum_{\beta=1}^{3}\sum_{i=1}^{2}\frac{2i\pi^3}{3}{\langle J(e_{\beta}), e_{i}\rangle}^2+\sum_{\beta=1}^{3}\frac{i\pi^2}{2}{\langle J(e_{\beta}), e_{3}\rangle}^2\Big)d{\rm Vol_{\partial M}}.\nonumber
\end{align}
\end{thm}

\vskip 1 true cm

\section{Declarations}
Ethics approval and consent to participate No applicable.\\

Consent for publication No applicable.\\

Availability of data and material The authors confirm that the data supporting the findings of this study are available within the article.\\

Competing interests The authors declare no conflict of interest.\\

Funding This research was funded by National Natural Science Foundation of China: No.11771070.\\

Authors' contributions All authors contributed to the study conception and design. Material preparation, data collection and analysis were performed by Siyao Liu and Yong Wang. The first draft of the manuscript was written by Siyao Liu and all authors commented on previous versions of the manuscript. All authors read and approved the final manuscript.\\

\vskip 1 true cm

%-----------------------------------------------------------------------------
%-----------------------------------------------------------------------------

\bigskip
\bigskip

\noindent {\footnotesize {\it S. Liu} \\
{School of Mathematics and Statistics, Northeast Normal University, Changchun 130024, China}\\
{Email: liusy719@nenu.edu.cn}

\noindent {\footnotesize {\it Y. Wang} \\
{School of Mathematics and Statistics, Northeast Normal University, Changchun 130024, China}\\
{Email: wangy581@nenu.edu.cn}

\end{document}